\documentclass[12pt]{article}
\title{Dp-finite fields IV: the rank 2 picture}
\author{Will Johnson}

\usepackage{amsmath, amssymb, amsthm}    	
\usepackage{fullpage} 	
\usepackage{amscd}
\usepackage{hyperref}
\usepackage[all]{xy}
\usepackage{centernot}

\DeclareMathOperator*{\forkindep}{\raise0.2ex\hbox{\ooalign{\hidewidth$\vert$\hidewidth\cr\raise-0.9ex\hbox{$\smile$}}}}

\newcommand{\dlog}{\partial \log}
\newcommand{\ACVF}{\mathrm{ACVF}}
\newcommand{\LDVF}{\mathrm{LDVF}}
\newcommand{\ADVF}{\mathrm{ADVF}}
\newcommand{\ACF}{\mathrm{ACF}}

\newcommand{\End}{\operatorname{End}}

\newcommand{\Ann}{\operatorname{Ann}}
\newcommand{\Frac}{\operatorname{Frac}}

\newcommand{\Der}{\operatorname{Der}}

\newcommand{\characteristic}{\operatorname{char}}

\newcommand{\Hom}{\operatorname{Hom}}

\newcommand{\Tr}{\operatorname{Tr}}
\newcommand{\res}{\operatorname{res}}
\newcommand{\wres}{\widehat{\res}}

\newcommand{\Mod}{\operatorname{Mod}}
\newcommand{\Vect}{\operatorname{Vect}}

\newcommand{\coker}{\operatorname{coker}}
\newcommand{\acl}{\operatorname{acl}}
\newcommand{\dcl}{\operatorname{dcl}}
\newcommand{\tp}{\operatorname{tp}}

\newcommand{\dom}{\operatorname{dom}}
\newcommand{\val}{\operatorname{val}}

\newcommand{\dpr}{\operatorname{dp-rk}}

\newcommand{\Sub}{\operatorname{Sub}}
\newcommand{\Dir}{\operatorname{Dir}}

\newtheorem{theorem}{Theorem}[section] 
\newtheorem{lemma}[theorem]{Lemma}

\newtheorem{corollary}[theorem]{Corollary}
\newtheorem{fact}[theorem]{Fact}

\newtheorem{conjecture}[theorem]{Conjecture}

\newtheorem{proposition}[theorem]{Proposition}
\newtheorem{proposition-eh}[theorem]{Proposition(?)}
\newtheorem*{theorem-star}{Theorem}
\newtheorem*{conjecture-star}{Conjecture}
\newtheorem*{lemma-star}{Lemma}

\theoremstyle{definition}
\newtheorem{definition}[theorem]{Definition}

\theoremstyle{remark}
\newtheorem{remark}[theorem]{Remark}
\newtheorem{claim}[theorem]{Claim}

\newtheorem*{acknowledgment}{Acknowledgments}

\newcommand{\Qq}{\mathbb{Q}}

\newcommand{\Zz}{\mathbb{Z}}
\newcommand{\Kk}{\mathbb{K}}
\newcommand{\Nn}{\mathbb{N}}

\newcommand{\Oo}{\mathcal{O}}
\newcommand{\mm}{\mathfrak{m}}
\newcommand{\pp}{\mathfrak{p}}

\newenvironment{claimproof}[1][\proofname]
               {
                 \proof[#1]
                 
               }
               {
                 \endproof
               }

\begin{document}
\maketitle

\begin{abstract}
  We investigate fields of characteristic 0 and dp-rank 2.  While we
  do not obtain a classification, we prove that any unstable field of
  characteristic 0 and dp-rank 2 admits a unique definable V-topology.
  If this statement could be generalized to higher ranks, we would
  obtain the expected classification of fields of finite dp-rank.

  We obtain the unique definable V-topology by investigating the
  ``canonical topology'' defined in \cite{prdf}.  Contrary to the
  expectations of \cite{prdf2}, the canonical topology need not be a
  V-topology.  However, we are able to characterize the canonical
  topology (on fields of dp-rank 2 and characteristic 0) in terms of
  differential valued fields.

  This differential valued structure is obtained through a partial
  classification of \emph{2-inflators}, a sort of ``generalized
  valuation'' that arises naturally in fields of finite rank.

  Additionally, we give an example of a dp-rank 2 expansion of ACVF
  with a definable set of full rank and empty interior.  This example
  interferes with certain strategies for proving the henselianity
  conjecture.
\end{abstract}

\section{Introduction}
NIP structures play a central role in modern model theory, and so it
would be desirable to classify the NIP theories of fields.  NIP can be
characterized via dp-rank: a structure $M$ is NIP iff $\dpr(M) <
\infty$.  For an overview of NIP and dp-rank, see (\cite{NIPguide},
Chapters 2 and 4).

From the point of view of dp-rank, the natural first step is fields of
dp-rank 1 (dp-minimal fields).  Dp-minimal fields were successfully
classified in (\cite{myself}, Chapter 9).  The hope is to
generalize this proof to the next simplest case---fields of finite dp-rank (dp-finite fields).

\subsection{The story so far} \label{sec:sofar}

The present paper continues \cite{prdf}, \cite{prdf2}, \cite{prdf3},
which made some partial progress on the classification of dp-finite fields.  The overall
strategy is to prove the dp-finite case of the Shelah conjecture:
\begin{conjecture}[Shelah conjecture, dp-finite case]
  Let $K$ be a dp-finite field.  Then one of the following holds:
  \begin{itemize}
  \item $K$ is finite
  \item $K$ is algebraically closed
  \item $K$ is real closed
  \item $K$ admits a non-trivial henselian valuation.
  \end{itemize}
\end{conjecture}
Modulo this conjecture, the classification of dp-finite fields is
known (\cite{halevi-hasson-jahnke}, Theorem~3.11\footnote{The
  classification in \cite{halevi-hasson-jahnke} is for strongly
  dependent fields, but the proof specializes to the case of dp-finite
  fields.  The fields appearing in the conjectured classification of
  strongly dependent fields are all dp-finite
  (\cite{halevi-hasson-jahnke}, Proposition~3.9).}).

Let $(K,+,\cdot,\ldots)$ be a dp-finite field, possibly with extra
structure.  If $K$ is stable, then $K$ must be algebraically closed or
finite (\cite{Palacin}, Proposition~7.2).  Assume $K$ is unstable.  In
\cite{prdf} and \cite{prdf2}, we constructed a field topology on $K$
characterized by the fact that the following family is a neighborhood
basis of 0:
\begin{equation*}
  \{ X - X : X \subseteq K,~ X \textrm{ is definable},~ \dpr(X) =
  \dpr(K)\}.
\end{equation*}
Here $X - Y$ denotes the set of differences
\[ \{x-y : x \in X, ~ y \in Y\},\]
rather than the set difference $X \setminus Y$.

We call this topology the \emph{canonical topology}.  In a monster
model $\Kk \succeq K$, define the \emph{$K$-infinitesimals} to be the
intersection of all $K$-definable basic neighborhoods:
\begin{equation*}
  J_K = \bigcap \{ X - X : X \subseteq \Kk,~ X \textrm{ is
    $K$-definable},~ \dpr(X) = \dpr(K)\}.
\end{equation*}
Using the infinitesimals, we proved the Shelah conjecture for
dp-finite fields of positive characteristic in (\cite{prdf},
Corollary~11.4).

In \cite{prdf2}, we sketched a strategy for attacking fields of
characteristic 0.  Say that $K$ is \emph{valuation type} if the
canonical topology is a V-topology.  (See \cite{prestel-ziegler} for a
reference on topological fields and V-topologies.)  We conjectured
\begin{conjecture}[Valuation conjecture] \label{con:vc}
  If $K$ is an unstable dp-finite field, then $K$ is valuation type.
\end{conjecture}
Modulo this conjecture, we proved the Shelah conjecture.  We also gave
a seemingly weaker criterion which implies the valuation conjecture:

\begin{fact}[Theorem~8.11 in \cite{prdf2}]
  If the $K$-infinitesimals $J_K$ contain a non-zero ideal of a
  multi-valuation ring on $\Kk$, then $K$ is valuation type.
\end{fact}
Here, a \emph{multi-valuation ring on $\Kk$} means a finite
intersection of valuation rings on $\Kk$.

\subsection{Main results for dp-finite fields}
In the present paper, we investigate unstable fields of dp-rank 2 and
characteristic 0.  We find a counterexample to the valuation
conjecture:
\begin{theorem}\label{oh-no:thm}
  There is a valued field $(K,\Oo)$ and a subset $R \subseteq K$ such that 
  \begin{itemize}
  \item The structure $(K,+,\cdot,\Oo,R)$ has dp-rank 2.
  \item The set $R$ has full rank $\dpr(R) = 2$, but has empty
    interior with respect to the valuation topology.
  \item The canonical topology is \emph{not} a V-topology.
  \end{itemize}
\end{theorem}
The counterexample does not contradict the Shelah conjecture, or the
expected classification of dp-finite fields and valued fields.  In
fact, $(K,\Oo) \models \ACVF_{0,0}$.

In spite of the counterexample, we are able to prove the following, statement, which would imply the Shelah conjecture if generalized to higher ranks:
\begin{theorem}\label{thm:intro-1}
  Let $K$ be an unstable field of characteristic 0 and dp-rank 2.
  \begin{itemize}
  \item The canonical topology on $K$ is definable, i.e., there is a
    uniformly definable basis of opens.
  \item There is a unique definable non-trivial V-topology on $K$.
  \end{itemize}
\end{theorem}
Additionally, we can give a rather explicit description of the canonical
topology, in the cases where it is not a V-topology.
\begin{theorem}
  Let $\Kk$ be a sufficiently resplendent unstable field of
  characteristic 0 and dp-rank 2.  Suppose the canonical topology on
  $\Kk$ is \emph{not} a V-topology.  Then there exists a valuation
  $\val : K \to \Gamma$ and a derivation $\delta : \Kk \to \Kk$ such
  that the following sets form a basis for the canonical topology:
  \begin{equation*}
    B_{a,b,\gamma} := \{ x \in \Kk : \val(x - a) > \gamma \text{ and }
    \val(\delta x - b) > \gamma\}.
  \end{equation*}
  Moreover, every $B_{a,b,\gamma}$ is non-empty.
\end{theorem}
Non-emptiness of the $B_{a,b,\gamma}$ expresses some degree of independence between the derivation and the valuation.  We call this sort of field topology a \emph{DV-topology}.  We investigate
DV-topologies in \S\ref{sec:dv-fields}.  

Although we are primarily interested in the case $\characteristic(K) =
0$, we only use $\characteristic(K) \ne 2$.  In fact, we prove the
valuation conjecture in odd characteristic:
\begin{theorem}
  Let $K$ be an unstable field of dp-rank 2, with $\characteristic(K)
  > 2$.  Then the canonical topology on $K$ is a V-topology.
\end{theorem}

\subsection{Inflators}
Theorem~\ref{thm:intro-1} is obtained through an algebraic analysis of
2-inflators.  An \emph{inflator} is some sort of ``generalized
valuation'' that occurs naturally when attempting to prove the
valuation conjecture.  The basic properties of inflators were
investigated in \cite{prdf3}.

The main point of \cite{prdf3} was that the infinitesimals $J_K$ are
``governed'' by an $r$-inflator, for some $r \le \dpr(\Kk)$.
\begin{fact} \label{hoho-fact}
  Let $\Kk$ be a sufficiently saturated unstable dp-finite field.
  Then there are small models $k_0 \preceq K \preceq \Kk$ and a
  malleable $k_0$-linear $r$-inflator $\varsigma$ on $\Kk$ such that
  \begin{itemize}
  \item $r \le \dpr(\Kk)$.
  \item The group $J_K$ of $K$-infinitesimals is an ideal in the
    fundamental ring $R_\varsigma$ of $\varsigma$.
  \item If $\varsigma'$ is any mutation of $\varsigma$, then there is
    a small model $K' \succeq K$ such that $J_{K'}$ is an ideal in
    $R_{\varsigma'}$.
  \end{itemize}
\end{fact}
For definitions of inflators, the fundamental ring, malleability, and
mutation, see Definitions~4.1, 5.8, 5.31, and 10.2 (respectively) in
\cite{prdf3}.  We verify Fact~\ref{hoho-fact} in \S\ref{sec:review} below.

An inflator $\varsigma$ is \emph{weakly multi-valuation type}
(\cite{prdf3}, Definition~5.27) if its fundamental ring $R_\varsigma$
contains a non-zero ideal of a multi-valuation ring.  Because of
Fact~\ref{hoho-fact}, we can focus our attention on 2-inflators
$\varsigma$ with the following properties:
\begin{enumerate}
\item $\varsigma$ is malleable.
\item No mutation of $\varsigma$ is weakly multi-valuation type.
  Otherwise, some $K' \equiv K$ would have valuation type, implying
  the same for $K$.
\item The underlying field $K$ has characteristic 0.  The
  classification of dp-finite fields is already known in positive
  characteristic.
\end{enumerate}
Sections \ref{sec:isotyp}--\ref{sec:der} carry out an algebraic
analysis of inflators satisfying these assumptions.  The original hope
was to rule out these ``wicked'' 2-inflators.  Instead, we get a
rather explicit algebraic description.
\begin{theorem}
  Let $(K,\Oo,\mm)$ be a valued field of characteristic 0, and $k_0$
  be a small subfield on which the valuation is trivial.  Let
  $\partial : \Oo \to K/\mm$ be a $k_0$-linear derivation.  Suppose
  that for every $x \in K/\mm$, the set
  \[ \{y \in \Oo : \partial y = x\}\]
  is dense in $\Oo$.  Let
  \begin{align*}
    R &= \{x \in \Oo : \partial x \in \Oo/\mm \} \\
    I &= \{x \in \mm : \partial x \in \mm/\mm = 0\}.
  \end{align*}
  Then $R$ is a subring of $K$ and $I$ is an ideal in $R$, and the
  quotient $R/I$ is isomorphic to $k[\varepsilon] =
  k[\varepsilon]/(\varepsilon^2)$, where $k$ is the residue field of $\Oo$.  There is a 2-inflator
  \begin{align*}
    \varsigma_n : \Sub_K(K^n) &\to \Sub_k((R/I)^n) \\
    V & \mapsto (V \cap R^n + I^n)/I^n.
  \end{align*}
\end{theorem}
We call these \emph{diffeovaluation inflators} and study them in
\S\ref{sec:dv-fields}.  Actually, the definition is a little more
general, allowing $K/\mm$ to be replaced with a ``mock $K/\mm$.''

Up to mutation, diffeovaluation inflators account for all the
``wicked'' 2-inflators:
\begin{theorem}
  Let $K$ be a field of characteristic 0.  Let $\varsigma$ be a
  malleable 2-inflator on $K$.  Suppose that no mutation of
  $\varsigma$ is weakly multi-valuation type.  Then some mutation of
  $\varsigma$ is a diffeovaluation inflator.
\end{theorem}
Using this explicit characterization, we prove
Theorem~\ref{thm:intro-1}.  The characterization strongly hints at how
to build the example in Theorem~\ref{oh-no:thm}; we check the details
in \S\ref{sec:grail}.

\subsection{Notation and conventions}
A ``field'' may contain additional structure beyond the pure field
structure.  We will use bold $\Kk$ for sufficiently saturated and
resplendent fields.

Unlike \cite{prdf}, \cite{prdf2}, and \cite{prdf3}, we will write the
$K$-infinitesimals as $J_K$, not $I_K$, to avoid conflict with the
fundamental ideal $I$ of a 2-inflator.

We will use the following definitions and facts from
\cite{prdf3}:
\begin{itemize}
\item Directories (Definition~2.1) and the characterization of
  semisimple directories (Theorem~2.7).
\item Inflators and equivalence of inflators (Definitions~4.1, 4.2).
\item The basic inflator calculus of \S5.1.
\item The fundamental ring and ideal, and the generalized residue map
  (Proposition~5.7, Definition~5.8).
\item Tame and wild elements (Lemma~5.22 and Definition~5.23).
\item The notions of (weakly) multi-valuation type (Definitions~5.26,
  5.27), and the characterization in terms of tame and wild elements
  (Proposition~5.25).
\item Mutation (Theorem~10.1, Definition~10.2), transitivity of
  mutation (Proposition~10.5), and commutativity of mutation
  (Remark~10.7).
\item Malleability (Definition~5.31), and its preservation under
  mutation (Proposition~10.13)
\end{itemize}
Unlike \cite{prdf3}, we will use $k_0$ for the small ground field,
rather than $K_0$.  Our inflators will be $k_0$-linear.  All rings
will be $k_0$-algebras, and all fields will extend $k_0$.

Over the course of \S \ref{sec:isotyp}--\ref{sec:der}, we will analyze
a 2-inflator $\varsigma : \Dir_K(K) \to \Dir_S(M)$, satisfying the
following assumptions:
\begin{enumerate}
\item \label{j1} The characteristic of $K$ (or its subfield $k_0$) is
  not 2.
\item $\varsigma$ is malleable
\item \label{j3} No mutation of $\varsigma$ is weakly multi-valuation
  type.
\end{enumerate}
For \S \ref{sec:canform}--\ref{sec:der}, we will add the additional
assumption
\begin{enumerate}
  \setcounter{enumi}{3}
\item \label{j4} $\varsigma$ is isotypic: if we write $M$ as $N \oplus
  N'$ with $N, N'$ simple, then $N \cong N'$.
\end{enumerate}
We call (\ref{j1}-\ref{j3}) the ``Weak Assumptions'' and (\ref{j1}-\ref{j4}) the
``Strong Assumptions.''

\begin{remark}\label{ohbtw}
  Until Lemma~\ref{weird-elements}, we will use only the following
  weaker form of (\ref{j3}): no mutation of $\varsigma$ is
  multi-valuation type.
\end{remark}

During our analysis, we will define a number of sets, rings, and
functions.  We include the following list as a reference:
\begin{itemize}
\item The 2-inflator will be
  \begin{equation*}
    \varsigma: \Dir_K(K) \to \Dir_S(M),
  \end{equation*}
  where $M$ is a semisimple $S$-module of length 2.  Beginning in
  \S\ref{sec:canform}, $S$ will be a skew field $k$, and following
  \S\ref{sec:kcomm}, $k$ will be commutative.
\item $k[\varepsilon]$ will denote the ring of dual numbers
  $k[\varepsilon]/(\varepsilon^2)$.
\item $R$ and $I$ will denote the fundamental ring and ideal of
  $\varsigma$.  The generalized residue map will be $\wres : R
  \to \End_S(M)$, or later $\wres : R \to \End_k(M)$.  In
  \S\ref{sec:canform}, we will arrange for $M = k[\varepsilon]$, and show
  that $\wres$ factors through
  \begin{equation*}
    k[\varepsilon] \cong \End_{ k[\varepsilon] } (k[\varepsilon])
  \subseteq \End_k(k[\varepsilon]) = \End_k(M).
  \end{equation*}
  Beginning in \S\ref{sec:the-val}, we will therefore view $\wres$ as
  a map
  \begin{equation*}
    \wres : R \to k[\varepsilon].
  \end{equation*}
\item In \S\ref{sec:canform}, $\mathcal{A}_0$ will denote the image of
  $\wres$ in $\End_k(M)$, and $\mathcal{A}$ will denote the
  $k$-algebra generated by $\mathcal{A}_0$.  In Lemma~\ref{lem:a0a}
  and Proposition~\ref{prop:no-split}, we will see
  \[ \mathcal{A}_0 = \mathcal{A} \cong k[\varepsilon].\]
\item $\pp$ and $Q$ will denote the sets
  \begin{align*}
    \pp &= \{x \in R : \wres(x) \in k\varepsilon\} \\
    Q &= \{x \in R : \wres(x) \in k\},
  \end{align*}
  where $k[\varepsilon] = k \oplus k\varepsilon$.  Then $\pp$ will be an
  ideal in $R$, and $Q$ will be a subring of $R$.
\item $\Oo$ will denote the integral closure of $R$.  In
  Corollary~\ref{cor:o}, we will see that $\Oo$ is a valuation ring.
  In Proposition~\ref{same-residue-2:prop}, we will see that the residue
  field of $\Oo$ is $k$.  We will let
  \begin{align*}
    \val & : K^\times \to \Gamma \\
    \res & : \Oo \to k
  \end{align*}
  denote the valuation and residue map.  In
  Proposition~\ref{same-residue-2:prop}, we will see that
  \begin{equation*}
    \wres(x) = s + t \varepsilon \implies \res(x) = s,
  \end{equation*}
  for $x \in R$ and $s, t \in k$.
\item In \S\ref{sec:der}, we will construct an $\Oo$-module $D$ and a
  $Q$-linear derivation $\partial : \Oo \to D$, as well as a valuation
  \[ \val : D \to \Gamma_{\le 0} \cup \{+\infty\}.\]
  ($D$ is essentially $K/\mm$).  Before constructing $D$ and
  $\partial$, we will define a map
  \[ \val_\partial : \Oo \to \Gamma_{\le 0} \cup \{+\infty\}\]
  in Definition~\ref{vp-def}.  Later, $\val_\partial(x)$ will turn out
  to be $\val(\partial x)$.
\end{itemize}

\section{Reduction to the isotypic case} \label{sec:isotyp}
For the duration of \S\ref{sec:isotyp}-\ref{sec:der}, \[\varsigma :
\Dir_K(K) \to \Dir_S(M)\] will be a $k_0$-linear 2-inflator, and the
following Weak Assumptions will be in place:
\begin{itemize}
\item The characteristic of $K$ (or its subfield $k_0$) is not 2.
\item $\varsigma$ is malleable (\cite{prdf3}, Definition~5.31).
\item No mutation of $\varsigma$ is weakly of multi-valuation type
  (\cite{prdf3}, Definition~5.27)
\end{itemize}
\begin{remark}\label{rem:inherit}
  If $\varsigma$ satisfies the Weak Assumptions, then so does any
  mutation $\varsigma'$, by (\cite{prdf3}, Propositions~10.5 and
  10.13).
\end{remark}
Since $M$ is semisimple of length 2, we can write it as an internal direct sum
\[ M = A \oplus B,\]
with $A, B$ simple.  We say that $\varsigma$ is \emph{isotypic} if $A
\cong B$.  This depends only on the isomorphism class of the directory
$\Dir_S(M)$.

By Theorem~2.7 in \cite{prdf3}, we can assume that we are in one of
two cases:
\begin{itemize}
\item $S = k$ and $M = k^2$, for some division algebra $k$ over $k_0$.
\item $S = k_1 \times k_2$ and $M = k_1 \oplus k_2$, for two
  division algebras $k_1, k_2$ over $k_0$.
\end{itemize}
The first case is isotypic, and the second case is non-isotypic.

\subsection{The degeneracy subspace}
For any $\alpha \in K$, let $\Theta_\alpha$ denote the line
\[ \Theta_\alpha := K \cdot (1,\alpha) = \{(x,\alpha x) : x \in K\}.\]
For any $\varphi \in \End_S(M)$, let $\Theta_\varphi$ denote the graph of
$\varphi$, i.e.,
\[ \Theta_\varphi = \{(x, \varphi(X)) : x \in M\}.\]
Recall from (\cite{prdf3}, Definition~5.8) that the \emph{fundamental
  ring} of $\varsigma$ is the set
\begin{equation*}
  R = \{\alpha \in K ~|~ \exists \varphi \in \End_S(M) :
  \varsigma_2(\Theta_\alpha) = \Theta_\varphi\}.
\end{equation*}
This is a subring of $K$.

For every $\alpha \in K$, one of two things occurs, by Lemma~5.22 in \cite{prdf3}.
\begin{itemize}
\item The fundamental ring $R$ contains all but at most two of the numbers
  \begin{equation}
    \{ \alpha \} \cup \left \{ \frac{1}{\alpha - c} : c \in k_0
    \right\}. \label{those-numbers}
  \end{equation}
\item The fundamental ring $R$ contains none of the numbers in
  (\ref{those-numbers}).
\end{itemize}
In the first case, $\alpha$ is called \emph{tame}, and in the second
case $\alpha$ is called \emph{wild} (\cite{prdf3}, Definition~5.23).
The fact that $\varsigma$ does not have multi-valuation type implies
that there is at least one wild $\alpha \in K$ (\cite{prdf3},
Proposition~5.25).

\begin{lemma}\label{lem:degen}
  There is an $S$-submodule $A \subseteq M$ of length 1 such that for every
  wild $\alpha$,
  \begin{equation*}
    \varsigma_2(\Theta_\alpha) = A \oplus A \subseteq M \oplus M.
  \end{equation*}
\end{lemma}
\begin{proof}
  Suppose $\alpha$ is wild.  Then $\alpha \notin R$, so
  $\varsigma_2(\Theta_\alpha)$ is not the graph of an endomorphism.
  Counting lengths, this implies $0 \oplus A \subseteq
  \varsigma_2(\Theta_\alpha)$ for some length-1 submodule $A \subseteq
  M$.  Similarly, $\alpha^{-1} \notin R$ implies that
  \[ 0 \oplus A' \subseteq \varsigma_2(\Theta_{\alpha^{-1}})\]
  for some length-1 submodule $A' \subseteq M$.  Or equivalently, by
  permutation invariance,
  \[ A' \oplus 0 \subseteq \varsigma_2(\Theta_\alpha).\]
  As $\varsigma_2(\Theta_\alpha)$ has length 2, we must have
  \begin{equation*}
    \varsigma_2(\Theta_\alpha) = A' \oplus A.
  \end{equation*}
  We claim that $A = A'$.  Suppose otherwise.  Then $M$ is an internal
  direct sum of $A$ and $A'$.  By $GL_2(k_0)$-equivariance,
  \begin{align*}
    \varsigma_2(\Theta_{1/(\alpha+1)}) &= \varsigma_2(\{(\alpha x + x,
    x) : x \in K\}) \\ &= \{(y+x, x) : x \in A', ~ y \in A\}.
  \end{align*}
  The latter expression is the graph of an endomorphism, so
  $1/(\alpha+1) \in R$ and $\alpha$ is tame, a contradiction.

  Thus, for any wild $\alpha$ there is some length-1 submodule
  $A_\alpha \subseteq M$ such that $\varsigma_2(\Theta_\alpha) =
  A_\alpha \oplus A_\alpha$.  It remains to show that $A_\alpha$
  doesn't depend on $\alpha$.  Suppose for the sake of contradiction
  that $\alpha, \beta$ are wild and $A_\alpha \ne A_\beta$.  Then
  \begin{align*}
    \varsigma_3(\{(x,\alpha x, y) : x, y \in K\}) &= A_\alpha \oplus
    A_\alpha \oplus M \\ \varsigma_3(\{(w, y, \beta y) : w, y \in K\})
    &= M \oplus A_\beta \oplus A_\beta \\ \varsigma_3(\{(x, \alpha x,
    \alpha \beta x) : x \in K\}) &= A _\alpha \oplus 0 \oplus A_\beta
  \end{align*}
  using Lemma~5.2.1 in \cite{prdf3} and the fact that $A_\alpha \cap A_\beta = 0$.  Then
  \begin{align*}
    \varsigma_3(\{(x, \alpha x, \alpha \beta x) : x \in K\}) &= A _\alpha \oplus 0 \oplus A_\beta \\
    \varsigma_3(\{(0, y, 0) : y \in K\}) &= 0 \oplus M \oplus 0 \\
    \varsigma_3(\{(x, y, \alpha \beta x) : x, y \in K\}) &= A_\alpha \oplus M \oplus A_\beta,
  \end{align*}
  using Lemma~5.2.2 in \cite{prdf3}.
  But a symmetric argument shows
  \begin{equation*}
    \varsigma_3(\{(x, y, \beta \alpha x) : x, y \in K\}) = A_\beta \oplus M \oplus A_\alpha
  \end{equation*}
  As $\alpha \beta = \beta \alpha$, it follows that
  \begin{equation*}
    A_\alpha \oplus M \oplus A_\beta = A_\beta \oplus M \oplus A_\alpha,
  \end{equation*}
  and $A_\alpha = A_\beta$, a contradiction.
\end{proof}
We call $A = A_\alpha$ the \emph{degeneracy subspace} of $M$.

\subsection{Reduction to the isotypic case}

\begin{corollary}\label{cor:rediso}
  Under the Weak Assumptions, $\varsigma$ has a mutation which is
  isotypic.
\end{corollary}
\begin{proof}
  Take wild $\alpha$, and let $A \subseteq M$ be the degeneracy
  subspace, so that
  \begin{equation*}
    \varsigma_2(K \cdot (1,\alpha)) = A \oplus A.
  \end{equation*}
  By definition of mutation, the mutation along $K \cdot (1, \alpha)$
  is of the form
  \[ \varsigma' : \Dir_K(K) \to \Dir_S(M'),\]
  where $M' = \varsigma_2(K \cdot (1,\alpha)) = A \oplus A$.  The
  $S$-module $M'$ is isotypic.
\end{proof}

\section{The explicit formula for $\varsigma$} \label{sec:canform}
For the duration of \S\ref{sec:canform}-\ref{sec:der}, \[\varsigma :
\Dir_K(K) \to \Dir_S(M)\] will be a $k_0$-linear 2-inflator, and the
following Strong Assumptions will be in place:
\begin{itemize}
\item $\characteristic(k_0) = \characteristic(K) \ne 2$
\item $\varsigma$ is malleable
\item No mutation of $\varsigma$ is weakly of multi-valuation type.
\item $\varsigma$ is isotypic, i.e., if we write $M$ as a direct sum
  of two simple $S$-modules $A$ and $B$, then $A \cong B$.
\end{itemize}
Isotypy is the new assumption, not present in the Weak Assumptions of
\S\ref{sec:isotyp}.  Isotypy implies that
\begin{equation*}
  \Dir_S(M) \cong \Dir_k(k^2)
\end{equation*}
for some division $k_0$-algebra $k$.  Therefore, we may assume $S = k$
and $M$ is a two-dimensional $k$-vector space.
\begin{remark} \label{rem:inherit2}
  As in Remark~\ref{rem:inherit}, the Strong Assumptions are preserved
  under mutations.  For isotypy, note that any mutation of $\varsigma$
  will have the form
  \begin{equation*}
    \varsigma' : \Dir_K(K) \to \Dir_k(M')
  \end{equation*}
  for some $k$-module $M'$.  The fact that $k$ is a division ring
  ensures that $M'$ is isotypic.
\end{remark}

\subsection{$k$ is commutative} \label{sec:kcomm}

\begin{proposition}
  The division ring $k$ is commutative (a field).
\end{proposition}
\begin{proof}
  Changing coordinates on $k^2$, we may assume that the degeneracy
  subspace is $0 \oplus k$.  If $\alpha \in K$ is wild, then
  \begin{equation} \label{wild-corral}
    \varsigma_2(\{(x, \alpha x) : x \in K\}) = (0 \oplus k) \oplus (0 \oplus k).
  \end{equation}
  We write elements of $M^n = (k^2)^n$ as tuples
  $(a_1,b_1;a_2,b_2;\ldots;a_n,b_n)$.

  Let $a, b$ be two non-commuting elements of $k$.  By malleability,
  we can find $\alpha, \beta \in K^\times$ such that
  \begin{align*}
    \varsigma_2(\{(x,\alpha x) : x \in K\}) & \supseteq \{(t,0;t a,0) : t \in k\} \\
    \varsigma_2(\{(x,\beta x) : x \in K\}) & \supseteq \{(t,0; tb, 0) : t \in k\}
  \end{align*}
  Neither $\alpha$ nor $\beta$ can be wild, by equation
  (\ref{wild-corral}).  By $GL_2(k_0)$-equivariance,
  \begin{align*}
    \varsigma_2(\{(x,\alpha x) : x \in K\}) & \supseteq \{(t,0;t a,0) : t \in k\} \\
    \varsigma_2(\{(x,(\alpha - 1) x) : x \in K\}) & \supseteq \{(t,0;t (a - 1),0) : t \in k\} \\
    \varsigma_2(\{(x,\alpha^{-1} x) : x \in K\}) & \supseteq \{(t,0;t a^{-1} ,0) : t \in k\} \\
    \varsigma_2(\{(x,(\alpha-1)^{-1} x) : x \in K\}) & \supseteq \{(t,0;t (a - 1)^{-1},0) : t \in k\}
  \end{align*}
  And one of $\alpha, 1/\alpha, 1/(\alpha - 1)$ is in $R$.  So,
  replacing $a$ with one of $\{a, a^{-1}, (a-1)^{-1}\}$, we may assume
  that $\alpha \in R$.  Similarly, we may assume $\beta \in R$.
  Then
  \begin{align*}
    \varsigma_2(\{(x,\alpha x) : x \in K\})& \\
    \varsigma_2(\{(x,\beta x) : x \in K\}) &
  \end{align*}
  are graphs of endomorphisms $\varphi_A, \varphi_B : k^2 \to k^2$.  Because
  \begin{align*}
    (t,0;ta,0) &\in \Gamma_A \\
    (t,0;tb,0) &\in \Gamma_B     
  \end{align*}
  it follows that $\varphi_A(t,0) = (ta,0)$ and $\varphi_B(t,0) = (tb,0)$.
  Thus $\varphi_A$ and $\varphi_B$ do not commute.  But this is impossible,
  as $\varphi_A, \varphi_B$ both lie in the image of the homomorphism $R
  \to \End_k(k^2)$, and $R$ is commutative.
\end{proof}
Thus $k$ is a field extending $k_0$.

\subsection{The algebra $\mathcal{A}$} \label{ssec:a}
Let $\mathcal{A}_0$ be the image of $R$ in $\End_k(k^2) = M_2(k)$.
This is a commutative $k_0$-algebra.
\begin{lemma}\label{no-central}
  The algebra $\mathcal{A}_0$ is not contained in the center of
  $M_2(k)$, i.e., $\mathcal{A}_0$ contains a matrix \emph{not} of the
  form $\begin{pmatrix} \lambda & 0 \\ 0 & \lambda \end{pmatrix}$.
\end{lemma}
\begin{proof}
  Changing coordinates on $M \cong k^2$, we may assume that $0 \oplus
  k$ is the degeneracy subspace, and so
  \begin{equation} \label{wild-corral-2}
    \varsigma_2(\{(x, \alpha x) : x \in K\}) = (0 \oplus k) \oplus (0 \oplus k).
  \end{equation}
  for any wild $\alpha$.  By malleability, choose $\alpha \in K$ such
  that
  \begin{equation*}
     \varsigma_2(\{(x,\alpha x) : x \in K\}) \supseteq \{(t,0;0,t) : t \in k\} \ni (1,0;0,1).
  \end{equation*}
  Then $\alpha$ cannot be wild.  By $GL_2(k_0)$-equivariance, we have
  \begin{align*}
    (1,0;0,1) & \in \varsigma_2(\{(x,\alpha x) : x \in K\}) \\
    (0,1;1,0) & \in \varsigma_2(\{(x, \alpha^{-1} x) : x \in K\}) \\
    (1,1;1,0) &\in \varsigma_2(\{(x, (\alpha + 1)^{-1}) : x \in K\}).
  \end{align*}
  By tameness of $\alpha$, one of the right-hand-sides is the graph of
  some endomorphism $\varphi \in \End_k(k^2)$.  By definition, $\varphi \in
  \mathcal{A}_0$.  But no central matrix $\begin{pmatrix} \lambda & 0
    \\ 0 & \lambda \end{pmatrix}$ can map $(1,0)$ to $(0,1)$ or map
  $(0,1)$ to $(1,0)$ or map $(1,1)$ to $(1,0)$.
\end{proof}
Let $\mathcal{A}$ be the $k$-subalgebra of $M_2(k)$ generated by
$\mathcal{A}_0$.  It is a commutative $k$-algebra.
Note that $M = k^2$ is naturally an $\mathcal{A}$-module.
\begin{proposition}\label{prop:descend-to-A}
  For any $V \subseteq K^n$, the specialization $\varsigma_n(V)$ is
  an $\mathcal{A}$-submodule of $M^n$.
\end{proposition}
\begin{proof}
  We already know that $\varsigma_n(V)$ is a $k$-submodule, so it
  remains to show that $\varsigma_n(V)$ is closed under multiplication
  by $\mathcal{A}_0$.  Let $a$ be an element of $R$, specializing to
  $\varphi \in \mathcal{A}_0$.  Then
  \begin{align*}
    \varsigma_{2n}(\{(\vec{x}, a \vec{x}) : \vec{x} \in K^n\}) &= \{(x_1,
    \ldots, x_n, \varphi x_1, \ldots, \varphi x_n) : \vec{x} \in M^n\} \\
    \varsigma_{2n}(\{(\vec{x}, \vec{y}) : \vec{x} \in V,~ \vec{y} \in K^n\}) &= \{(x_1, \ldots, x_n, y_1, \ldots, y_n) : \vec{x} \in \varsigma_n(V), \vec{y} \in M^n\} \\
    \varsigma_{2n}(\{(\vec{x}, \vec{y}) : \vec{x} \in K^n,~ \vec{y} \in V\}) &= \{(x_1, \ldots, x_n, y_1, \ldots, y_n) : \vec{x} \in M^n, \vec{y} \in \varsigma_n(V)\} \\
    \varsigma_{2n}(\{(\vec{x}, a \vec{x}) : \vec{x} \in V\}) &= \{(x_1,
    x_2, \ldots, x_n, \varphi x_1, \ldots, \varphi x_n) : \vec{x} \in \varsigma_n(V)\}.
  \end{align*}
  The first line holds because $a$ specializes to $\varphi$, using
  compatibility with $\oplus$ and permutations.  The second and third lines hold
  by compatibility with $\oplus$.  The fourth line holds by
  intersecting the first and second lines (using Lemma~5.2.1 in
  \cite{prdf3}).
  Now $V$ is a $K$-submodule of $K^n$, and hence an $R$-submodule.  Therefore
  \begin{equation*}
    \{(\vec{x}, a \vec{x}) : \vec{x} \in V\} \subseteq \{(\vec{x},
    \vec{y}) : \vec{x} \in K^n,~ \vec{y} \in V\}.
  \end{equation*}
  As $\varsigma_{2n}$ is order-preserving,
  \begin{align*}
    & \{(x_1, x_2, \ldots, x_n, \varphi x_1, \ldots, \varphi x_n) :
    \vec{x} \in \varsigma_n(V)\} \\ & \subseteq \{(x_1, \ldots, x_n, y_1, \ldots, y_n)
    : \vec{x} \in M^n, \vec{y} \in \varsigma_n(V)\},
  \end{align*}
  which implies that $\varsigma_n(V)$ is closed under multiplication
  by $\varphi$.
\end{proof}
Proposition~\ref{prop:descend-to-A} says that the inflator $\varsigma
: \Dir_K(K) \to \Dir_k(M)$ factors through $\Dir_{\mathcal{A}}(M)
\subseteq \Dir_k(M)$.
\begin{corollary}
  The degeneracy subspace $X \subseteq M$ is an
  $\mathcal{A}$-submodule of $M$.
\end{corollary}
\begin{proof}
  If $\alpha \in K$ is wild, then
  \[ \varsigma_2(K \cdot (1,\alpha)) = X \oplus X,\]
  and so $X \oplus X$ is an $\mathcal{A}$-submodule of $M \oplus M$.
  This implies $X$ is an $\mathcal{A}$-submodule of $M$.
\end{proof}

\begin{lemma}\label{lem:matrix-nonsense}
  The $k$-algebra $\mathcal{A}$ is isomorphic to one of the following
  \begin{itemize}
  \item $k \times k$
  \item $k[\varepsilon] = k[\varepsilon]/(\varepsilon^2)$.
  \end{itemize}
  Moreover, $M$ is a free $\mathcal{A}$-module of rank 1.
\end{lemma}
\begin{proof}
  View $\mathcal{A}$ as a commutative $k$-subalgebra of $M_2(k)$.
  Take $\mu \in \mathcal{A}$ non-central.  Changing the identification
  $M \cong k^2$, we may assume we are in one of three cases:
  \begin{itemize}
  \item $\mu = \begin{pmatrix} a & 0 \\ 0 & b \end{pmatrix}$ for some
    $a \ne b$ in $k$.
  \item $\mu = \begin{pmatrix} a & 0 \\ b & a \end{pmatrix}$ for some
    $a, b \in k$ with $b \ne 0$.
  \item $\mu = \begin{pmatrix} 0 & 1 \\ a & b \end{pmatrix}$ for some
    monic irreducible quadratic polynomial $x^2 - bx - a \in k[x]$.
  \end{itemize}
  The degeneracy subspace is a one-dimensional subspace of $k^2$,
  preserved by $\mu$, and so $\mu$ has an eigenvector.  This rules out
  the third case.

  In the first case, the $k$-subalgebra generated by $\mu$ is
  \begin{equation*}
    \mathcal{A}' = 
    \left\{
    \begin{pmatrix}
      x & 0 \\ 0 & y
    \end{pmatrix} :
    x, y \in k
    \right\} \cong k \times k.
  \end{equation*}
  Then $\mathcal{A} \supseteq \mathcal{A}'$.  In particular,
  $\mathcal{A}$ contains the matrix $\mu' = \begin{pmatrix} 1 & 0 \\ 0
    & 0 \end{pmatrix}$, and $\mathcal{A}$ lies in the centralizer of
  $\mu'$.  By inspection, this centralizer is $\mathcal{A}'$.  Thus
  $\mathcal{A} = \mathcal{A}' \cong k \times k$.  Also, the vector
  $(1,1)$ freely generates $k^2$ as an $\mathcal{A}'$-module.

  In the second case, the $k$-subalgebra generated by $\mu$ is
  \begin{equation*}
    \mathcal{A}'' =
    \left\{ 
    \begin{pmatrix}
      x & 0 \\ y & x
    \end{pmatrix} :
    x,y \in k\right\} \cong k[\varepsilon].
  \end{equation*}
  Then $\mathcal{A} \supseteq \mathcal{A}''$.  In particular,
  $\mathcal{A}$ contains the matrix $\mu'' = \begin{pmatrix} 0 & 0
    \\ 1 & 0 \end{pmatrix}$, and $\mathcal{A}$ lies in the centralizer
  of $\mu''$.  By inspection, this centralizer is $\mathcal{A}''$.
  Thus $\mathcal{A} = \mathcal{A}'' \cong k[\varepsilon]$.  Also, the
  vector $(1,0)$ freely generates $k^2$ as an $\mathcal{A}''$-module.
\end{proof}
\begin{lemma} \label{lem:a0a}
  The two algebras $\mathcal{A}_0$ and $\mathcal{A}$ are equal.
\end{lemma}
\begin{proof}
  Changing $M$ up to isomorphism, we may assume $M = \mathcal{A}$.
  Let $\mu$ be any element of $\mathcal{A}$; we will show $\mu \in
  \mathcal{A}_0$.  By malleability, there is a line $L \subseteq K^2$
  such that
  \begin{equation*}
    \varsigma_2(L) \supseteq k \cdot (1,\mu).
  \end{equation*}
  But $\varsigma_2(L)$ is an $\mathcal{A}$-submodule of $M^2$
  (Proposition~\ref{prop:descend-to-A}), and the
  $\mathcal{A}$-submodule generated by $(1,\mu)$ is
  \begin{equation*}
    \{ (x, \mu x) : x \in \mathcal{A}\}.
  \end{equation*}
  Thus
  \[ \varsigma_2(L) \supseteq \{(x,\mu x) : x \in \mathcal{A}\}.\]
  Both sides have length two as $k$-modules, and so equality holds.
  Let $a$ be the slope of $L$.  Evidently, $a \ne \infty$, or else
  $\varsigma_2(L)$ would be $0 \oplus \mathcal{A}$.  So $a \in R$ and
  $a$ specializes to $\mu$.  Therefore $\mu \in \mathcal{A}_0$.
\end{proof}
In particular, the natural homomorphism
\[ R \twoheadrightarrow \mathcal{A}_0 \hookrightarrow \mathcal{A}\]
is surjective.  If $I$ is the fundamental ideal (the kernel of $R
\to \End_k(M)$), then $R/I$ is isomorphic to $k \times k$ or to
$k[\varepsilon]$.

\begin{proposition}\label{exact-form}
  Fix any $\mathcal{A}$-module isomorphism $M \cong \mathcal{A}$.
  Then the specialization maps $\varsigma_n : \Sub_K(K^n) \to
  \Sub_k(\mathcal{A}^n)$ are given by the formula
  \begin{equation*}
    \varsigma_n(V) =
    \{(\wres(x_1),\ldots,\wres(x_n)) :
    (x_1,\ldots,x_n) \in V \cap R^n\}
  \end{equation*}
  where $\wres$ is the generalized residue map $R \twoheadrightarrow
  R/I \cong \mathcal{A}$.
\end{proposition}
\begin{proof}
  First suppose that $(a_1,\ldots,a_n) \in R^n \cap V$, and $a_i$
  specializes to $b_i \in \mathcal{A}$ for each $i$.  Then the line
  \begin{equation*}
    L = \{(x, a_1x, \ldots, a_nx) : x \in K\}
  \end{equation*}
  is contained in $K \oplus V$, and so
  \begin{equation*}
    \varsigma_{n+1}(L) = \{(x, b_1x, \ldots, b_nx) : x \in
    \mathcal{A}\} \subseteq \varsigma_{n+1}(K \oplus V) = \mathcal{A} \oplus \varsigma_n(V).
  \end{equation*}
  In particular, $(1, b_1, \ldots, b_n) \in \mathcal{A} \oplus
  \varsigma_n(V)$, and so
  \begin{equation*}
    (\wres(a_1),\ldots,\wres(a_n)) =
    (b_1,\ldots,b_n) \in \varsigma_n(V).
  \end{equation*}
  We have seen
  \begin{equation*}
    \varsigma_n(V) \supseteq \{(\wres(x_1),\ldots,\wres(x_n)) :
    (x_1,\ldots,x_n) \in V \cap R^n\}
  \end{equation*}
  Conversely, suppose that $(b_1,\ldots,b_n) \in \varsigma_n(V)$.
  Then
  \begin{equation*}
    (1, b_1, \ldots, b_n) \in \mathcal{A} \oplus \varsigma_n(V) =
    \varsigma_{n+1}(K \oplus V).
  \end{equation*}
  By malleability, there is a line $L \subseteq K \oplus V$ such that
  \begin{equation*}
    \varsigma_{n+1}(L) \supseteq k \cdot (1, b_1, \ldots, b_n) \ni (1,
    b_1, \ldots, b_n).
  \end{equation*}
  The left hand side is an $\mathcal{A}$-module, so
  \begin{equation*}
    \varsigma_{n+1}(L) \supseteq \mathcal{A} \cdot (1, b_1, \ldots, b_n).
  \end{equation*}
  Both sides have length two over $k$, so equality holds.  Now $L$
  must be the graph of a $K$-linear function $K \to K^n$; otherwise $L
  \subseteq 0 \oplus K^n$ and $\varsigma_{n+1}(L)$ would be contained
  in $0 \oplus \mathcal{A}^n$, which is visibly false.  Thus
  \begin{equation*}
    L = K \cdot (1, a_1, \ldots, a_n)
  \end{equation*}
  for some $a_i \in K$.  So
  \begin{equation*}
    \varsigma_{n+1}(\{(x,a_1x, \ldots, a_nx) : x \in K\}) = \{(x,
    b_1 x, \ldots, b_n x) : x \in \mathcal{A}\}.
  \end{equation*}
  Joining this with
  \begin{equation*}
    \varsigma_{n+1}(0 \oplus K^{i-1} \oplus 0 \oplus K^{n - i}) = 0
    \oplus \mathcal{A}^{i-1} \oplus 0 \oplus \mathcal{A}^{n-i},
  \end{equation*}
  we obtain
  \begin{align*}
    \varsigma_{n+1}(&\{(x,y_1,\ldots,y_{i-1},a_ix,y_{i+1},\ldots,y_n)
    : x,y_1,\ldots,y_n \in K\}) \\ = &
    \{(x,y_1,\ldots,y_{i-1},b_ix,y_{i+1},\ldots,y_n) :
    x,y_1,\ldots,y_n \in \mathcal{A}\}.
  \end{align*}
  By permutation invariance and $\oplus$-compatibility, we see
  \begin{equation*}
    \varsigma_2(\{(x, a_ix) : x \in K\}) = \{(x, b_i x) : x
    \in \mathcal{A}\}.
  \end{equation*}
  So each $a_i$ is in $R$, and $\wres(a_i) = b_i$.  The fact
  that $L \subseteq K \oplus V$ implies that $\vec{a} \in V$.  So we
  have shown that
  \begin{equation*}
    \varsigma_n(V) \subseteq \{(\wres(x_1),\ldots,\wres(x_n)) :
    (x_1,\ldots,x_n) \in V \cap R^n\}.  \qedhere
  \end{equation*}
\end{proof}

\begin{corollary}\label{cor:fof}
  $K = \Frac(R)$.
\end{corollary}
\begin{proof}
  Let $\alpha$ be any element of $K^\times$.  By
  Proposition~\ref{exact-form},
  \begin{equation*}
    \varsigma_2(\{(x, \alpha x) : x \in K\}) = \{(\wres(x),\wres(y)) :
    x,y \in R,~ y/x = \alpha\}.
  \end{equation*}
  The fact that $\varsigma_2(\{(x, \alpha x) : x \in K\}) \ne 0$
  implies that there exist non-trivial $x, y \in R$ such that $y/x =
  \alpha$.
\end{proof}

\subsection{Ruling out the split case} \label{ssec:nosplit}
Let $I$ denote the \emph{fundamental ideal} (\cite{prdf3},
Definition~5.8), i.e., the kernel of the generalized residue map
\[ \wres : R \to \mathcal{A}.\]
\begin{remark}
  Every maximal ideal of $R$ comes from a maximal ideal of the
  artinian ring $R/I \cong \mathcal{A}$.  Indeed, this follows from
  the fact that $I \subseteq Jac(R)$ (\cite{prdf3},
  Proposition~5.7.4).
\end{remark}
For example, if $\mathcal{A} \cong k \times k$, then $R$ has two
maximal ideals.
\begin{proposition}\label{prop:no-split}
  The algebra $\mathcal{A}$ is isomorphic to $k[\varepsilon]$.
\end{proposition}
\begin{proof}
  Otherwise, $\mathcal{A} \cong k \times k$.  Let $p_1, p_2 : R \to k$
  be the two maps such that
  \[ \wres(x) = (p_1(x),p_2(x)) \in k \times k \cong \mathcal{A}.\]
  The two maximal ideals of $R$ are the kernels of $p_1$ and $p_2$.

  The degeneracy subspace of $\varsigma$ is some rank 1 submodule of
  $\mathcal{A}$, necessarily $k \times 0$ or $0 \times k$.  Without
  loss of generality, it is $0 \times k$.

  Take some wild element $a \in K$.  By Proposition~\ref{exact-form}
  and the definition of the degeneracy subspace,
  \begin{align*}
  \varsigma_2(K \cdot (1,a)) &=
      \{(\wres(x),\wres(y)) : x, y \in R, ~ y/x = a\} \\
    &= \{ (p_1(x),p_2(x);p_1(y),p_2(y)) : x, y \in R, ~ y/x = a\} \\
    &= \{(0,t;0,s) : t, s \in k\}.
  \end{align*}
  Therefore, we can find $x, y, x', y' \in R$ such that
  \begin{align*}
    (p_1(x),p_2(x)) &= (0,0) \\
    (p_1(y),p_2(y)) &= (0,1) \\
    y &= ax \\
    (p_1(x'),p_2(x')) &= (0,1) \\
    (p_1(y'),p_2(y')) &= (0,0) \\
    y' &= ax'.
  \end{align*}
  Then $xy' = yx'$, and we obtain a contradiction:
  \begin{equation*}
    1 = p_2(y)p_2(x') = p_2(yx') = p_2(xy') = p_2(x)p_2(y') = 0. \qedhere
  \end{equation*}
\end{proof}

\begin{corollary}
  The fundamental ring $R$ is a local ring.
\end{corollary}
\begin{proof}
  Its maximal ideals come from $k[\varepsilon]$, which is a local ring.
\end{proof}
The following corollary will be useful later:
\begin{corollary}\label{cor:det-tr}
  If $\varsigma : \Dir_K(K) \to \Dir_k(k^2)$ satisfies the Strong
  Assumptions, and $a \in K$ specializes to a matrix $\mu
  = \begin{pmatrix} b & c \\ d & e \end{pmatrix}$, then $\mu$ must
  have a repeated eigenvalue, and so
  \[ (b + e)^2 = (\Tr(\mu))^2 = 4 \det(\mu) = 4(be - cd).\]
\end{corollary}
\begin{proof}
  The element $a$ lies in $R$ and its image in $\mathcal{A}$ is $\mu$.
  If $\mu$ is central, the identity is clear.  Otherwise, the proof of
  Lemma~\ref{lem:matrix-nonsense} shows that $\mu$ must be conjugate to a
  matrix of the form $\begin{pmatrix} x & y \\ 0 & x \end{pmatrix}$ or
  $\begin{pmatrix} x & 0 \\ 0 & y \end{pmatrix}$.  In the second case,
  $\mathcal{A} \cong k \times k$, contradicting
  Proposition~\ref{prop:no-split}.  So we may assume $\mu
  = \begin{pmatrix} x & y \\ 0 & x \end{pmatrix}$, and then the
  desired identity is clear.
\end{proof}
Now that we have identified $\mathcal{A}$, we can specify the
degeneracy locus:
\begin{lemma}\label{lem:which-degen}
  Under any isomorphism of $\mathcal{A}$-modules $M \cong
  \mathcal{A}$, the degeneracy subspace is the principal ideal
  $k\varepsilon = (\varepsilon) \lhd k[\varepsilon] \cong \mathcal{A}$.
\end{lemma}
\begin{proof}
  There are only three $k[\varepsilon]$-submodules of $k[\varepsilon]$, and
  $k\varepsilon$ is the only one having dimension 1 over $k$.
\end{proof}
We summarize the picture in the following Theorem.
\begin{theorem}\label{thm:so-far}
  Let $\varsigma$ be an isotypic malleable $k_0$-linear 2-inflator on
  a field $K$ with $\characteristic(K) \ne 2$.  Suppose that no
  mutation of $\varsigma$ is weakly of multi-valuation type.  Then
  there are
  \begin{itemize}
  \item A field $k$ extending $k_0$.
  \item A subring $R \subseteq K$.
  \item An ideal $I \lhd R$
  \item An isomorphism of $k_0$-algebras $R/I \cong k[\varepsilon] :=
    k[\varepsilon]/(\varepsilon^2)$
  \end{itemize}
  such that $\varsigma$ is isomorphic to
  \begin{align*}
    \varsigma : \Dir_K(K) &\to \Dir_k(k[\varepsilon]) \\
    \varsigma_n(V) &=
    \{(\wres(x_1),\ldots,\wres(x_n)) : \vec{x} \in V
    \cap R^n\}
  \end{align*}
  where $\wres$ is the quotient map \[ R \twoheadrightarrow R/I \cong
  k[\varepsilon].\]  Moreover,
  \begin{itemize}
  \item $R$ is the fundamental ring of $\varsigma$, $I$ is the
    fundamental ideal, and $\wres$ is the generalized residue
    map (in the sense of Definition~5.8 in \cite{prdf3}).
  \item $R$ is a local ring, whose unique maximal ideal $\mm$ is the
    pullback of $k \cdot \varepsilon$ along $\wres$.
  \item $\Frac(R) = K$.
  \end{itemize}
\end{theorem}

\section{The associated valuation} \label{sec:the-val}
Continue the Strong Assumptions of \S\ref{sec:canform}.  In light of
Theorem~\ref{thm:so-far}, we assume that $\varsigma$ has the form
\begin{align*}
  \varsigma : \Dir_K(K) & \to \Dir_k(k[\varepsilon]) \\
  \varsigma_n : \Sub_K(K^n) &\to \Sub_k(k[\varepsilon]^n) \\
  V & \mapsto     \{(\wres(x_1),\ldots,\wres(x_n)) : \vec{x} \in V
    \cap R^n\}
\end{align*}
where $\wres : R \twoheadrightarrow k[\varepsilon]$ is the generalized
residue map.  The fundamental ideal $I \lhd R$ is the kernel of
$\wres$.

\subsection{Finitely generated ideals} \label{ssec:fgi}
Let $Q$ be the set of $a \in R$ such that $\wres(a)$ lies in
$k$, i.e.,
\[ \wres(a) = x + 0\varepsilon\]
for some $x \in k$.  The set $Q$ is a $k_0$-subalgebra of $R$.  Note
that $I$ is an ideal in $Q$.  Moreover, if $a \in Q \setminus I$, then
\begin{equation*}
  \wres(a) = x + 0\varepsilon
\end{equation*}
for some non-zero $x$, and so $a \in R^\times$.  Then
\begin{equation*}
  \wres(a^{-1}) = (x + 0 \varepsilon)^{-1} = x^{-1} + 0\varepsilon,
\end{equation*}
so that $a^{-1} \in Q$.  Thus $Q$ is a local ring and $I$ is its
maximal ideal.
\begin{lemma}\label{three-generators}
  Let $a, b, c$ be three elements of $K$.  Then there exist $x, y, z
  \in Q$ such that $ax + by + cz = 0$, and at least one of $x, y, z$
  is 1.  The same holds for $R$ instead of $Q$.
\end{lemma}
\begin{proof}
  Consider the vector space $V = \{(x,y,z) \in K^3 : ax + by + cz =
  0\}$.  Then
  \begin{equation*}
    \varsigma_3(V) = \{(\wres(x),\wres(y),\wres(z)) : (x,y,z) \in R^3 \text{ and } ax + by + cz = 0\}.
  \end{equation*}
  Also,
  \begin{equation*}
    \dim_k(\varsigma_3(V)) = 2 \cdot \dim_K(V) = 4.
  \end{equation*}
  Counting dimensions, $\varsigma_3(V)$ must have non-trivial
  intersection with the subspace
  \begin{equation*}
    k^3 = \{(s+0\varepsilon, t + 0\varepsilon, u + 0\varepsilon) : s, t, u \in
    k^3\} \subseteq k[\varepsilon]^3.
  \end{equation*}
  Therefore, there exist $x, y, z \in R$ such that
  \begin{align*}
    \wres(x) &= s + 0\varepsilon \\
    \wres(y) &= t + 0\varepsilon \\
    \wres(z) &= u + 0\varepsilon
  \end{align*}
  for some $s, t, u \in k$, not all zero.  Then $x, y, z \in Q$ and at
  least one of the three is in $Q^\times$.  If $x \in Q^\times$, we may replace $x,
  y, z$ with $x/x, y/x, z/x$, and arrange for $x = 1$.  This handles
  the case of $Q$, and the case of $R$ follows as $R \supseteq Q$.
\end{proof}
We shall return to the ring $Q$ in \S\ref{sec:neut}.
\begin{corollary}
  Any finitely-generated $R$-submodule of $K$ is generated by at most
  two elements.  In particular, any ideal of $R$ is generated by at
  most two elements.
\end{corollary}
\begin{proof}
  It suffices to consider the case of three generators: $R \cdot a +
  R \cdot b + R \cdot c \le K$.  Then the lemma implies that one
  of $a, b, c$ is in the $R$-submodule generated by the other two.
\end{proof}

\subsection{The integral closure of $R$} \label{sec:is-a-val}
We let $\mathfrak{p}$ denote the unique maximal ideal of $R$, i.e.,
the set of $x \in R$ such that $\wres(x)$ has the form $0 + t
\varepsilon$ for some $t \in k$.  Because $R$ is local, $R^\times = R
\setminus \pp$.  Note that $\pp$ is the pullback of the principal
ideal $(\varepsilon) = k\varepsilon \lhd k[\varepsilon]$ along $\wres(-)$.
\begin{lemma}\label{discriminant}
  Suppose $\alpha \in K$ satisfies a monic quadratic equation over
  $R$:
  \begin{equation*}
    \alpha^2 + b \alpha + c = 0
  \end{equation*}
  for some $b, c \in R$.  If $\alpha$ is wild, then $b^2 - 4c \in
  \pp$.
\end{lemma}
\begin{proof}
  Let $b_0, b_1, c_0, c_1 \in k$ be such that
  \begin{align*}
    \wres(b) &= b_0 + b_1 \varepsilon \\
    \wres(c) &= c_0 + c_1 \varepsilon.
  \end{align*}
  We must show $b_0^2 = 4c_0$.  Let $L$ be the line $K \cdot (1, \alpha)$.
  By Lemma~\ref{lem:which-degen},
  \begin{equation*}
    \varsigma_2(L) = \varsigma_2(\{(x,\alpha x) : x \in K\}) = k\varepsilon \oplus k\varepsilon =
    \{(s\varepsilon,t\varepsilon) : s, t \in k\}.
  \end{equation*}
  It follows that
  \begin{align*}
    \varsigma_4(\{(x,\alpha x, -cx, - b \alpha x) : x \in K\}) & = \{(s
    \varepsilon, t \varepsilon, -c_0 s \varepsilon, -b_0 t \varepsilon) : s, t \in
    k\} \\
    \varsigma_3(\{(x, \alpha x, -b \alpha x - c x) : x \in K\}) & = \{(s
    \varepsilon, t \varepsilon, (-b_0t - c_0 s)\varepsilon) : s, t \in k\} \\
    \varsigma_4(\{(x, \alpha x, \alpha x, \alpha^2 x) : x \in K\}) & = \{(s
    \varepsilon, t \varepsilon, t \varepsilon, (-b_0 t - c_0 s) \varepsilon) : s,
    t \in k\}
  \end{align*}
  Here, we are using the identities
  \begin{align*}
    \wres(b) \cdot (s \varepsilon) &= (b_0 s) \varepsilon \\
    \wres(c) \cdot (s \varepsilon) &= (c_0 s) \varepsilon \\
    \alpha^2 &= - b \alpha - c.
  \end{align*}
  Now let $\varsigma' : \Dir_K(K) \to \Dir_k(M')$ be the mutation
  along $L$.  Then
  \begin{equation*}
    M' = \varsigma_2(L) = \{(s\varepsilon, t\varepsilon) : s, t \in k\} \cong k^2,
  \end{equation*}
  and
  \begin{equation*}
    \varsigma'_2(L) = \varsigma_4(\{(x, \alpha x; \alpha x, \alpha^2 x ) :
    x \in K\}) = \{(s,t;t,(-b_0t - c_0s)) : s, t \in k\}.
  \end{equation*}
  Note that $\varsigma'(L)$ is the graph of the $k$-linear map
  \begin{align}
    k^2 &\to k^2 \\
    (s,t) &\mapsto (t, -b_0t - c_0s) \label{split-map}
  \end{align}
  Therefore $\alpha$ lies in the fundamental ring $R'$ of
  $\varsigma'$.  By Remark~\ref{rem:inherit2}, $\varsigma'$ continues
  to satisfy the Strong Assumptions, and then by
  Corollary~\ref{cor:det-tr}, the linear map (\ref{split-map}) must
  have a repeated eigenvalue.  Therefore
  \begin{equation*}
    b_0^2 = \left( \Tr \begin{pmatrix}0 & 1 \\ - c_0 & -
      b_0 \end{pmatrix} \right)^2 = 4 \det \begin{pmatrix}0 & 1 \\ -
      c_0 & - b_0 \end{pmatrix} = 4c_0. \qedhere
  \end{equation*}
\end{proof}
\begin{lemma}\label{tilde-o-very-valuation}
  Let $\alpha$ be an element of $K^\times$.  Then $\alpha$ or
  $\alpha^{-1}$ is integral over $R$, and in fact one of $\alpha$ or
  $\alpha^{-1}$ satisfies a monic polynomial equation of degree $d$,
  where
  \begin{equation*}
    d = 
    \begin{cases}
      1 & \alpha \textrm{ tame} \\
      2 & \alpha \textrm{ wild.}
    \end{cases}
  \end{equation*}
\end{lemma}
\begin{proof}
  Let $\alpha$ be an element of $K$.  First suppose $\alpha$ is tame.
  Then there is some $b \in k_0$ such that the number $\alpha' := 1/(\alpha - b)$
  lies in $R$.  Because $R$ is a $k_0$-algebra, it contains $b \alpha'
  + 1$.  Because $R$ is a local $k_0$-algebra, at least one of
  $\alpha'$ and $b \alpha' + 1$ is invertible.  Therefore at least one
  of the following lies in $R$:
  \begin{align*}
    \alpha &= \frac{b \alpha' + 1}{\alpha'} \\
    1/\alpha &= \frac{\alpha'}{b \alpha' + 1}.
  \end{align*}
  Next suppose $\alpha$ is wild.  By Lemma~\ref{three-generators},
  there are $x, y, z \in R$ such that
  \[ x + y \alpha + z \alpha^2 = 0,\]
  and at least one of $x, y, z$ is in $R^\times$.  If $z$ is
  invertible, then
  \begin{equation*}
    \alpha^2 + (y/z)\alpha + (x/z) = 0,
  \end{equation*}
  so $\alpha$ is integral over $R$.  Similarly, if $x$ is invertible,
  then $1/\alpha$ is integral over $R$.  So we may assume $y$ is
  invertible and $x, z$ are not.  Then
  \[ 0 \equiv x \equiv z \not \equiv y \pmod{\pp}.\]
  Let $\beta = \frac{1+\alpha}{1 - \alpha}$, so that $\alpha =
  \frac{\beta - 1}{\beta + 1}$.  Then
  \begin{align*}
    x + y \alpha + z\alpha^2 &= 0 \\
    (\beta + 1)^2x + (\beta-1)(\beta+1)y + (\beta-1)^2z &= 0 \\
    (\beta^2 + 2\beta + 1)x + (\beta^2 - 1)y + (\beta^2 - 2\beta + 1)z &= 0 \\
    (x + y + z)\beta^2 + (2x - 2z)\beta + (x + z - y) &= 0 \\
    \beta^2 + \frac{2x-2z}{x + z + y}\beta + \frac{x+z-y}{x+z+y} &= 0
  \end{align*}
  where in the final line we have used the fact that $(R,\pp)$ is a local ring and
  \begin{equation*}
    x + z + y \equiv y \not \equiv 0 \pmod{\pp}.
  \end{equation*}
  Note also that
  \begin{align*}
    \frac{2x-2z}{x+z+y} &\equiv \frac{0}{y} \pmod{\pp} \\
    \frac{x+z-y}{x+z+y} &\equiv \frac{-y}{y} \pmod{\pp}.
  \end{align*}
  Thus $\beta^2 + b'\beta + c' = 0$ for some $b', c' \in R$ with $b'
  \equiv 0$ and $c' \equiv -1$.  Because $\beta$ and $\alpha$ are
  related by a fractional linear transformation over $k_0$, we know
  that $\beta$ is wild.  As we are not in characteristic 2,
  \begin{equation*}
    (b')^2 = 0 \not \equiv -4 \equiv 4c' \pmod{\pp},
  \end{equation*}
  contradicting Lemma~\ref{discriminant}.
\end{proof}
\begin{corollary} \label{cor:o}
  Let $\Oo$ denote the integral closure of $R$ (in $K$).
  Then $\Oo$ is a valuation ring on $K$.
\end{corollary}
\subsection{Tameness and $\Oo$} \label{sec:tw}
Let $\mm$ denote the maximal ideal of the valuation ring $\Oo$.
\begin{lemma}\label{same-residue-1:lem}
  The intersection $\mm \cap R$ is exactly the prime ideal $\pp$.
\end{lemma}
\begin{proof}
  Let $\alpha$ be an element of $R$.  First suppose $\alpha \notin
  \pp$.  Then $\alpha^{-1} \in R \subseteq \Oo$, so $\alpha
  \notin \mm$.  Conversely, suppose $\alpha \in \pp$ but $\alpha
  \notin \mm$.  Then $\alpha^{-1} \in \Oo$, so there exist
  $c_0, c_1, \ldots, c_{n-1} \in R$ such that
  \begin{equation*}
    \alpha^{-n} + c_{n-1} \alpha^{1-n} + \cdots + c_1 \alpha^{-1} + c_0 = 0,
  \end{equation*}
  or equivalently
  \begin{equation*}
    -1 = c_{n-1} \alpha + c_{n-2} \alpha^2 + \cdots + c_1 \alpha^{n-1} + c_0 \alpha^n.
  \end{equation*}
  But the right-hand side is in $\pp$ and the left hand side is not, a
  contradiction.
\end{proof}
\begin{corollary}
  The valuation ring $\Oo$ is non-trivial.
\end{corollary}
\begin{proof}
  By surjectivity of $\wres : R \to k[\varepsilon]$, we can find $x
  \in R$ with $\wres(x) = \varepsilon$.  Then $x \ne 0$, but $x \in
  \pp$.  Therefore $\mm \ne 0$ and $\Oo \ne K$.
\end{proof}
\begin{lemma}\label{tame-is-r}
  If $\alpha \in \Oo$, then $\alpha$ is tame if and only if
  $\alpha \in R$.
\end{lemma}
\begin{proof}
  If $\alpha \in R$ then $\alpha$ is tame by definition.  Conversely,
  suppose $\alpha$ is tame.  By Lemma~\ref{tilde-o-very-valuation},
  one of $\alpha$ or $1/\alpha$ is in $R$.  If $\alpha \in R$, we are
  done.  Otherwise, $1/\alpha \in R$ and $\alpha \notin R$, so
  $1/\alpha$ is a non-invertible element of $R$.  Then $1/\alpha \in
  \pp \subseteq \mm$, so $\alpha$ has negative valuation,
  contradicting the assumption that $\alpha \in \Oo$.
\end{proof}

\subsection{The residue map} \label{sec:res}
\begin{proposition}\label{same-residue-2:prop}
  The induced map
  \begin{equation*}
    k \cong R/\pp = R/(R \cap \mm) \hookrightarrow \Oo/\mm
  \end{equation*}
  is onto, hence an isomorphism.  In particular, $\Oo$ has
  residue field isomorphic to $k$.
\end{proposition}
\begin{proof}
  We must show that for every $x \in \Oo$ there
  exists $y \in R$ such that
  \begin{equation*}
    x \equiv y \pmod{\mm}.
  \end{equation*}
  If $x \in R$ we can take $y = x$, so we may assume $x
  \notin R$.  Then $x$ is wild by Lemma~\ref{tame-is-r}.  If
  $x \in \mm$ we can take $y = 0$.  So we may assume that
  $x^{-1} \in \Oo$.  By Lemma~\ref{tilde-o-very-valuation}, at
  least one of $x$ or $x^{-1}$ satisfies a monic quadratic
  polynomial equation over $R$.

  First suppose it is $x$.  Then
  \begin{equation} \label{the-eq} x^2 + c_1 x + c_0 = 0 \end{equation}
  for some $c_0, c_1 \in R$.  By Lemma~\ref{discriminant},
  \begin{equation*}
    c_1^2 - 4c_0 \in \pp.
  \end{equation*}
  As we are not in characteristic 2, we may rewrite (\ref{the-eq}) as
  \begin{equation*}
    \left(x + \frac{c_1}{2}\right)^2 = \left(\frac{c_1^2}{4} - c_0\right).
  \end{equation*}
  The right hand side is in $\pp \subseteq \mm$, so it has positive
  valuation.  Therefore $x + c_1/2$ also has positive valuation:
  \begin{equation*}
    x + \frac{c_1}{2} \in \mm.
  \end{equation*}
  Thus we can take $y = -c_1/2$.

  Next suppose $x^{-1}$ satisfies a monic quadratic polynomial
  equation over $R$:
  \[ x^{-2} + c_1 x^{-1} + c_0 = 0.\]
  The same argument shows that $x^{-1} \equiv b \pmod{\mm}$ for
  some $b \in R$.  Then $x \equiv b^{-1} \pmod{\mm}$, and
  \begin{equation*}
    x \notin \mm \implies b \notin \mm \implies b \notin \pp
    \implies b \in R^\times \implies b^{-1} \in R,
  \end{equation*}
  so we can take $y = b^{-1}$.
\end{proof}

We let $\res : \Oo \to k$ denote the natural residue map.  Note that
if $a \in R$ and
\[ \wres(a) = x + y \varepsilon,\]
then $\res(a) = x$.

\subsection{The limiting ring} \label{sec:limring}

\begin{lemma}\label{possibly-useless}
  Let $\alpha$ be an element of $K$.  Suppose $\alpha \notin \Oo$, and
  $\alpha^{-1} \notin \pp$.
  Let $\varsigma'$ denote the mutation along $K \cdot (1, \alpha^{-1})$,
  let $R'$ denote the fundamental ring of $\varsigma'$, and let $\pp'$
  denote the maximal ideal of $R'$.  Then $\alpha^{-1} \in \pp'$.
\end{lemma}
\begin{proof}
  First note that $\alpha^{-1} \in \mm \subseteq \Oo$.  If
  $\alpha^{-1}$ is tame, then $\alpha^{-1} \in R$ by
  Lemma~\ref{tame-is-r}, and then $\alpha^{-1} \in \pp$ by
  Lemma~\ref{same-residue-1:lem}.  This contradicts the assumptions,
  and so $\alpha^{-1}$ and $\alpha$ are wild.  By
  Lemma~\ref{tilde-o-very-valuation}, one of $\alpha$ and
  $\alpha^{-1}$ satisfies a monic quadratic polynomial equation over
  $R$.  Since $\alpha$ does not lie in the integral closure $\Oo$ of
  $R$, it must be $\alpha^{-1}$ that satisfies the equation:
  \[ \alpha^{-2} = b \alpha^{-1} + c.\]
  Then $b^2 + 4c \in \pp$ by Lemma~\ref{discriminant}.  We claim that
  $c \in \pp$.  Otherwise, $c \in R^\times$, and
  \[ c^{-1} = c^{-1} b \alpha + \alpha^2,\]
  contradicting the fact that $\alpha$ is not integral over $R$.  So
  \begin{align*}
    b^2 + 4c & \equiv 0 \pmod{\pp} \\
    c & \equiv 0 \pmod{\pp}
  \end{align*}
  and therefore $b \in \pp$ as well.  Let $\beta, \gamma$ be
  $\wres(b)$ and $\wres(c)$, respectively.  The fact
  that $b, c \in \pp$ implies that $\beta, \gamma \in k\varepsilon$, and
  therefore $\beta, \gamma$ annihilate $k\varepsilon$.

  As $\alpha^{-1}$ is wild and $k \varepsilon$ is the degeneracy subspace,
  \begin{equation*}
    \varsigma_2(\{(x, \alpha^{-1} x) : x \in K\}) =\{(s\varepsilon,t\varepsilon)
    : s, t \in k\}.
  \end{equation*}
  By the usual inflator calculus, one sees that
  \begin{align*}
    \varsigma_3(\{(x, \alpha^{-1} x, \alpha^{-2} x) : x \in K\}) &=
    \varsigma_3(\{(x, \alpha^{-1} x, b \alpha^{-1} x + c x) : x \in K\}) \\
    &= \{(s \varepsilon, t \varepsilon, \beta t \varepsilon + \gamma s \varepsilon) : s, t \in k\} \\
    &= \{(s \varepsilon, t \varepsilon, 0) : s, t \in k\}.
  \end{align*}
  Now let $\varsigma'$ be the mutation of $\sigma$ along $K \cdot (1,
  \alpha^{-1})$.  Then
  \begin{align*}
    \varsigma'_2(\{(x, \alpha^{-1}x) : x \in K\}) & =
    \varsigma_4(\{(x, \alpha^{-1} x; \alpha^{-1} x, \alpha^{-2} x) : x
    \in K\}) \\ &= \{(s,t;t,0) : s, t \in k\}.
  \end{align*}
  Thus $\alpha^{-1}$ specializes to the endomorphism
  \begin{equation*}
    (s,t) \mapsto (t,0),
  \end{equation*}
  and so $\alpha^{-1} \in R'$.  This endomorphism fails to be
  invertible, so $\alpha^{-1} \in \pp'$.
\end{proof}

\begin{lemma}\label{lem:bite-size}
  Let $\varsigma'$ be a mutation of $\varsigma$, with fundamental ring
  $R'$.  Let $\pp'$ be the maximal ideal of $R'$.  Then $\pp \subseteq
  \pp'$.
\end{lemma}
\begin{proof}
  Suppose $\varsigma'$ is the mutation along $L = K \cdot (a_1,
  \ldots, a_n)$.  Then $\varsigma'$ is a map
  \[ \Dir_K(K) \to \Dir_k(M'),\]
  where $M' = \varsigma_n(L) \subseteq (k[\varepsilon])^n$.  By
  Proposition~\ref{prop:descend-to-A}, $M'$ is a
  $k[\varepsilon]$-submodule of $(k[\varepsilon])^n$.

  Take $b \in \pp$.  Then $\wres_\varsigma(b) = s \varepsilon$
  for some $s \in k$.  This means that $b \in R$ specializes (with
  respect to $\varsigma$) to the endomorphism
  \begin{align*}
    k[\varepsilon] & \to k[\varepsilon] \\
    z & \mapsto s \varepsilon z.
  \end{align*}
  With respect to the mutation $\varsigma'$, the element $b$
  specializes to the endomorphism
  \begin{align*}
    M' & \to M' \\
    \vec{v} & \mapsto s \varepsilon \vec{v},
  \end{align*}
  by Lemma~10.3 in \cite{prdf3}.  This map is not onto, by Nakayama's
  lemma (over the Noetherian ring $k[\varepsilon]$).  Therefore $b \in
  R'$, but $b^{-1} \notin R'$, implying that $b \in \pp'$.
\end{proof}
\begin{proposition}\label{persistence}
  If $\varsigma'$ is a mutation of $\varsigma$, with fundamental
  ring $R'$, then $R' \subseteq \Oo$.  Consequently,
  $R'$ has the same integral closure as $R$.
\end{proposition}
\begin{proof}
  Let $b$ be an element of $R'$ that is not in $\Oo$.  First suppose that $b^{-1} \in \pp$.  Then Lemma~\ref{lem:bite-size} implies $b^{-1} \in \pp'$.  Therefore $b \notin R'$, a contradiction.
  
  Next suppose that $b^{-1} \notin \pp$.  Let $\tau$ and $\tau'$ be the
  mutations of $\varsigma$ and $\varsigma'$ along $K \cdot (1,
  b^{-1})$.    By
  Lemma~\ref{possibly-useless}, $b^{-1} \in \pp_\tau$.  By commutativity of mutation (Remark~10.7 in
  \cite{prdf3}), $\tau'$ is a mutation of $\tau$.  By Lemma~\ref{lem:bite-size},
  \[ b^{-1} \in \pp_\tau \subseteq \pp_{\tau'}.\]
  This implies $b \notin R_{\tau'}$.  But $b \in R' = R_{\varsigma'} \subseteq R_{\tau'}$, a contradiction.
\end{proof}

\begin{corollary} \label{cor:o2}
  The integral closure $\Oo$ is the limiting ring $R_\infty$ of
  \cite{prdf3}, Definition~10.9.
\end{corollary}
\begin{proof}
  The limiting ring $R_\infty$ is integrally closed, so $R_\infty
  \supseteq \Oo$.  On the other hand $R_\infty$ is a union
  of rings $R'$ obtained by mutation.  Proposition~\ref{persistence} says $R' \subseteq \Oo$.  Thus $R_\infty \subseteq \Oo$.
\end{proof}

\section{Differential structure} \label{sec:der}
Continue the Strong Assumptions of \S\ref{sec:canform}-\ref{sec:the-val}.

\begin{remark}
  Over the next few sections, we will carry out a number of convoluted
  calculations.  The motivated reader may wish to keep two running
  examples in mind:
  \begin{itemize}
  \item The diffeovaluation inflators of \S\ref{dvi:sec} below.
  \item The ``endless mutation'' example of \S 12.3 in \cite{prdf3}.
  \end{itemize}
  The second example doesn't actually satisfy the Strong Assumptions,
  but this won't matter until Lemma~\ref{weird-elements}.  (See
  Remark~\ref{ohbtw}.)
\end{remark}

\subsection{Double mutation lemma}
The idea of the next few lemmas is that we can calculate the residue
$\res(r)$ of an element $r$ by passing to a mutation where $r$ becomes
tame; Proposition~\ref{persistence} ensures that the valuation does
not change in the mutation.
\begin{lemma} \label{advance-notice}
  Let $r$ be an element of $K$, let $q$ be an element of $k$, and let
  $L$ be a line (a one-dimensional subspace) in $K^n$.  Suppose that
  every element of
  \begin{equation*}
    \varsigma_{2n}(\{(\vec{x}, r \vec{x}) : \vec{x} \in L\})
  \end{equation*}
  is of the form $(\vec{x}, q \vec{x})$.  Then $\val(r) \ge 0$ and
  $\res(r) = q$.
\end{lemma}
\begin{proof}
  Let $M' = \varsigma_n(L)$.  Then
  \begin{align*}
    \varsigma_{2n}(\{(\vec{x}, r \vec{x}) : \vec{x} \in L\}) &\subseteq \{(\vec{x}, q \vec{x}) : \vec{x} \in k[\varepsilon]^n\} \\
    \varsigma_{2n}(\{(\vec{x}, r \vec{x}) : \vec{x} \in L\}) & \subseteq \varsigma_{2n}(L \oplus L) = M' \oplus M' \\
    \varsigma_{2n}(\{(\vec{x}, r \vec{x}) : \vec{x} \in L\}) & \subseteq \{(\vec{x}, q \vec{x}) : \vec{x} \in M'\}.
  \end{align*}
  The first line is by assumption, the second line is by
  order-preservation and $\oplus$-compatibility, and the third line
  follows by intersecting the first two lines.  Counting lengths,
  equality must hold in the second line.  Let
  \[ \varsigma' : \Dir_K(K) \to \Dir_k(M')\]
  be the mutation of $\varsigma$ along $L$.  Let $R', \pp', I', \Oo',
  \mm'$ denote the analogues of $R, \pp, I, \Oo, \mm$ for the mutation
  $\varsigma'$.  Then
  \begin{equation*}
    \varsigma'_{2}(K \cdot (1,r)) = \varsigma_{2n}(\{(\vec{x}, r
    \vec{x}) : \vec{x} \in L\}) = \{(\vec{x}, q \vec{x}) : \vec{x} \in
    M'\}.
  \end{equation*}
  It follows that $r$ specializes with respect to
  $\varsigma'$ to the endomorphism
  \begin{align*}
    M' & \to M' \\
    x & \mapsto qx.
  \end{align*}
  Thus $r \in R' \subseteq \Oo'$.  Choose some $p \in R$ such that $\wres(p) = q = q + 0\varepsilon$.
  By Lemma~10.3 in \cite{prdf3}, the element $p$ is also in $R'$, and
  also specializes to this endomorphism.  Therefore
  \[ r - p \in I' \subseteq \pp' \subseteq \mm'.\]
  By Proposition~\ref{persistence}, $\Oo = \Oo'$ and $\mm = \mm'$,
  implying that $r \in \Oo$ and $r - p \in \mm$.  Therefore $\val(r)
  \ge 0$ and $\res(r) = \res(p) = q$.
\end{proof}
\begin{lemma}\label{out-of-context}
  Suppose $a \in K$ has $\val(a) > 0$, and suppose
  \begin{equation*}
    (u,v) \in \varsigma_2(K \cdot (1,a))
  \end{equation*}
  for some $u, v \in k[\varepsilon]$ with $v \ne 0$.  Then there is $a' \in K$ such
  that
  \begin{equation*}
    (k \cdot v) \oplus (k \cdot v) = \varsigma_2(K \cdot
    (1,a')).
  \end{equation*}
\end{lemma}
\begin{proof}
  If $a$ is wild, then
  \[ \varsigma_2(K \cdot (1,a)) = k\varepsilon \oplus k\varepsilon\]
  by Lemmas~\ref{lem:degen} and \ref{lem:which-degen}.  Then $v \in
  k\varepsilon$, and
  \[ (k \cdot v) \oplus (k \cdot v) = k\varepsilon \oplus k\varepsilon.\]
  So we may take $a' = a$.

  Otherwise, $a$ is tame, and so $a \in R \cap \mm = \pp$ by
  Lemmas~\ref{same-residue-1:lem} and \ref{tame-is-r}.  Then $\wres(a)
  = b \varepsilon$ for some $b \in k$, and
  \begin{equation*}
    (u,v) \in \varsigma_2(K \cdot (1,a)) =
    \{(x+y\varepsilon,(b\varepsilon)(x+y\varepsilon)) : x, y \in k\} =
    \{(x+y\varepsilon,0+(xb)\varepsilon) : x, y \in k\}.
  \end{equation*}
  Thus $v \in k\varepsilon$.  Then we can take $a'$ to be any wild
  element, and \[ (k \cdot v) \oplus (k \cdot v) = k\varepsilon
  \oplus k\varepsilon = \varsigma_2(K \cdot (1,a')). \qedhere\]
\end{proof}
\begin{lemma}[Double mutation lemma] \label{lem:wtf}
  Let $a, r$ be elements of $K$, with $\val(a) > 0$.  Suppose that
  \[ (s, t, u, qu) \in \varsigma_4(K \cdot (1,r,a,ar))\]
  for some $s, t, u \in k[\varepsilon]$ and $q \in k$ with $u$ nonzero.
  Then $\val(r) \ge 0$ and $\res(r) = q$.
\end{lemma}
\begin{proof}
  Let $\varsigma' : \Dir_K(K) \to \Dir_k(M')$ be the mutation of
  $\varsigma$ along $K \cdot (1,r)$, where $M' = \varsigma_2(M')
  \subseteq (k[\varepsilon])^2$.  Then
  \begin{equation*}
    \varsigma'_2(K \cdot (1,a)) = \varsigma_4(K \cdot (1,r,a,ar)) \ni
    (s, t, u, qu).
  \end{equation*}
  By Lemma~\ref{out-of-context} applied to $\varsigma'$, there is some
  $a' \in K$ such that
  \begin{equation*}
    \varsigma'_2(K \cdot (1,a')) = \{(x u, x qu; y u, y qu) : x, y \in
    k\}.
  \end{equation*}
  Equivalently, then
  \begin{align*}
    \varsigma_4(K \cdot (1,r, a', a'r)) &=
    \{(xu, x qu, yu, yqu) : x, y \in k\} \\
    \varsigma_4(K \cdot (1, a', r, a'r)) &=
    \{(xu, yu, xqu, yqu) : x, y \in k\}.
  \end{align*}
  By Lemma~\ref{advance-notice} applied to the line $L = K \cdot
  (1,a')$, it follows that $\val(r) \ge 0$ and $\res(r) = q$.
\end{proof}

\subsection{Neutralizers} \label{sec:neut}
Recall from \S\ref{ssec:fgi} that $Q$ is the subring \[ Q = \{x \in R
: \wres(x) \in k\},\] where we view $k$ as a subset of
$k[\varepsilon]$ in the natural way.

\begin{definition} \label{def:neut}
  If $a \in \Oo$, a \emph{neutralizer} is an $a^\dag \in Q$ such that
  $aa^\dag \in R \setminus Q$.
\end{definition}
Neutralizers need not be unique.

\begin{lemma}\label{initial}
  If $a \in \Oo \setminus Q$, then $a$ has a neutralizer.
\end{lemma}
\begin{proof}
  First suppose $a \in R \setminus Q$.  Then $1$ is a neutralizer.

  Next suppose $a \in \Oo \setminus R$.  By Lemma~\ref{tame-is-r}, the
  element $a$ is wild.  Then
  \begin{equation*}
    \varsigma_2(K \cdot (1,a)) = k\varepsilon \oplus k\varepsilon.
  \end{equation*}
  On the other hand, by Proposition~\ref{exact-form},
  \begin{equation*}
    \varsigma_2(K \cdot (1,a)) =
    \{(\wres(x),\wres(ax)) : x \in R, ax \in R\}.
  \end{equation*}
  Therefore, there is some $a^\dag \in R$ such that $aa^\dag \in R$,
  and
  \begin{align*}
    \wres(a^\dag) &= 0 \\
    \wres(aa^\dag) &= \varepsilon.
  \end{align*}
  Then $a^\dag \in Q$ and $aa^\dag \in R \setminus Q$.
\end{proof}
\begin{lemma}\label{major-mess-0}
  Let $a$ be a wild element.
  \begin{itemize}
  \item If $b \in R$ and $ab \in R$, then
    \begin{align*}
      \wres(b) &= p \varepsilon \\
      \wres(ab) &= q \varepsilon
    \end{align*}
    for some $p, q \in k$.
  \item If $a^\dag$ is a neutralizer of $a$, then
    \begin{align*}
      \wres(a^\dag) &= 0 \\
      \wres(aa^\dag) &= q\varepsilon.
    \end{align*}
    for some nonzero $q \in k$.
  \end{itemize}
\end{lemma}
\begin{proof}
  Because $a$ is wild and $k\varepsilon$ is the degeneracy subspace,
  \begin{equation*}
    \{(\wres(x),\wres(ax)) : x \in R,~ ax \in R\} = \varsigma_2(K \cdot
    (1,a)) = k\varepsilon \oplus k\varepsilon.
  \end{equation*}
  The fact that $b, ab \in R$ thus implies that $\wres(b)$ and
  $\wres(ab)$ lie in $k\varepsilon$, and so
  \begin{align*}
    \wres(b) &= 0 + p\varepsilon \\
    \wres(ab) &= 0 + q\varepsilon,
  \end{align*}
  for some $p, q \in k$.  When $b$ is a neutralizer $a^\dag$, we must
  have $p = 0$ and $q \ne 0$, because $a^\dag \in Q$ and $aa^\dag \notin Q$.
\end{proof}

\begin{lemma}\label{major-mess}
  Let $a$ be a wild element with $\val(a) > 0$, and let $a^\dag$ be a
  neutralizer of $a$.  Suppose $b \in R$ and $ab \in R$.  Then
  $\val(b) \ge \val(a^\dag)$.  Moreover,
  \begin{equation*}
    ab \in Q \iff \val(b) > \val(a^\dag).
  \end{equation*}
\end{lemma}
\begin{proof}
  By Lemma~\ref{major-mess-0},
  \begin{align*}
    \wres(a^\dag) &= 0 \\
    \wres(b) &= p\varepsilon \\
    \wres(aa^\dag) &= s\varepsilon \\
    \wres(ab) &= q\varepsilon,
  \end{align*}
  for some $p, q, s \in k$ with $s \ne 0$.  Note $ab \in Q \iff q = 0$.
  It suffices to show that $\val(b/a^\dag) \ge 0$ and $\res(b/a^\dag) = q/s$.

  By the inflator calculus,
  \begin{align*}
    \varsigma_5(K \cdot (1,a^\dag,b,aa^\dag,ab)) &=
    \varsigma_5(\{(x,a^\dag x, bx, aa^\dag x, abx) : x \in K\}) \\
    &= \{(x, 0, (p \varepsilon) x, (s \varepsilon) x, (q \varepsilon) y : x \in
    k[\varepsilon] \} \\
    &= (k[\varepsilon]) \cdot (1,0,p \varepsilon, s \varepsilon, q \varepsilon),
  \end{align*}
  and therefore
  \begin{equation*}
    (0,p \varepsilon, s \varepsilon, q \varepsilon) \in \varsigma_4(K \cdot
    (a^\dag, b, aa^\dag, ab)) = \varsigma_4(K \cdot (1, b/a^\dag, a,
    ab/a^\dag)).
  \end{equation*}
  By the Double Mutation Lemma~\ref{lem:wtf} with $r = b/a^\dag$, it
  follows that $\val(b/a^\dag) \ge 0$ and $\res(b/a^\dag) = q/s$.
\end{proof}

\begin{lemma}\label{major-mess-2}
  Let $a$ be a wild element with $\val(a) > 0$, and let $a^\dag$ be a
  neutralizer.  Suppose $b \in R$ and $ab \notin R$.  Then $\val(b) <
  \val(a^\dag)$.
\end{lemma}
\begin{proof}
  Note that $b \in R \subseteq \Oo$ and $a \in \mm \subseteq \Oo$, so
  $ab \in \Oo$.  The fact that $ab \notin R$ then implies that $ab$ is
  wild, by Lemma~\ref{tame-is-r}.  Also, $ab \notin Q$, as $Q
  \subseteq R$.  By Lemma~\ref{initial}, $ab$ has a neutralizer
  $(ab)^\dag$.  By Lemma~\ref{major-mess-0},
  \begin{align*}
    \wres((ab)^\dag) &= 0 \\
    \wres(ab(ab)^\dag) & = s \varepsilon,
  \end{align*}
  for some non-zero $s \in k$.
  Let $c = b(ab)^\dag$.  Then $c \in R$ (because $b \in R$ and
  $(ab)^\dag \in Q \subseteq R$).  Also
  \begin{equation*}
    ac = (ab)(ab)^\dag \in R \setminus Q.
  \end{equation*}
  By Lemma~\ref{major-mess},
  \begin{equation*}
    ac \notin Q \implies \val(c) = \val(a^\dag).
  \end{equation*}
  Then
  \begin{equation*}
    \val(a^\dag) = \val(c) = \val(b) + \val((ab)^\dag) > \val(b),
  \end{equation*}
  because $\res((ab)^\dag) = 0$.
\end{proof}
Lemmas~\ref{major-mess} and \ref{major-mess-2} combine to yield the
following:
\begin{lemma}\label{major-mess-3}
  Let $a$ be a wild element with $\val(a) > 0$, and let $a^\dag$ be a
  neutralizer of $a$.  Suppose $b \in R$.
  \begin{align*}
    ab \in R &\iff \val(b) \ge \val(a^\dag) \\
    ab \in Q &\iff \val(b) > \val(a^\dag).
  \end{align*}
\end{lemma}
Next, we weaken the assumption on $a$, allowing $\val(a) = 0$:
\begin{lemma}\label{major-mess-4}
  Let $a$ be a wild element with $\val(a) \ge 0$, and let $a^\dag$ be
  a neutralizer of $a$.  Suppose $b \in R$.
  Then
  \begin{equation*}
    ab \in R \iff \val(b) \ge \val(a^\dag).
  \end{equation*}
  If $b \in Q \subseteq R$, then
  \begin{equation*}
    ab \in Q \iff \val(b) > \val(a^\dag).
  \end{equation*}
\end{lemma}
\begin{proof}
  Let $\res(a) = \gamma$, and choose $c \in R$ with
  $\wres(c) = \gamma + 0\varepsilon$.  Then $c \in Q$ and
  $\res(c) = \gamma$.  Let $a' = a - c$.  Then $\res(a') =
  \res(a) - \res(c) = \gamma - \gamma = 0$.  So $\val(a') > 0$.
  Because $a \notin R$ and $c \in R$, we have $a' = a -
  c \notin R$, and so $a'$ is wild by Lemma~\ref{tame-is-r}.
  Moreover,
  \begin{equation*}
    a'a^\dag = aa^\dag - c a^\dag \in (R \setminus Q) - Q = R \setminus Q.
  \end{equation*}
  Therefore $a^\dag$ is a neutralizer of $a'$, and we can apply
  Lemma~\ref{major-mess-3} to $a', a^\dag, b$.  Then
  \begin{equation*}
    ab \in R \iff a'b \in R \iff \val(b) \ge \val(a^\dag),
  \end{equation*}
  since $ab - a'b = c b \in R$.  If $b \in Q$, then $c b \in
  Q$ and
  \begin{equation*}
    ab \in Q \iff a'b \in Q \iff \val(b) > \val(a^\dag). \qedhere
  \end{equation*}
\end{proof}

\subsection{The secondary valuation}

\begin{definition}\label{vp-def}
  For $a \in \Oo$, let $\val_\partial(a)$ denote
  \begin{equation*}
    \val_\partial(a) = \begin{cases}
      +\infty & \text{ if } a \in Q \\
      -\val(a^\dag) & \text{ if $a$ has a neutralizer $a^\dag$}
    \end{cases}
  \end{equation*}
\end{definition}

\begin{lemma}\label{vp-well-def}
  $\val_\partial(a)$ is well-defined: for any $a \in \Oo$,
  exactly one of the following holds
  \begin{enumerate}
  \item $a \in Q$
  \item \label{dag2} $a$ has a neutralizer $a^\dag$,
  \end{enumerate}
  and in case (\ref{dag2}) the valuation $\val(a^\dag)$ is independent
  of the choice of a neutralizer $a^\dag$.
\end{lemma}
\begin{proof}
  If $a \notin Q$, then a neutralizer $a^\dag$ exists by
  Lemma~\ref{initial}.  Conversely, if a neutralizer $a^\dag$ exists,
  then $a^\dag \in Q$ and $aa^\dag \in R \setminus Q$.  As $Q$ is a
  subring, $a \notin Q$.

  Now suppose $a^\dag$ and $b$ are two neutralizers of $a$.  If $a$ is
  wild, then Lemma~\ref{major-mess-4} applies, so
  \begin{equation*}
    ab \in R \setminus Q \implies \val(b) = \val(a^\dag).
  \end{equation*}

  If $a$ is tame, we can write
  \begin{align*}
    \wres(a) &= x + y\varepsilon \\
    \wres(a^\dag) &= z + 0\varepsilon \\
    \wres(b) &= w + 0\varepsilon
  \end{align*}
  for some $x, y, z, w \in k$, using the fact that $a^\dag, b \in Q$.
  Then
  \begin{align*}
    \wres(aa^\dag) &= xz + yz\varepsilon \\
    \wres(ab) &= xw + yw \varepsilon,
  \end{align*}
  The fact that $aa^\dag, ab \in R \setminus Q$ implies that $yz$ and
  $yw$ are non-zero.  Therefore $z, w$ are non-zero and $\res(a^\dag)$
  and $\res(b)$ are non-zero, implying $\val(a^\dag) = 0 = \val(b)$.
\end{proof}

\begin{lemma}\label{vp-class}
  For $a \in \Oo$,
  \begin{itemize}
  \item $\val_\partial(a) \ge 0 \iff a \in R$
  \item $\val_\partial(a) > 0 \iff a \in Q$.
  \end{itemize}
\end{lemma}
\begin{proof}
  First suppose $a \in Q$.  Then $\val_\partial(a) = +\infty$ by
  definition.  Next suppose that $a \in R \setminus Q$.  Then $1$ is a
  neutralizer of $a$, so $\val_\partial(a) = -\val(1) = 0$.

  Lastly, suppose $a \in \Oo \setminus R$.  Let $a^\dag$ be a
  neutralizer.  By Lemma~\ref{major-mess-0}, $\wres(a^\dag) = 0$,
  implying $\res(a^\dag) = 0$, $\val(a^\dag) > 0$, and
  $\val_\partial(a) < 0$.
\end{proof}

\begin{lemma}\label{vp-scalars}
  If $a \in \Oo$ and $b \in Q$, then
  \begin{equation*}
    \val_\partial(ab) = \val_\partial(a) + \val(b)
  \end{equation*}
  unless the right hand side is positive, in which case
  \begin{equation*}
    \val_\partial(ab) = +\infty.
  \end{equation*}
\end{lemma}
\begin{proof}
  First suppose $a \in Q$.  Then $\val_\partial(a) = +\infty$, and the
  conclusion says that $ab \in Q$, which is true ($Q$ is a ring).

  Next, suppose that $a \in R \setminus Q$.  Then $\val_\partial(a) =
  0$.  As, $a, ab \in R$ and $b \in Q$, we may write
  \begin{align*}
    \wres(a) &= x + y \varepsilon \\
    \wres(b) &= z + 0\varepsilon \\
    \wres(ab) &= xz + yz \varepsilon.
  \end{align*}
  for some $x, y, z \in k$, with $y \ne 0$.  Then
  \begin{align*}
    \val(b) > 0 & \implies z = \res(b) = 0 \implies yz = 0 \implies ab
    \in Q \implies \val_\partial(ab) = +\infty \\
    \val(b) = 0 & \implies z = \res(b) \ne 0 \implies yz \ne 0
    \implies ab \in R \setminus Q \implies \val_\partial(ab) = 0.
  \end{align*}
  Lastly, suppose $a \in \Oo \setminus R$, so $a$ is wild.  Take a
  neutralizer $a^\dag$.  We break into cases according to the sign of
  $\val_\partial(a) + \val(b)$.
  \begin{itemize}
  \item If $\val_\partial(a) + \val(b) > 0$, then $\val(b) >
    \val(a^\dag)$, so $ab \in Q$ by Lemma~\ref{major-mess-4}.  Thus
    $\val_\partial(ab) = +\infty$.
  \item If $\val_\partial(a) + \val(b) = 0$, then $\val(b) =
    \val(a^\dag)$, so $ab \in R \setminus Q$, by
    Lemma~\ref{major-mess-4}.  Thus $\val_\partial(ab) = 0$.
  \item If $\val_\partial(a) + \val(b) < 0$, then $\val(b) <
    \val(a^\dag)$, and $ab \notin R$ by Lemma~\ref{major-mess-4}.  On
    the other hand, $a, b \in \Oo$, so $ab \in \Oo \setminus R$, and
    $ab$ is wild (Lemma~\ref{tame-is-r}).  Take a neutralizer
    $(ab)^\dag$, and let $c = b(ab)^\dag$.  Then $b, (ab)^\dag \in
    Q$, so $c \in Q$.  Also $ac = (ab)(ab)^\dag \in R \setminus Q$, and so $c$
    is a neutralizer of $a$.  Then
    \begin{equation*}
      \val_\partial(ab) = -\val((ab)^\dag) = -\val(c) + \val(b) =
      \val_\partial(a) + \val(b). \qedhere
    \end{equation*}
  \end{itemize}
\end{proof}
The next lemma says that for any $\gamma$, the set
\[ \{x \in \Oo : \val_\partial(x) \ge \gamma\}\]
is a subring of $\Oo$.
\begin{lemma}\label{vp-ops}
  For any $a, b \in \Oo$, let $\gamma =
  \min(\val_\partial(a),\val_\partial(b))$.  Then
  \begin{align*}
    \val_\partial(a+b) &\ge \gamma \\
    \val_\partial(ab) &\ge \gamma.
  \end{align*}
\end{lemma}
\begin{proof}
  If $\gamma > 0$ (i.e., $\gamma = +\infty$), this holds because $Q$
  is a ring.  If $\gamma = 0$, this holds because $R$ is a ring.
  So we may assume
  \begin{equation*}
    0 > \gamma = \val_\partial(a) \le \val_\partial(b),
  \end{equation*}
  swapping $a$ and $b$ if necessary.  The fact that $\val_\partial(a)
  < 0$ implies $a$ is wild.  Take a neutralizer $a^\dag$ of $a$.  Then
  $a^\dag \in Q$, so by Lemma~\ref{vp-scalars},
  \begin{equation*}
    \val_\partial(a^\dag b) \ge \val(a^\dag) + \val_\partial(b) =
    \val_\partial(b) - \val_\partial(a) \ge 0,
  \end{equation*}
  so $a^\dag b \in R$ by Lemma~\ref{vp-class}.  Then
  \begin{equation*}
    a^\dag(a+b) = aa^\dag + a^\dag b \in R + R = R,
  \end{equation*}
  so
  \begin{equation*}
    \val_\partial(a^\dag(a+b)) \ge 0.
  \end{equation*}
  By Lemma~\ref{vp-scalars}, it follows that
  \begin{equation*}
    \val_\partial(a+b) + \val(a^\dag) \ge 0,
  \end{equation*}
  or equivalently, that
  \begin{equation*}
    \val_\partial(a+b) \ge -\val(a^\dag) = \val_\partial(a) = \gamma.
  \end{equation*}
  Also, Lemma~\ref{major-mess-4} shows that
  \begin{equation*}
    (a^\dag b \in R \text{ and } \val(a^\dag b) \ge \val(a^\dag)) \implies aa^\dag b \in R,
  \end{equation*}
  as $a$ is wild.  Thus $a^\dag (a b) \in R$, and $\val_\partial(a^\dag
  (a b)) \ge 0$.  As in the case of $a+b$, this implies that
  \[ \val_\partial(ab) + \val(a^\dag) \ge 0,\]
  or equivalently, that $\val_\partial(ab) \ge \gamma$.
\end{proof}
Later (Corollary~\ref{vp-mult-better}), we will get an improved rule
for $\val_\partial(ab)$, but for now we content ourselves with the
following cases:
\begin{lemma}\label{square-into}
  If $a \in \Oo$ and $\val(a) + \val_\partial(a) > 0$, then $a^2 \in
  R$, i.e., $\val_\partial(a^2) \ge 0$.
\end{lemma}
\begin{proof}
  We may assume $a \notin R$, so $a$ is wild.  Take a neutralizer
  $a^\dag$.  Then
  \begin{equation*}
    0 < \val(a) + \val_\partial(a) = \val(a) - \val(a^\dag),
  \end{equation*}
  and so $\val(a) > \val(a^\dag)$.  Therefore $a^\dag/a$ is not
  integral over $R$.  By Lemma~\ref{tilde-o-very-valuation}, the
  inverse $a/a^\dag$ satisfies a monic quadratic polynomial equation
  over $R$.  Therefore
  \[ a^2 = b a a^\dag + c (a^\dag)^2\]
  for some $b, c \in R$.  As $aa^\dag, b, c, a^\dag$ are all in $R$,
  this implies $a^2 \in R$.
\end{proof}
\begin{lemma}\label{small-products}
  Let $\gamma$ be a positive element of the valuation group, and $a,
  b$ be elements of $\Oo$.  Suppose
  \begin{align*}
    \val(a) &> \gamma \\
    \val(b) &> \gamma \\
    \val_\partial(a) &> -\gamma \\
    \val_\partial(b) &> -\gamma.
  \end{align*}
  Then $ab \in R$, i.e., $\val_\partial(ab) \ge 0$.
\end{lemma}
\begin{proof}
  Note that $\val(a) + \val_\partial(a) > 0$ and $\val(b) +
  \val_\partial(b) > 0$, so $a^2, b^2 \in R$ by
  Lemma~\ref{square-into}.  By Lemma~\ref{vp-ops}, $\val_\partial(a +
  b) > -\gamma$, and so similarly $(a+b)^2 \in R$.  Since $R$ is an
  algebra over a field $k_0$ of characteristic $\ne 2$,
  \begin{equation*}
    ab = \frac{(a+b)^2 - a^2 - b^2}{2} \in R. \qedhere
  \end{equation*}
\end{proof}

\subsection{Density}
\begin{lemma}\label{cofinal-valuations}
  For every non-zero $a \in K$, there is non-zero $b \in Q$ such
  that $\val(b) \ge \val(a)$.
\end{lemma}
\begin{proof}
  We may assume $a \in \Oo$ (otherwise take $b = 1$).  By
  Corollary~\ref{cor:fof}, there are non-zero $x, y \in R$ such that
  $x = ay$.  By surjectivity of $\wres : R \to k[\varepsilon]$, there is
  $z \in R$, necessarily non-zero, such that $\wres(z) = \varepsilon$.
  Set $b = xz^2 = ayz^2$.  Then
  \begin{equation*}
    \val(b) \ge \val(a)
  \end{equation*}
  because $y, z \in R \subseteq \Oo$.  And $b \ne 0$, because $y, z
  \ne 0$.  Lastly,
  \begin{equation*}
    \wres(b) = \wres(x)\wres(z^2) = \wres(x) \cdot \varepsilon^2 = 0,
  \end{equation*}
  and so $b \in Q$.
\end{proof}
\begin{lemma}\label{weird-elements}
  For every $\gamma$ in the value group, there exists $a$
  with $\val(a) > \gamma$ and $\val_\partial(a) < -\gamma$.
\end{lemma}
\begin{proof}
  Increasing $\gamma$, we may assume $0 < \gamma = \val(b)$ for some
  $b \in Q$, by Lemma~\ref{cofinal-valuations}.  Since $\varsigma$ is
  not weakly multi-valuation type (Definition~5.27 in \cite{prdf3}),
  the ball of valuative radius $2 \gamma$ cannot be contained in $R$.
  Therefore there is $c \in K$ with $\val(c) > 2 \gamma$ and $c \notin
  R$.  Let $a = c/b$.  Then
  \begin{equation*}
    \val(a) = \val(c) - \val(b) > 2\gamma - \gamma = \gamma > 0.
  \end{equation*}
  Therefore $a \in \Oo$ and $\val_\partial(a)$ is meaningful.  If
  $\val_\partial(a) \ge -\gamma$, then
  \begin{equation*}
    0 \le \val_\partial(a) + \gamma = \val_\partial(a) + \val(b) \le
    \val_\partial(ab) = \val_\partial(c),
  \end{equation*}
  by Lemma~\ref{vp-scalars}.  Then $c \in R$ by Lemma~\ref{vp-class},
  a contradiction.
\end{proof}

\begin{lemma}\label{application}
  If $a, b \in \mm$ and $\val_\partial(a) < \val_\partial(b)$, then
  \begin{equation*}
    b = p + qa
  \end{equation*}
  for some $p, q \in Q$.
\end{lemma}
\begin{proof}
  Applying Lemma~\ref{three-generators} to the set $\{1,a,b\}$ we
  obtain one of three cases:
  \begin{itemize}
  \item $1 = pa + qb$ for some $p, q \in Q$.  This cannot happen, as
    \[ pa + qb \in \Oo \mm + \Oo \mm = \mm \not\ni 1.\]
  \item $a = p + qb$ for some $p, q \in Q$.  By Lemma~\ref{vp-ops},
    \begin{equation*}
      \val_\partial(a) = \val_\partial(p + qb) \ge
      \min(\val_\partial(p),\val_\partial(q),\val_\partial(b)) =
      \val_\partial(b),
    \end{equation*}
    as $\val_\partial(p) = \val_\partial(q) = +\infty$.  This
    contradicts the assumption.
  \item $b = p + qa$ for some $p, q \in Q$. \qedhere
  \end{itemize}
\end{proof}

\begin{proposition}\label{density}
  The set $Q$ is dense in $\Oo$, with respect to the valuation
  topology.
\end{proposition}
\begin{proof}
  We first show that the closure of $Q$ contains $\mm$.  Let $b$ be some element
  of $\mm$.  Let $\gamma$ be a given positive element of the valuation
  group.  We will find $c \in Q$ such that $\val(c - b) > \gamma$.  If
  $b \in Q$, we can take $c = b$.  Otherwise, $\val_\partial(b) \le 0$.
  By Lemma~\ref{weird-elements}, there is $a$ such that $\val(a) >
  \gamma > 0$ and $\val_\partial(a) < \val_\partial(b) \le 0$.  In
  particular, $a \in \mm$.  By Lemma~\ref{application}, there are $c,
  q \in Q$ such that
  \begin{equation*}
    b = c + qa.
  \end{equation*}
  Then $\val(b - c) = \val(q) + \val(a) \ge \val(a) > \gamma$,
  because $q \in Q \subseteq \Oo$.

  Next let $b$ be any element of $\Oo$.  Take $d \in R$ such
  that
  \begin{equation*}
    \wres(d) = \res(b) + 0\varepsilon.
  \end{equation*}
  Then $d \in Q$ and $b - d \in \mm$.  So we can approximate $b-d$
  arbitrarily closely by elements of $Q$.  Equivalently, $b$ is in the
  closure of $d + Q = Q$.
\end{proof}

\subsection{The derivation $\partial$} \label{ssec:der}
Let $D$ be the $Q$-module
\begin{equation*}
  D := \Oo/Q
\end{equation*}
and let $\partial : \Oo \twoheadrightarrow D$ be the natural
$Q$-linear map.
\begin{proposition}\label{prop:q-to-o}
  The $Q$-module structure on $D$ extends to an $\Oo$-module structure
  as follows: for $a, b \in \Oo$,
  \begin{equation*}
    a \cdot \partial b := \partial (a' \cdot b),
  \end{equation*}
  where $a' \in Q$ and $\val(a' - a) + \val_\partial(b) > 0$.  In
  particular, the choice of $a'$ doesn't matter.
\end{proposition}
\begin{proof}
  We first check that $a \cdot \partial b$ is well defined.  We can
  find an $a' \in Q$ such that $\val(a' - a) + \val_\partial(b) > 0$
  by Proposition~\ref{density}.  If $a''$ is another such choice, then
  $a'' - a' \in Q$ and $\val(a'' - a') + \val_\partial(b) > 0$.  By
  Lemma~\ref{vp-scalars}, it follows that
  \begin{equation*}
    \val_\partial((a''-a')b) = +\infty,
  \end{equation*}
  and $(a''-a')b \in Q$ by Lemma~\ref{vp-class}.  Thus $a''b - a'b \in
  Q$ and $\partial(a''b) = \partial(a'b)$.  So the action of $\Oo$ on
  $D$ is well-defined.  Furthermore, the action of $\Oo$ on $D$
  extends the action of $Q$.  (If $a \in Q$, we can take $a' = a$.)

  Next we check the module axioms.
  For the associative law
  \begin{equation*}
    (a_1 \cdot a_2) \partial b \stackrel{?}{=} a_1 (a_2 \partial b),
  \end{equation*}
  take $a_1', a_2' \in Q$ such that
  \begin{align*}
    \val(a_1'a_2' - a_1a_2) & > -\val_\partial (b) \\
    \val(a_2' - a_2) &> -\val_\partial (b) \\
    \val(a_1' - a_1) &> -\val_\partial (b).
  \end{align*}
  This is possible by density of $Q$ and the fact that multiplication
  is continuous.  As $a_2' \in Q$,
  \begin{equation*}
    \val_\partial (a_2' b) \ge \val(a_2') + \val_\partial (b) \ge
    \val_\partial(b),
  \end{equation*}
  by Lemma~\ref{vp-scalars}, and so
  \begin{equation*}
    \val(a_1' - a_1) > -\val_\partial(b) \ge -\val_\partial(a_2'b).
  \end{equation*}
  Thus
  \begin{equation*}
    (a_1a_2)\partial b = \partial(a_1'a_2'b) = a_1\partial(a_2'b) =
    a_1 \cdot (a_2 \partial b).
  \end{equation*}
  The other three module axioms
  \begin{align*}
    (a_1 + a_2) \partial b &= a_1 \partial b + a_2 \partial b \\
    a (\partial b_1 + \partial b_2) &= (a \partial b_1) + (a \partial b_2) \\
    1 \partial b &= b
  \end{align*}
  are proven similarly: one replaces the $a$'s with very close
  elements of $Q$.\footnote{For the second distributive law, one must
    choose $a' \in Q$ such that
    \begin{equation*}
      \val(a' - a) + \min\{\val_\partial(b_1), \val_\partial(b_2),
      \val_\partial(b_1+b_2)\} > 0.
    \end{equation*}
  }
\end{proof}
\begin{proposition}\label{derivation}
  The map $\partial : \Oo \to D$ is a $Q$-linear derivation:
  \begin{itemize}
  \item $\partial q = 0$ for $q \in Q$.
  \item $\partial(ab) = a \partial b + b \partial a$ for $a, b \in \Oo$.
  \end{itemize}
\end{proposition}
\begin{proof}
  The map is $Q$-linear with kernel $Q$ by construction.  Note that $I
  \ne 0$, as $R$ is a domain and $R/I \cong k[\varepsilon]$ is not.  Take
  non-zero $u \in I$.  Then $u \in Q$ and $\val(u) > 0$.  Take
  $\gamma$ a positive element of the value group such that
  \begin{align*}
    \min(\val_\partial(a),\val_\partial(b)) & > - \gamma \\
    \val(u) & < \gamma.
  \end{align*}
  By Proposition~\ref{density}, we can find $a_Q, b_Q \in Q$ and $a',
  b' \in \Oo$ such that
  \begin{align*}
    a &= a_Q + a' \\
    b &= b_Q + b' \\
    \val(a') &> 3 \gamma \\
    \val(b') &> 3 \gamma.
  \end{align*}
  By Lemmas~\ref{vp-ops} and \ref{vp-class},
  \begin{equation*}
    \val_\partial(a') \ge \min(\val_\partial(a),\val_\partial(-a_Q)) =
    \min(\val_\partial(a),+\infty) = \val_\partial(a) > - \gamma.
  \end{equation*}
  If $\val_\partial(a'/u) \le - 2 \gamma$, then
  \begin{equation*}
    \val_\partial(a') = \val_\partial(a'/u) + \val(u) < -2 \gamma + \gamma = - \gamma < 0
  \end{equation*}
  by Lemma~\ref{vp-scalars}, a contradiction.  Thus
  \begin{align*}
    \val_\partial(a'/u) & > - 2 \gamma \\
    \val_\partial(b'/u) & > - 2 \gamma,
  \end{align*}
  where the second line follows by a similar argument.  Also,
  \begin{align*}
    \val(a'/u) = \val(a') - \val(u) & > 3 \gamma - \gamma = 2 \gamma \\
    \val(b'/u) = \val(b') - \val(u) & > 3 \gamma - \gamma = 2 \gamma.
  \end{align*}
  Thus, by Lemma~\ref{small-products}, $(a'/u)(b'/u) \in R$.  Then
  \begin{equation*}
    a'b' = (a'/u)(b'/u)(u^2) \in R \cdot u^2 \subseteq I \subseteq Q,
  \end{equation*}
  so we see that
  \begin{equation*}
    \partial(a'b') = 0.
  \end{equation*}
  Also,
  \begin{equation*}
    \val(0 - a') + \val_\partial(b') = \val(a') + \val_\partial(b) > 3 \gamma - \gamma > 0,
  \end{equation*}
  so $a' \partial b' = \partial(0 \cdot b') = 0$.  Similarly, $b'
  \partial a' = 0$.  So
  \begin{equation*}
    \partial(a'b') = 0 = a'\partial b' + b'\partial a'.
  \end{equation*}
  The other three equations
  \begin{align*}
    \partial(a_Qb') &= a_Q \partial b' + b' \partial a_Q \\
    \partial(a'b_Q) &= a' \partial b_Q + b_Q \partial a' \\
    \partial(a_Q b_Q) &= a_Q \partial b_Q + b_Q \partial a_Q
  \end{align*}
  hold by $Q$-linearity and the fact that $\partial$ vanishes on $Q$.
  Adding these four equations, we obtain the desired Leibniz rule
  \begin{equation*}
    \partial(ab) = a \partial b + b \partial a. \qedhere
  \end{equation*}
\end{proof}
Note that
\begin{equation*}
  Q = \{a \in \Oo : \partial a = 0\}.
\end{equation*}
Let $\Gamma$ be the value group of $\Oo$.  We define
\begin{equation*}
  \val : D \to \Gamma_{\le 0} \cup \{+\infty\}
\end{equation*}
by the equation
\begin{equation*}
  \val(\partial a) = \val_\partial(a).
\end{equation*}
By Lemmas~\ref{vp-class} and \ref{vp-ops}, this is well-defined, and
satisfies the identities
\begin{align*}
  \val(a + b) &\ge \min(\val(a),\val(b)) \\
  \val(a) = +\infty &\iff a = 0
\end{align*}
for $a, b \in D$.
\begin{lemma}\label{mult-scaling}
  For $a \in \Oo$ and $b \in D$,
  \begin{equation*}
    \val(ab) = \val(a) + \val(b),
  \end{equation*}
  unless the right hand side is positive, in which case
  \begin{equation*}
    \val(ab) = +\infty.
  \end{equation*}
\end{lemma}
\begin{proof}
  Write $b$ as $\partial c$ for some $c \in \Oo$.  Take $a' \in Q$
  such that $\val(a - a') + \val(\partial c) > 0$ and $\val(a - a') >
  \val(a)$.  Then by definition, $a \partial c = \partial(a'c)$.  Also
  $\val(a) = \val(a')$.  By Lemma~\ref{vp-scalars},
  \begin{equation*}
    \val(a \partial c) = \val(\partial(a'c)) = \val_\partial(a'c) = \val(a') +
    \val_\partial(c) = \val(a) + \val(\partial c),
  \end{equation*}
  unless the right hand side is positive, in which case $\val(ab) =
  +\infty$.
\end{proof}
Using this and the fact that $\partial$ is a derivation, we get an
improved version of the multiplication statement in Lemma~\ref{vp-ops}.
\begin{corollary}\label{vp-mult-better}
  If $a, b \in \Oo$, then
  \begin{equation*}
    \val_\partial(ab) \ge \min(\val(a) + \val_\partial(b), \val_\partial(a) + \val(b)).
  \end{equation*}
\end{corollary}

\subsection{The module of differentials}
The $\Oo$-module $D$ of differentials shares many properties with
$K/\mm$.
\begin{lemma}\label{all-gammas}
  For any $\gamma \le 0$ in the value group, there is $b \in D$ such
  that $\val(b) = \gamma$.
\end{lemma}
\begin{proof}
  By Lemma~\ref{mult-scaling}, it suffices to show that
  \[ \{\val(b) : b \in D\}\]
  has no lower bound, which follows by Lemma~\ref{weird-elements}.
\end{proof}

\begin{lemma}\label{divisibility-sort-of}
  If $a, b \in D$ and $\val(a) < \val(b)$, then $b \in \Oo
  \cdot a$.
\end{lemma}
\begin{proof}
  We may assume $b \ne 0$, in which case $\val(a) < \val(b) \le 0$.
  By Proposition~\ref{density}, $\Oo = \mm + Q$.  Therefore, we may
  write $a = \partial a'$ and $b = \partial b'$ for some $a', b' \in
  \mm$.  By Lemma~\ref{application}, we can write
  \begin{equation*}
    b' = p + qa'
  \end{equation*}
  for some $p, q \in Q$.  Then $\partial p = \partial q = 0$, so
  \begin{equation*}
    b = \partial b' = q \partial a' = q a. \qedhere
  \end{equation*}
\end{proof}
Of course, we can replace $\val(a) < \val(b)$ with a non-strict
inequality:
\begin{proposition}\label{divisibility-v2}
  If $a, b \in D$ and $\val(a) \le \val(b)$, then $b \in \Oo \cdot a$.
\end{proposition}
\begin{proof}
  We may assume $b, a \ne 0$.
  By Lemma~\ref{all-gammas}, there is $z \in D$ with $\val(z) <
  \val(a)$.  Then $a = \alpha z$ and $b = \beta z$ for some $\alpha,
  \beta \in \Oo$.  By Lemma~\ref{mult-scaling}
  \begin{align*}
    \val(\alpha) &= \val(a) - \val(z) \\
    \val(\beta) &= \val(b) - \val(z).
  \end{align*}
  Then $\val(\alpha) \le \val(\beta)$, and so $\beta \in \Oo \alpha$
  as $\Oo$ is a valuation ring.  Therefore
  \[ b = \beta z = \gamma \alpha z = \gamma a \in \Oo a\]
  for some $\gamma \in \Oo$.
\end{proof}
\begin{proposition}\label{divisibility-4real}
  $D$ is divisible as an $\Oo$-module: for any $b \in D$ and non-zero
  $a \in \Oo$, there is $x \in D$ such that $ax = b$.
\end{proposition}
\begin{proof}
  We may assume $b \ne 0$.  By Lemma~\ref{all-gammas}, there is $c \in
  D$ such that $\val(c) \le \val(b) - \val(a)$.  Then $\val(ac) \le
  \val(b)$, and so
  \begin{equation*}
    b \in \Oo \cdot ac \subseteq a \cdot D. \qedhere
  \end{equation*}
\end{proof}
\begin{proposition}\label{res-connection-v2}
  Let $D_0$ be the $\Oo$-submodule of $x \in D$ such that $\val(x) \ge
  0$.
  \begin{itemize}
  \item $D_0$ is the image of $R$ under $\partial$.
  \item Viewing $k$ as the $\Oo$-module $\Oo/\mm$, there is a unique
    $\Oo$-module isomorphism $\res_2 : D_0 \to k$ such that
    \begin{equation*}
      \wres(x) = \res(x) + \res_2(\partial x) \varepsilon
    \end{equation*}
    for any $x \in R$.
  \end{itemize}
\end{proposition}
\begin{proof}
  The first point is clear from Lemma~\ref{vp-class}:
  \[ \val(\partial x) \ge 0 \iff \val_\partial(x) \ge 0 \iff x \in R.\]
  Take some $w \in R$ such that $\wres(w) = 0 + \varepsilon$.  Then $w
  \notin Q$, so
  \[ \val(\partial w) = \val_\partial(w) = 0\]
  by Lemma~\ref{vp-class}.  By Proposition~\ref{divisibility-v2}, $\partial
  w$ generates $D_0$ as an $\Oo$-module.  Also, $\Ann_\Oo(\partial w)$
  is $\mm$ by Lemma~\ref{mult-scaling}.  Thus there is an isomorphism
  \begin{align*}
    \res_2 : D_0 & \to \Oo/\mm \\
    y \partial w & \mapsto \res(y).
  \end{align*}
  Now let $x \in R$ be given.  Then
  \begin{equation*}
    \wres(x) = s + t \varepsilon
  \end{equation*}
  for some $s, t \in k$.  We already know that $s = \res(k)$, and we
  must show that
  \[ \res_2(\partial x) = t.\]
  Take $y \in R$ with $\wres(y) = t + 0\varepsilon$.  Then $y \in Q$.  Also,
  \begin{equation*}
    \wres(x - wy) = \wres(x) - \wres(w)\wres(y) = s + t \varepsilon -
    \varepsilon t = s,
  \end{equation*}
  so $x - wy \in Q$.  Then
  \begin{equation*}
    \partial x = \partial (w y) = y \partial w,
  \end{equation*}
  and so
  \begin{equation*}
    \res_2(\partial x) = \res_2(y \partial w) = \res(y) = t.
  \end{equation*}
  This proves the formula
  \begin{equation*}
    \wres(x) = \res(x) + \res_2(\partial x) \varepsilon.
  \end{equation*}
  Finally, this formula uniquely determines $\res_2$, because
  $\partial : R \to D_0$ is onto.
\end{proof}

\subsection{Odd positive characteristic}
\begin{proposition} \label{prop:perfect}
  If $K$ is perfect, then $\characteristic(K) = 0$.
\end{proposition}
\begin{proof}
  Suppose for the sake of contradiction that $\characteristic(K) = p >
  2$.  By construction, the derivation $\partial : \Oo \to D$ is onto.
  Also, $D$ cannot vanish, since we have constructed a submodule $D_0$
  isomorphic to $k$.  Therefore $\partial a \ne 0$ for some $a \in
  \Oo$.  By perfection of $K$, we can write $a = b^p$.  Then
  \[ \partial a = \partial (b^p) = p b^{p-1} \partial b = 0,\]
  a contradiction.
\end{proof}

\section{Application to fields of dp-rank 2} \label{sec:app}
Recall that a topology on a structure is \emph{definable} if it admits
a uniformly definable basis of open sets.
\begin{theorem} \label{thm:app}
  Let $(K,+,\cdot,0,1,\ldots)$ be a field of characteristic 0,
  possibly with extra structure.  Suppose $K$ has dp-rank 2 and is
  unstable.
  \begin{enumerate}
  \item \label{fact-1} $K$ does not admit two independent definable
    valuation rings.
  \item \label{fact-2} $K$ admits a definable non-trivial
    V-topology.
  \item \label{fact-3} The canonical topology on $K$ is definable.
    (See \S\ref{sec:sofar}).
  \end{enumerate}
\end{theorem}
We prove these statements in
\S\ref{sec:val-case}--\ref{sec:cantopdef}, but for now, we give some
motivation.
\begin{proposition}
  Under the assumptions of Theorem~\ref{thm:app}, $K$ admits a
  unique definable non-trivial V-topology.
\end{proposition}
\begin{proof}
  Existence follows from Theorem~\ref{thm:app}.\ref{fact-2}.  For
  uniqueness, suppose $K$ admits two independent definable
  V-topologies.  We may replace $K$ with an $\aleph_0$-saturated
  elementary extension.  Then the two V-topologies are induced by
  externally definable valuation rings $\Oo_1, \Oo_2$, by
  Proposition~3.5 in \cite{hhj-v-top}.  Replacing $K$ with its
  Shelah expansion $K^{\textrm{Sh}}$, we obtain two independent
  valuation rings.  The Shelah expansion continues to have dp-rank
  2---this is a simple exercise using quantifier elimination in the
  Shelah expansion (\cite{NIPguide}, Proposition~3.23).
\end{proof}
\begin{proposition}\label{prop-compare}
  Let $(K,\Oo_1,\Oo_2,\ldots)$ be a field with two definable
  valuation rings $\Oo_1, \Oo_2$, and possibly additional structure.
  If $\dpr(K) \le 2$, then $\Oo_1$ and $\Oo_2$ are comparable.
\end{proposition}
\begin{proof}
  Suppose $\Oo_1, \Oo_2$ are incomparable.  The join $\Oo_1 \cdot
  \Oo_2$ is itself a valuation ring.  Let $K'$ be the residue field
  of $\Oo_1 \cdot \Oo_2$.  Then $\Oo_1$ and $\Oo_2$ induce two
  independent valuation rings $\Oo_1'$ and $\Oo_2'$ on $K'$.
  Indeed, there is an isomorphism between
  \begin{itemize}
  \item The poset of valuation rings on $K'$.
  \item The poset of valuation rings on $K$ that are contained in
    $\Oo_1 \cdot \Oo_2$.
  \end{itemize}
  Thus $\Oo_1'$ and $\Oo_2'$ are incomparable, and $\Oo_1' \cdot
  \Oo_2'$ must be the maximal valuation ring on $K'$, which is
  $K'$ itself.  Thus $\Oo_1'$ and $\Oo_2'$ are incomparable and
  independent.

  Replacing $(K,\Oo_1,\Oo_2)$ with $(K',\Oo_1',\Oo_2')$, we may
  assume that $\Oo_1$ and $\Oo_2$ are independent.  Then we get a contradiction:
  \begin{itemize}
  \item If $\dpr(K) \le 1$, use Lemma~9.4.14 in \cite{myself}.
  \item If $\characteristic(K) > 0$, use Lemma~2.6 in \cite{prdf}.  
  \item If $\dpr(K) = 2$ and $\characteristic(K) = 0$, then
    Theorem~\ref{thm:app} applies. \qedhere
  \end{itemize}
\end{proof}
\begin{proposition}
  Suppose parts \ref{fact-1} and \ref{fact-2} of Theorem~\ref{thm:app}
  hold for all ranks.  In other words, suppose the following hold:
  \begin{itemize}
  \item No dp-finite field of characteristic 0 admits two independent
    definable valuation rings.
  \item Every unstable dp-finite field of characteristic 0 admits a
    (non-trivial) definable V-topology.
  \end{itemize}
  Then the Shelah conjecture holds for dp-finite fields: every
  dp-finite field is either finite, algebraically closed, real closed,
  or henselian.  Therefore, the conjectured classification of
  (\cite{halevi-hasson-jahnke}, Theorem~3.11) holds.
\end{proposition}
\begin{proof}
  As in Proposition~\ref{prop-compare}, we conclude that any two
  definable valuation rings $\Oo_1, \Oo_2$ on a dp-finite field are
  comparable.  As in the proof of Theorem~2.8 in \cite{prdf}, this
  implies the henselianity conjecture for dp-finite fields: every
  definable valuation ring on a dp-finite field is henselian.

  Because the Shelah expansion of a dp-finite structure is dp-finite,
  it follows that any externally definable valuation ring must also be
  henselian.
  \begin{claim}\label{cl0}
    If $K$ is dp-finite of characteristic 0, then one of the following holds:
    \begin{itemize}
    \item $K$ is algebraically closed.
    \item $K$ is real closed.
    \item $K$ is finite
    \item $K$ admits a definable non-trivial valuation ring.
    \end{itemize}
  \end{claim}
  \begin{claimproof}
    We may replace $K$ with a sufficiently saturated elementary
    extension.  If $K$ is stable, then $K$ is finite or algebraically
    closed, by Proposition~7.2 in \cite{Palacin}.  Otherwise, $K$
    admits a non-trivial definable V-topology, by assumption.  By
    Proposition~3.5 in \cite{hhj-v-top}, $K$ admits a non-trivial
    externally definable valuation ring $\Oo$.  This valuation ring
    must be henselian.  By Theorem~5.2 in \cite{JK}, either $K$ admits
    a definable valuation ring or $K$ is real closed or algebraically
    closed.\footnote{If $K$ is separably closed, then $K$ is
      algebraically closed, because dp-finite fields are perfect.}
  \end{claimproof}
  This in turn implies the Shelah conjecture for dp-finite fields of
  characteristic 0.  The case of positive characteristic is
  Corollary~11.4 \cite{prdf}.

  The classification in (\cite{halevi-hasson-jahnke}, Theorem~3.11) is
  proven conditional on the Shelah conjecture.  (The proof is for
  strongly dependent fields, but can be restricted to the smaller
  class of dp-finite fields.)
\end{proof}

\subsection{The pedestal machine} \label{sec:review}
We review the setup from (\cite{prdf3}, Part II).  Let $\Kk$ be an unstable monster field, possibly with extra structure, with $\dpr(\Kk) \le 2$.

Fix a \emph{magic subfield} $k_0 \preceq \Kk$, i.e., a small model
with the following property (Definition 8.3 in \cite{prdf2}):
\begin{quote}
  For every $k_0$-linear subspace $G \le (\Kk,+)$, if $G$ is
  type-definable (over any small set), then $G = G^{00}$.
\end{quote}
Magic subfields exist by (\cite{prdf}, Corollary~8.7).

Let $\Lambda$ denote the lattice of type-definable $k_0$-linear
subspaces of $\Kk$.  Recall from (\cite{prdf}, Definition~9.13) that a
\emph{strict $n$-cube} in $\Lambda$ is an injection
\begin{align*}
  Pow(n) & \hookrightarrow \Lambda \\
  S & \mapsto G_S
\end{align*}
that preserves the unbounded lattice operations:
\begin{align*}
  G_{S_1 \cup S_2} = G_{S_1} + G_{S_2} \\
  G_{S_1 \cap S_2} = G_{S_1} \cap G_{S_2}.
\end{align*}
We call $G_\emptyset$ the \emph{base} of the cube; the base need not
be 0.

The \emph{reduced rank} of $\Lambda$ is the maximum $r$ such that a
strict $r$-cube exists (Definition~9.17 in \cite{prdf}).  By
Proposition~10.1.7 in \cite{prdf}, the reduced rank is 1 or 2.

If $r$ is the reduced rank, a \emph{pedestal} in $\Lambda$ is a group
$G \in \Lambda$ that is the base of a strict $r$-cube (Definition~8.4
in \cite{prdf2}\footnote{Pedestals were called ``special groups'' in
  \S 10 of \cite{prdf}}).  Since $r$ is small, we can describe what
this means explicitly:
\begin{itemize}
\item If $r = 1$, then an $r$-cube is a chain of length two, and a
  pedestal is any $G \in \Lambda$ other than $\Kk$ itself.
\item If $r = 2$, then an $r$-cube is
  \[ \{G \cap H, G, H, G + H \} \]
  for two incomparable $G, H \in \Lambda$.  Therefore, a pedestal is a
  group of the form $G \cap H$ where $G, H$ are incomparable elements
  of $\Lambda$.
\end{itemize}
\begin{fact}[Proposition~10.4.1 in \cite{prdf}] \label{nze}
  Non-zero pedestals exist.
\end{fact}
In Theorem~9.3 of \cite{prdf3}, we associated an $r$-inflator to any
non-zero pedestal $H$.
\begin{fact}\label{fact-summary}
  Let $H$ be a non-zero pedestal with associated $r$-inflator
  $\varsigma$.
  \begin{enumerate}
  \item \label{fs-mal} $\varsigma$ is malleable.
  \item The fundamental ring $R_H$ of $\varsigma$ is given as
    \begin{equation*}
      R_H = \{ x \in \Kk : xH \subseteq H\}.
    \end{equation*}
  \item If $H$ is type-definable over a small model $K$ containing
    $k_0$, then the infinitesimals $J_K$ are contained in the
    fundamental ideal $I_H$.
  \item \label{fs-ideal} If $H$ is type-definable over a small model
    $K$ containing $k_0$, then $R_H \cdot J_K \subseteq J_K$, and so
    $J_K$ is a sub-ideal of the fundamental ideal $I_H$.
  \item \label{fs-mut} If $\varsigma'$ is obtained by mutating
    $\varsigma$ along a line $\Kk \cdot (a_1, \ldots, a_n)$, then
    $\varsigma'$ is the $r$-inflator associated to the group
    \begin{equation*}
      H' =  (a_1^{-1} H) \cap \cdots \cap (a_n^{-1} H)
    \end{equation*}
    In particular, $H'$ is itself a non-zero pedestal.
  \end{enumerate}
\end{fact}
This follows from (\cite{prdf3}, Theorem~9.3, Remark~9.5,
Proposition~10.15), (\cite{prdf}, Proposition~10.15.5, Lemma~10.20),
and Lemma~\ref{ideal-lemma} below.

\begin{remark}\label{0-ind-def}
  By construction (\cite{prdf}, Theorem~4.20.4, Definition~6.3), the
  family of basic neighborhoods is uniformly ind-definable across all
  models.  In other words, there is a set of formulas
  $\{\psi_i(x;\vec{z}_i)\}_{i \in I}$ such that for any model $K$, the
  collection of basic neighborhoods on $K$ is exactly
  \[ \{ \psi_i(K;\vec{c}) : i \in I, ~ \vec{c} \in K^{|\vec{z}_i|} \}.\]
\end{remark}

\begin{lemma}\label{ideal-lemma}
  Let $G$ be a non-zero pedestal, type-definable over a small model $K
  \preceq \Kk$, with $K$ extending $k_0$.  Let $R$ be the stabilizer
  ring of $G$:
  \[ R = \{x \in \Kk : x G \subseteq G\}.\]
  Let $J_K$ be the group of $K$-infinitesimals.  Then $J_K$ is an
  ideal in $R$.
\end{lemma}
The following proof was sketched in Remark~10.19 of \cite{prdf}.
\begin{proof}
  First of all, $J_K \subseteq R$ by Proposition~10.15.(2,5) in
  \cite{prdf}.  It remains to show that $R \cdot J_K \subseteq J_K$.
  Take a non-zero element $j_0 \in G$.  Take a small model $K' \preceq
  \Kk$ with $K' \supseteq K \cup \{j_0\}$.  As $G$ is type-definable
  over the larger model $K'$, we see that
  \[ J_{K'} \cdot G \subseteq J_{K'}\]
  by (\cite{prdf}, Proposition~10.4.3).  Now for any $\varepsilon \in
  J_{K'}$ and $a \in R$, we have
  \begin{equation*}
    \varepsilon \cdot a \cdot j_0 \in J_{K'} \cdot R \cdot G \subseteq
    J_{K'} \cdot G \subseteq J_{K'},
  \end{equation*}
  implying that $\varepsilon \cdot a \in j_0^{-1} J_{K'}$.  As $J_{K'}$
  is invariant under scaling by elements of $(K')^\times$
  (\cite{prdf}, Remark~6.9.3), we see that $j_0^{-1} J_{K'} = J_{K'}$,
  and
  \[ \varepsilon \cdot a \in J_{K'}.\]
  As $a \in R$ and $\varepsilon \in J_{K'}$ were arbitrary,
  \begin{equation}
    R \cdot J_{K'} \subseteq J_{K'}. \label{eq-foo}
  \end{equation}
  \begin{claim}
    If $S \subseteq R$ is type-definable over $K$, and $U$ is a
    $K$-definable basic neighborhood, then there is a $K$-definable
    basic neighborhood $V$ such that
    \begin{equation*}
      S \cdot V \subseteq U.
    \end{equation*}
  \end{claim}
  \begin{claimproof}
    Since $K' \supseteq K$, the neighborhood $U$ is $K'$-definable and
    contains $J_{K'}$.  Therefore,
    \[ S \cdot J_{K'} \subseteq R \cdot J_{K'} \subseteq J_{K'} \subseteq U,\]
    by (\ref{eq-foo}).  By compactness, there is a $K$-definable set
    $S' \supseteq S$, and a $K'$-definable basic neighborhood $V'
    \supseteq J_{K'}$ such that
    \[ S' \cdot V' \subseteq U.\]
    We can write $V'$ as $\psi_i(\Kk;\vec{b})$ for one of the formulas
    $\psi_i$ in Remark~\ref{0-ind-def}.  Since $S'$ and $U$ are
    $K$-definable, we can find $\vec{c}$ from $K$ such that
    \[ S' \cdot \psi_i(\Kk;\vec{c}) \subseteq U.\]
    Take $V = \psi_i(\Kk;\vec{c})$.  Then
    \[ S \cdot V \subseteq S' \cdot V \subseteq U. \qedhere\]
  \end{claimproof}
  Now by compactness, it follows that for any subset $S \subseteq R$
  that is type-definable over $K$, we have
  \[ S \cdot J_K \subseteq J_K.\]
  As the ring $R$ is $K$-invariant, it is a union of such subsets $S$,
  and therefore
  \[ R \cdot J_K \subseteq J_K,\]
  as desired.
\end{proof}

\subsection{The valuation-type case} \label{sec:val-case}

In \cite{prdf2}, we considered the case where the canonical topology
is a V-topology.  We say that $\Kk$ is \emph{valuation type} if this
holds.  We showed in this case that
\begin{itemize}
\item The canonical topology is a definable V-topology (\cite{prdf2},
  Lemma~7.1).
\item Any two definable valuation rings are dependent (\cite{prdf2},
  Lemma~9.5).
\end{itemize}
Thus, the three parts of Theorem~\ref{thm:app} are automatic in this
case.

\begin{fact}[Theorem~8.11 in \cite{prdf2}]
  If $K$ is a small submodel and if $J_K$ contains a non-zero ideal of
  some multi-valuation ring on $\Kk$, then the canonical topology on
  $\Kk$ is a V-topology.
\end{fact}
This has several consequences:
\begin{corollary}\label{joke-cor}
  Let $G$ be a non-zero pedestal with stabilizer $R$ and associated
  $r$-inflator $\varsigma$.
  \begin{enumerate}
  \item If $R$ contains a non-zero ideal of a multi-valuation ring,
    then $\Kk$ is valuation type.
  \item If $\varsigma$ is weakly multi-valuation type, then $\Kk$ is
    valuation type.
  \item\label{jc3} If $r = 1$, then $\Kk$ is valuation type.
  \item\label{jc4} If some mutation of $\varsigma$ is weakly
    multi-valuation type, then $\Kk$ is valuation type.
  \end{enumerate}
\end{corollary}
\begin{proof}
  \begin{enumerate}
  \item Lemma~\ref{ideal-lemma}---the point is that if $R'$ is a
    multi-valuation ring, and
    \[ a_1 R' \subseteq R,\]
    then $a_2 a_1 R' \subseteq a_2 R \subseteq J_K$ for any non-zero $a_2 \in J_K$.
  \item $R$ is the fundamental ring of $\varsigma$, and ``weakly
    multi-valuation type'' means that the fundamental ring contains a
    non-zero multi-valuation ideal (\cite{prdf3}, Definition~5.27).
  \item 1-inflators are multi-valuation type (Proposition~5.19 in
    \cite{prdf3}).
  \item If $\varsigma'$ is obtained from $\varsigma$ by mutation, then
    $\varsigma'$ is the $r$-inflator associated to some other non-zero
    pedestal $G'$ (Fact~\ref{fact-summary}.\ref{fs-mut}).  In
    particular, if $\varsigma'$ is weakly of multi-valuation type,
    then $G'$ shows that $\Kk$ is valuation type. \qedhere
  \end{enumerate}  
\end{proof}
\begin{theorem} \label{thm:what}
  If $\Kk$ is not valuation type (and characteristic 0 and unstable),
  then there is a small model $K$ and a 2-inflator $\varsigma$
  satisfying the Strong Assumptions of
  \S\ref{sec:canform}--\ref{sec:der}, such that the infinitesimals
  $J_K$ are an ideal in the fundamental ring $R$ of $\varsigma$.
\end{theorem}
\begin{proof}
  By Fact~\ref{nze}, non-zero pedestals exist.  Let $G$ be some
  non-zero pedestal and $\varsigma$ be the associated inflator.  Then
  $G$ satisfies the Weak Assumptions:
  \begin{itemize}
  \item $K$ has characteristic 0 by assumption.
  \item $\varsigma$ is malleable by
    Fact~\ref{fact-summary}.\ref{fs-mal}.
  \item $r = 2$ by Corollary~\ref{joke-cor}.\ref{jc3}.
  \item No mutation of $\varsigma$ is weakly multi-valuation type, by
    Corollary~\ref{joke-cor}.\ref{jc4}.
  \end{itemize}
  By Corollary~\ref{cor:rediso}, there is a mutation $\varsigma'$ of
  $\varsigma$ such that $\varsigma'$ is isotypic.  Then $\varsigma'$
  inherits the other properties from $\varsigma$ (see
  Remark~\ref{rem:inherit}), and therefore $\varsigma'$ satisfies the
  Strong Assumptions.  By Fact~\ref{fact-summary}.\ref{fs-mut},
  $\varsigma'$ is the 2-inflator coming from some other pedestal $G'$.
  Let $K$ be a small model containing $k_0$, and type-defining $G'$.
  By Fact~\ref{fact-summary}.\ref{fs-ideal}, $J_K$ is an ideal in the
  fundamental ring of $\varsigma'$.
\end{proof}

In the remainder of \S\ref{sec:app}, we therefore assume
\begin{enumerate}
\item $\Kk$ is a monster model of an unstable field of dp-rank 2 and
  characteristic 0.
\item $k_0$ is a magic subfield.
\item $\varsigma$ is a $k_0$-linear 2-inflator on $\Kk$ satisfying the
  Strong Assumptions of \S\ref{sec:canform}-\ref{sec:der}, including
  isotypy.
\item $R$ and $I$ are the fundamental ring and ideal of $\varsigma$,
  and $D$ and $\partial$ are as in \S\ref{ssec:der}.
\item $J$ is the group of $K$-infinitesimals over some small model $K
  \preceq \Kk$ containing $k_0$.  In particular,
  \begin{itemize}
  \item $J$ is type-definable
  \item $J$ is contained in every $K$-definable basic neighborhood.
  \item $J$ is non-zero (\cite{prdf}, Remark~6.9.1).
  \end{itemize}
\item $J$ is an ideal in $R$, contained in the fundamental ideal $I$.
\end{enumerate}

\subsection{Independent valuation rings} \label{sec:newhome}
Recall that the valuation ring $\Oo$ is the integral closure of $R$.
\begin{proposition}\label{prop:contra}
  Let $\Oo'$ be a valuation ring on $\Kk$, independent from $\Oo$.
  Then $\Oo' \not\supseteq J$.
\end{proposition}
\begin{proof}
  Assume for the sake of contradiction that $\Oo' \supseteq J$.  Take
  nonzero $e \in J$.  Then $R \cdot e \subseteq J$, so
  \begin{equation*}
    R \subseteq e^{-1} \cdot J \subseteq e^{-1} \Oo'.
  \end{equation*}
  Let $\val'$ be the valuation from $\Oo'$ and let $\gamma = \val'(e^{-1})$.  Then
  \begin{equation*}
    x \in R \implies \val'(x) \ge \val'(e^{-1}) = \gamma.
  \end{equation*}
  We claim that for all $a \in K$,
  \begin{equation*}
    x \in \mm \implies \val'(x) \ge \min(\gamma,\gamma/2)
  \end{equation*}
  Indeed, suppose $x \in \mm$.  Then $x^{-1}$ isn't integral over
  $R$, so by Lemma~\ref{tilde-o-very-valuation},
  \begin{equation*}
    x^2 + bx + c = 0
  \end{equation*}
  for some $b, c \in R$.  By Newton polygons,
  \begin{equation*}
    \val'(x) \ge \min\left(\val'(b),\frac{\val'{c}}{2}\right) \ge \min(\gamma,\gamma/2).
  \end{equation*}
  On the other hand, $\Oo$ is independent from $\Oo'$, so by
  the approximation theorem, there is $x \in K$ with $\val(x) > 0$ and
  $\val'(x) < \min(\gamma,\gamma/2)$, a contradiction.
\end{proof}

\begin{corollary}
  If $\Oo_1, \Oo_2$ are two 0-definable valuation rings on $\Kk$, then
  $\Oo_1$ and $\Oo_2$ are not independent.
\end{corollary}
\begin{proof}
  The definable set $\Oo_i$ has full dp-rank for $i = 1, 2$.  It follows that
  $\Oo_i - \Oo_i$ is a 0-definable basic neighborhood, and so
  \begin{equation*}
    J \subseteq \Oo_i - \Oo_i = \Oo_i
  \end{equation*}
  for $i = 1, 2$.  By Proposition~\ref{prop:contra}, both $\Oo_1$ and
  $\Oo_2$ induce the same topology as the valuation ring $\Oo$---the
  integral closure of $R$.
\end{proof}

\subsection{The definable V-topology}
Let $\val : \Kk \to \Gamma$ be the valuation associated to $\Oo$.  We
will show that the associated valuation topology is definable.
\begin{lemma}
  \label{tdef2}
  There is a type-definable set $B \subseteq \Kk$ and some $\gamma \in
  \Gamma$ such that
  \begin{equation*}
    \val(x) > \gamma \implies x \in B \implies \val(x) \ge 0.
  \end{equation*}
  for $x \in \Kk$.
\end{lemma}
\begin{proof}
  Let $B$ be the set
  \[ B = \{x \in \Kk ~|~ \exists y, z \in J : x^2 = yx + z\}.\]
  Then $B$ is type-definable.  If $x \in B$, then
  \[ x^2 = yx + z\]
  for some $y, z \in J \subseteq R$, and so $x$ lies in the integral
  closure $\Oo$ of $R$.

  Now take non-zero $c \in J$, and let $\gamma = \val(c)$.  Note $c
  \in R \subseteq \Oo$, so $\gamma \ge 0$.  Suppose $\val(x) >
  \gamma$.  Then $\val(x/c) > 0$, so $x/c \in \mm$ and $c/x \notin
  \Oo$.  By Lemma~\ref{tilde-o-very-valuation}, either $x/c$ or $c/x$
  satisfies a monic polynomial equation of degree 2 over $R$.  As
  $c/x$ is not in the integral closure $\Oo$ of $R$, we see that $x/c$
  satisfies the equation:
  \[ (x/c)^2 = (x/c)y_0 + z_0\]
  for some $y_0, z_0 \in R$.  Then
  \[ x^2 = (cy_0)x + c^2z_0,\]
  and $cy_0, c^2z_0 \in J$.  Thus $x \in B$.
\end{proof}
Say that two subsets $X, Y \subseteq \Kk$ are ``co-embeddable'' if
there exist $a, b \in \Kk^\times$ such that
\begin{align*}
  a \cdot X &\subseteq Y \\
  b \cdot Y &\subseteq X.
\end{align*}
This is an equivalence relation.
\begin{remark} \label{interpolate}
  Suppose $X$ and $Y$ are co-embeddable, $X$ is type-definable, and
  $Y$ is $\vee$-definable.  Then there is a definable set $Z$
  co-embeddable with $X$ and $Y$.  Indeed, after rescaling, we may
  assume
  \begin{equation*}
    X \subseteq Y.
  \end{equation*}
  Then we may find a definable set $Z$ interpolating $X$ and $Y$, by
  compactness:
  \begin{equation*}
    X \subseteq Z \subseteq Y.
  \end{equation*}
\end{remark}
\begin{lemma}
  \label{lem:def2}
  There is a definable set $B$ that is co-embeddable with $\Oo$.
\end{lemma}
\begin{proof}
  By Lemma~\ref{tdef2}, there is a type-definable set $B_0$ and
  $\gamma \in \Gamma$ such that
  \begin{equation*}
    \val(x) > \gamma \implies x \in B_0 \implies \val(x) \ge 0.
  \end{equation*}
  Therefore $B_0$ is co-embeddable with $\Oo$.  Let $B_1$ be the
  $\vee$-definable set $B_1 = \{0\} \cup \{y \in \Kk^\times : y^{-1}
  \notin B_0\}$.  Note that
  \begin{equation*}
    \val(y) > 0 \implies y \in B_1 \implies \val(y) \ge -\gamma.
  \end{equation*}
  Thus $B_1$ is co-embeddable with $\Oo$.  By
  Remark~\ref{interpolate}, there is a definable set in the
  co-embeddability class of $B_0, B_1,$ and $\Oo$.
\end{proof}
Recall from Lemma~2.1(d) in \cite{prestel-ziegler}, that a set $S$ in
a topological field $K$ is \emph{bounded} if and only if for every
open neighborhood $U \ni 0$, there is non-zero $a \in K^\times$ such
that
\begin{equation*}
  a \cdot S \subseteq U.
\end{equation*}
\begin{theorem}\label{th:that-vee}
  The V-topology induced by $\Oo$ is definable.
\end{theorem}
\begin{proof}
  Take a definable set $B$ that is co-embeddable with $\Oo$.  Then $B$
  is a bounded neighborhood of 0, with respect to the V-topology
  induced by $\Oo$.  Therefore, the following definable family is a
  neighborhood basis of 0, by Lemma 2.1(e) in \cite{prestel-ziegler}:
  \begin{equation*}
    \{a B : a \in \Kk^\times\}.
  \end{equation*}
  This proves definability, by Lemma~\ref{duh} below.
\end{proof}
\begin{lemma}\label{duh}
  Let $(K,+,\cdot,\ldots)$ be a field, possibly with extra structure.
  Let $\tau$ be a field topology on $K$.  Then $\tau$ is definable if
  and only if there is a definable neighborhood basis of 0.
\end{lemma}
\begin{proof}
  If $\{U_a\}_{a \in Y}$ is a definable basis of opens, then
  \begin{equation*}
    \{ U_a : a \in Y \text{ and } 0 \in U_a\}
  \end{equation*}
  is a definable neighborhood basis of 0.  Conversely, suppose
  $\{N_a\}_{a \in Y}$ is a definable neighborhood basis of 0.  Let
  \begin{equation*}
    N_a^{\mathrm{int}} = \{x \in K ~|~ \exists b \in Y : x + N_b \subseteq N_a\}.
  \end{equation*}
  Then $\{N_a^{\mathrm{int}}\}_{a \in Y}$ is a definable basis of open
  neighborhoods around 0, and
  \begin{equation*}
    \{b + N_a^{\mathrm{int}} : b \in K, ~ a \in Y\}.
  \end{equation*}
  is a definable basis of open sets.
\end{proof}

\subsection{Definability of the canonical topology} \label{sec:cantopdef}
\begin{lemma}
  \label{exotic2}
  There is nonnegative $\gamma \in \Gamma$, and a type-definable set
  $S$, such that for $x \in \Kk$ with $\val(x) > \gamma$, we have
  \begin{equation*}
    \val(\partial x) < - \gamma \implies x \in S \implies
    \val(\partial x) \le 0.
  \end{equation*}
\end{lemma}
\begin{proof}
  Take some non-zero $c \in J$.  Take $\gamma = \val(c)$.  Then
  $\gamma \ge 0$, as
  \[ J \subseteq R \subseteq \Oo.\]
  Take $e_0 \in R$ such that $\wres(e_0) = s + t \varepsilon$, with $t
  \ne 0$.  Then $e_0 \in R \setminus Q$.  Let $B$ be the open ball of
  valuative radius $\gamma$.  By Proposition~\ref{density}, $\Oo = B +
  Q$.  Therefore there is $e \in B$ with $e - e_0 \in Q$.  Then $e \in
  R \setminus Q$, and $\val(e) > \gamma = \val(c)$.

  Let $S$ be the type-definable set of $x \in \Kk$ such that
  \[ \exists y, z \in J : e = xy + z.\]
  Suppose $\val(x)
  > \gamma$ and $\val(\partial x) < - \gamma$.  Apply
  Lemma~\ref{three-generators} to the set $\{xc, e, c\}$.  There are
  three cases:
  \begin{itemize}
  \item $c$ is generated by $xc$ and $e$.  This cannot happen, since
    $R \subseteq \Oo$, since $\val(c) < \val(e)$ (by choice of $e$),
    and since $\val(c) < \val(xc)$ as $\val(x) > \gamma \ge 0$.
  \item $xc$ is generated by $e$ and $c$.  As $e, c \in R$, this would
    imply $xc \in R$.  But $c \in J \subseteq I \subseteq Q$, so by
    Lemma~\ref{vp-scalars},
    \[ \val(\partial(xc)) = \val(c) + \val(\partial x) < \val(c) - \gamma = 0.\]
    By Lemma~\ref{vp-class}, $xc \notin R$, a contradiction.
  \item $e$ is generated by $c$ and $xc$.  Then
    \begin{equation*}
      e = xcy_0 + cz_0,
    \end{equation*}
    for some $y_0, z_0 \in R$.
    If $y = cy_0$ and $z = cz_0$, then $y, z \in J$ (as $J \lhd R$),
    and $e = xy + z$.  So $x \in S$.
  \end{itemize}
  Conversely, suppose $x \in S$ and $\val(x) > \gamma$.  Then there
  are $y, z \in J \subseteq I \subseteq Q$ such that
  \[ e = xy + z.\]
  Now $Q$ is a subring, and $y, z \in Q$, $e \notin Q$.  Therefore $x
  \notin Q$.  On the other hand, $\val(x) > \gamma \ge 0$, so $x \in
  \Oo$.  Therefore $x \in \Oo \setminus Q$, which implies
  $\val(\partial x) \le 0$ by Lemma~\ref{vp-class}.
\end{proof}
\begin{lemma}\label{vee-co}
  Some $\vee$-definable set is co-embeddable with $R$.
\end{lemma}
\begin{proof}
  Take $\gamma$ and $S$ as in Lemma~\ref{exotic2}.  By
  Theorem~\ref{th:that-vee}, we can find $\gamma'$ and a definable set
  $B$ such that
  \[ \val(x) > \gamma' \implies x \in B \implies \val(x) > \gamma.\]
  \begin{claim}\label{cl1}
    If $\val(x) > \gamma'$ and $\val(\partial x) > 0$, then $x \in B
    \setminus S$.
  \end{claim}
  \begin{claimproof}
    Because $\val(x) > \gamma'$, we have $x \in B$ and $\val(x) >
    \gamma$, and so Lemma~\ref{exotic2} applies.  Then $\val(\partial
    x) > 0$ implies $x \notin S$, by the contrapositive to
    Lemma~\ref{exotic2}.
  \end{claimproof}
  \begin{claim}\label{cl2}
    If $x \in B \setminus S$, then $\val(x) > \gamma$ and
    $\val(\partial x) \ge - \gamma$.
  \end{claim}
  \begin{claimproof}
    The fact that $x$ is in $B$ implies that $\val(x) > \gamma$, and
    thus that Lemma~\ref{exotic2} applies.  By the contrapositive to
    Lemma~\ref{exotic2}, $x \notin S$ implies $\val(\partial x) \ge
    \gamma$.
  \end{claimproof}
  By Lemma~\ref{cofinal-valuations}, there is $b \in Q$ such that
  $\val(b) > \gamma'$.  Then
  \begin{equation}
    b R \subseteq B \setminus S. \label{hah1}
  \end{equation}
  Indeed, if $x \in R$, then
  \begin{align*}
    \val(bx) &= \val(b) + \val(x) > \gamma' + 0 \\
    \val(\partial(bx)) &\ge \val(b) + \val(\partial x) > \gamma' + 0 \ge 0,
  \end{align*}
  and Claim~\ref{cl1} applies.  Also,
  \begin{equation}
    b \cdot (B \setminus S) \subseteq R. \label{hah2}
  \end{equation}
  Indeed, if $x \in B \setminus S$, then
  \begin{align*}
    \val(x) &> \gamma \\
    \val(\partial x) &\ge -\gamma
  \end{align*}
  by Claim~\ref{cl2}.  But then $bx \in R$:
  \begin{align*}
    \val(bx) &= \val(b) + \val(x) > \gamma' + \gamma \ge 0 \\
    \val(\partial(bx)) &\ge \val(b) + \val(\partial x) > \gamma' - \gamma \ge 0.
  \end{align*}
  By (\ref{hah1})-(\ref{hah2}), the $\vee$-definable set $B \setminus
  S$ is co-embeddable with $R$.
\end{proof}
\begin{lemma}\label{lem:jbound}
  The set $J$ is bounded with respect to the canonical topology on
  $\Kk$: for any basic neighborhood $U$, there is $a \in \Kk^\times$
  such that $a J \subseteq U$.
\end{lemma}
\begin{proof}
  Recall that $J$ is the set $J_K$ of $K$-infinitesimals.  Let $K'$ be
  a small model containing $K$ and defining $U$.  Then $U$ contains
  the group $J_{K'}$ of $K'$-infinitesimals.  By Corollary~8.9 in
  \cite{prdf2}, there is non-zero $a$ such that
  \begin{equation*}
    a \cdot J_K \subseteq J_{K'} \subseteq U. \qedhere
  \end{equation*}
\end{proof}
Recall from \cite{prestel-ziegler}, \S 2, that a ring topology is
\emph{locally bounded} if there is a bounded neighborhood of 0.  If
$R$ is a proper subring of a field $K = \Frac(R)$, then $R$ induces a
locally bounded ring topology on $K$, as in \cite{prestel-ziegler},
Example 1.2 and Theorem 2.2(a).
\begin{theorem} \label{thm:definable}
  The canonical topology on $\Kk$ is locally bounded, definable, and
  induced by $R$.
\end{theorem}
\begin{proof}
  Note that $J$ and $R$ are co-embeddable, as
  \[ c R \subseteq J \subseteq R\]
  for any non-zero $c \in J$.  By Lemma~\ref{vee-co}, some
  $\vee$-definable set $U$ is co-embeddable with $R$ and $J$.  As $J$
  itself is type-definable, we can take $U$ to be \emph{definable} by
  Remark~\ref{interpolate}.  Rescaling $U$, we may assume $J \subseteq
  U$.  By compactness, there is a $K$-definable basic neighborhood $V$
  such that
  \[ J \subseteq V \subseteq U,\]
  as $J$ is the directed intersection of such neighborhoods.
  Therefore $U$ is a neighborhood of 0.  Also, $U$ is bounded, because
  it is co-embeddable with the bounded set $J$.  Therefore the
  canonical topology is locally bounded.  By Lemma 2.1(e) in
  \cite{prestel-ziegler}, the family 
  \begin{equation*}
    \{a U : a \in \Kk^\times\}
  \end{equation*}
  is a neighborhood basis of 0.  Then the canonical topology is
  definable by Lemma~\ref{duh}.  The family
  \[ \{a R : a \in \Kk^\times\}\]
  is also a neighborhood basis of 0, because $U$ and $R$ are co-embeddable.
\end{proof}
Once the canonical topology is definable on the monster, it is
uniformly definable on all models:
\begin{theorem}\label{thm:uni-def}
  ~
  \begin{enumerate}
  \item \label{ud1} There is a formula $\varphi(x;\vec{y})$ such that
    for every small model $K \preceq \Kk$, the family of sets
    \[ \{ \varphi(K;\vec{b}) : \vec{b} \in K^{|\vec{y}|}\}\]
    is a neighborhood basis of 0 for the canonical topology on $K$.
  \item \label{ud2} If $K, K'$ are two small submodels, then $K$ and $K'$ with
    their canonical topologies are ``locally equivalent'' in the sense
    of \cite{prestel-ziegler}.
  \end{enumerate}
\end{theorem}
\begin{proof}
  Let $\{\psi_i(x;\vec{z}_i)\}_{i \in I}$ be as in
  Remark~\ref{0-ind-def}, so that
  \[ \{ \psi_i(K;\vec{c}) : i \in I, ~ \vec{c} \in K^{|\vec{z}_i|} \}\]
  is the set of basic neighborhoods on any $K \equiv \Kk$.
  
  On $\Kk$, Theorem~\ref{thm:definable} gives a ($\Kk$-)definable
  neighborhood basis $\mathcal{N}$.  By saturation, there must be a
  finite subset $I_0 \subseteq I$ such that every set in $\mathcal{N}$
  has the form $\psi_i(K;\vec{c})$ for some $i \in I_0$.  The fact
  that $\mathcal{N}$ is a neighborhood basis implies that
  \begin{equation*}
    \forall j \in I ~ \forall \vec{c} ~ \exists i \in I_0 ~ \exists
    \vec{e} : \psi_i(\Kk;\vec{e}) \subseteq \psi_j(\Kk;\vec{c}).
  \end{equation*}
  This is a small conjunction of first-order sentences, so it holds in
  submodels $K \preceq \Kk$.  Then for any small model $K$, the family
  \begin{equation*}
    \{ \psi_i(K;\vec{c}) : i \in I_0, ~ \vec{c} \in K^{|\vec{z}_i|}\}
  \end{equation*}
  is a neighborhood basis of 0.  Because $I_0$ is finite, this can be
  written as
  \begin{equation*}
    \{ \varphi(K;\vec{b}) : \vec{b} \in K^{|\vec{y}|}\}
  \end{equation*}
  for some formula $\varphi(x;\vec{y})$.

  This proves the first point.  The second point is immediate, because
  local sentences can be evaluated on a neighborhood basis
  (\cite{prestel-ziegler}, Theorem 1.1(a)).
\end{proof}

\subsection{Odd positive characteristic}
In Theorem~\ref{thm:what}, we can weaken the assumption
$\characteristic(K) = 0$ to $\characteristic(K) \ne 2$; the same proof
works.  But the positive characteristic case then leads to a
contradiction:
\begin{proposition}
  Let $\Kk$ be a monster model of an unstable field with $\dpr(\Kk)
  \le 2$.  If $\Kk$ is not of valuation type, then
  $\characteristic(\Kk)$ is 0 or 2.
\end{proposition}
\begin{proof}
  Suppose $\characteristic(\Kk) > 2$.  As in the proof of
  Theorem~\ref{thm:what}, there would be a 2-inflator satisfying the
  Strong Assumptions of \S\ref{sec:canform}--\ref{sec:der}.  But strongly
  dependent fields are perfect, and so $\characteristic(\Kk) = 0$ by
  Proposition~\ref{prop:perfect}.
\end{proof}
Therefore
\begin{theorem}
  If $K$ is a field with $\dpr(K) \le 2$ and $\characteristic(K) > 2$,
  then either $K$ is stable, or $K$ is valuation type.
\end{theorem}
Perhaps this can be proven in characteristic 2 as well.

\section{Reduced rank and generators} \label{sec:wn}
Let $R$ be a noncommutative ring, and $M$ be an $R$-module.  Say that
$M$ has \emph{property $W_n$} if the following holds: for any $a_0,
a_1, \ldots, a_n \in M$, there is some $0 \le i \le n$ such that
\[ a_i \in R \cdot a_1 + \cdots + R \cdot a_{i-1} + R \cdot a_{i+1} + \cdots + R \cdot a_n.\]
In other words, any submodule of $M$ generated by a set $S$ of size
$n+1$ is generated by an $n$-element subset of $S$.
\begin{remark}
  This property appeared in Lemma~\ref{three-generators}, which said
  that $K$ has property $W_2$ as a $Q$-module or $R$-module.
\end{remark}
\begin{lemma} \label{lem:wn}
  ~
  \begin{enumerate}
  \item If $M_0, \ldots, M_n$ are non-zero, then $M_0 \oplus \cdots \oplus M_n$ does not have property $W_n$.
  \item If $M$ has property $W_n$ and $N \le M$, then $N$ has property $W_n$.
  \item If $M$ has property $W_n$ and $N \le M$, then $M/N$ has property $W_n$.
  \end{enumerate}
\end{lemma}
\begin{proof}
  ~
  \begin{enumerate}
  \item Take $a_i$ a non-zero element of $M_i$, viewed as an element
    of the direct sum.  Then $\{a_0,a_1,\ldots,a_n\}$ violates
    property $W_n$.
  \item Clear.
  \item Given $a_i \in M/N$, lift them to $\tilde{a_i} \in M$, apply
    property $W_n$ in $M$ to obtain $i$ and $r_0, \ldots, r_n \in R$
    such that
    \[ \tilde{a}_i = r_0 \tilde{a}_0 + \cdots + r_{i-1} \tilde{a}_{i_1} + r_{i+1} \tilde{a}_{i+1} + \cdots + r_n \tilde{a}_n,\]
    and then project back to $M/N$. \qedhere
  \end{enumerate}
\end{proof}
\begin{proposition} \label{prop:wn}
  $M$ has property $W_n$ if and only if the reduced rank of
  $\Sub_R(M)$ is at most $n$.
\end{proposition}
\begin{proof}
  If the reduced rank of $\Sub_R(M)$ is greater than $n$, then there
  is a strict $(n+1)$-cube in $M$.  This corresponds to submodules
  $M^- \le M^+ \le M$ and an isomorphism
  \[ M^+/M^- \cong N_0 \oplus \cdots \oplus N_n\]
  where the $N_i$ are non-zero $R$-modules.  By Lemma~\ref{lem:wn},
  the right hand side does not satisfy $W_n$, and therefore neither do
  $M^+$ or $M$.

  Conversely, suppose $W_n$ fails, witnessed by $a_0, \ldots, a_n \in
  M$.  Let $N_i = M \cdot a_i$.  Then
  \begin{equation*}
    N_0 + \cdots + N_n > N_0 + \cdots + N_{i-1} + N_{i+1} + \cdots + N_n,
  \end{equation*}
  for any $0 \le i \le n$.  By Proposition~6.3.2 in \cite{prdf3},
  $\Sub_R(M)$ has reduced rank greater than $n$.
\end{proof}

\section{Diffeovaluation data} \label{sec:dv-fields}
In \S\ref{sec:dv-fields}, all fields will have characteristic 0, and
all rings will be $\Qq$-algebras.
\subsection{Mock $K/\mm$'s} \label{sec:mock}
Let $K$ be a valued field with valuation ring $\Oo$, maximal ideal
$\mm$, and residue field $k = \Oo/\mm$.
\begin{definition}
  A \emph{mock $K/\mm$} is a divisible $\Oo$-module $D$ extending $k$
  satisfying the following property: for any $x, y \in D$,
  \[ x \in \Oo \cdot y \text{ or } y \in \Oo \cdot x.\]
\end{definition}
Note that $K/\mm$ is naturally a mock $K/\mm$.

There is a theory $T$ whose models are pairs $(K,D)$, where $K$ is a
valued field and $k \hookrightarrow D$ is a mock $K/\mm$.
\begin{proposition} \label{prop:mog}
  Let $(K,D)$ be a model of $T$.  If $(K,D)$ is countable or
  $\aleph_1$-resplendent, then $D$ is isomorphic (as an extension of
  $k$) to $K/\mm$.
\end{proposition}
\begin{proof}
  The resplendent case follows from the countable case.  Assume $K, D$
  are countable.  Then the value group $\Gamma$ has countable
  cofinality.  Take a sequence
  \[ a_0, a_1, \ldots \]
  in $K$ such that $a_0 = 1$ and the sequence
  \[ \val(a_0), \val(a_1), \ldots\]
  is descending with no lower bound.  Then
  \begin{equation*}
    \Oo = \Oo \cdot a_0 \subseteq \Oo \cdot a_1 \subseteq \cdots 
  \end{equation*}
  and the union of this chain is $K$.

  By divisibility, we can find a sequence
  \begin{equation*}
    b_0, b_1, \ldots
  \end{equation*}
  in $D$ such that
  \begin{itemize}
  \item $b_0$ is the image of $1$ under the embedding $k
    \hookrightarrow D$.
  \item $b_{i-1} = (a_{i-1}/a_i) b_i$, for all $i \ge 1$.
  \end{itemize}
  By induction on $i$, the $b_i$ are all non-zero.  Define $f_i : \Oo
  \cdot a_i \to D$ by $f_i(x) = (x/a_i)b_i$.  If $x \in \Oo \cdot
  a_i$, then
  \[ f_i(x) = (x/a_i)b_i = (x/a_i)(a_i/a_{i+1})b_{i+1} = (x/a_{i+1})b_{i+1} = f_{i+1}(x).\]
  Therefore the $f_i$ glue together to yield a morphism
  \[ f : K \to D.\]
  Moreover, $f(a_i) = f_i(a_i) = (a_i/a_i)b_i = b_i$ for all $i$.
  \begin{claim}
    For any $x \in K$,
    \[ f(x) = 0 \iff \val(x) > 0.\]
  \end{claim}
  \begin{claimproof}
    First suppose $\val(x) \ge 0$.  Then $x \in \Oo = \Oo \cdot a_0 =
    \dom(f_0)$, and so
    \[ f(x) = f_0(x) = (x/a_0)b_0 = xb_0.\]
    By choice of $b_0$, the annihilator $\Ann_\Oo(b_0)$ is exactly
    $\mm$, and so
    \[ f(x) = 0 \iff xb_0 = 0 \iff x \in \mm \iff \val(x) > 0.\]
    Next suppose $\val(x) \le 0$.  Then $1/x \in \Oo$, and so
    \[ f(1) = f((1/x)x) = (1/x) f(x),\]
    because $f$ is $\Oo$-linear.  By the first case,
    \[ \val(1) = 0 \implies f(1) \ne 0 \implies f(x) \ne 0. \qedhere\]
  \end{claimproof}
  Therefore $\ker(f) = \mm$, and $f$ induces an embedding
  \[ K/\mm \hookrightarrow D.\]
  Restricted to $k = \Oo/\mm$, this embeding is
  \[ (x + \mm) \mapsto f(x) = f_0(x) = (x/a_0)b_0 = xb_0.\]
  By choice of $b_0$, this is the given embedding of $k$ into $D$.

  It remains to show that $f$ is onto.  Suppose not.  Take an element
  $y$ of $D$ that is not in the image of $f$.  Then
  \[ y \notin \Oo \cdot b_i,\]
  for any $i$, since $\Oo \cdot b_i$ is the image of $f_i$.  By
  definition of mock $K/\mm$, it follows that
  \begin{equation*}
    b_i \in \Oo \cdot y
  \end{equation*}
  for all $i$.  In particular, there are $c_i \in \Oo$ such that $b_i = c_i y$.  Take $i$ large enough that $\val(a_ic_0) < 0$.  Then
  \begin{equation*}
    (1/a_i)y = \frac{1}{a_ic_0} c_0 y = \frac{1}{a_ic_0} b_0 = 0,
  \end{equation*}
  because $\Ann(b_0) = \mm$.  But then
  \begin{equation*}
    b_0 = f_i(a_0) = f_i(1) = (1/a_i)b_i = (1/a_i)c_iy = c_i(1/a_i)y = c_i0 = 0,
  \end{equation*}
  contradicting the choice of $b_0$.
\end{proof}
\begin{corollary}
  Let $D$ be a mock $K/\mm$.  Then there is a map $\val : D \to \Gamma
  \cup \{+\infty\}$ with the following properties:
  \begin{enumerate}
  \item $\val(x) \le 0$ or $\val(x) = +\infty$ for all $x
    \in D$.
  \item $\val(x) = +\infty$ if and only if $x = 0$.
  \item $\val(x) \ge 0$ if and only if $x$ is in the image of $k
    \hookrightarrow D$.
  \item For any $a \in \Oo$ and $x \in D$,
    \begin{equation*}
      \val(ax) = 
      \begin{cases}
        \val(a) + \val(x) & \text{ if } \val(a) + \val(x) \le 0 \\
        +\infty & \text{ if } \val(a) + \val(x) > 0.
      \end{cases}
    \end{equation*}
  \item For any $x, y \in D$,
    \begin{equation*}
      \val(x+y) \ge \min(\val(x),\val(y)).
    \end{equation*}
  \item If $\val(x) \le \val(y)$, then $y \in \Oo \cdot x$.
  \item For any $\gamma \in \Gamma$, there is $x \in D$ such that
    $\val(x) \le \gamma$.
  \end{enumerate}
\end{corollary}
\begin{proof}
  Let $D^*$ be the image of $k \setminus \{0\}$ in $D$.  We claim that
  \begin{itemize}
  \item For every non-zero $x$ in $D$, there is an $a \in \Oo$ such
    that $ax \in D^*$.
  \item If $ax \in D^*$ and $bx \in D^*$, then $\val(a) = \val(b)$.
  \item If we define $\val : D \to \Gamma \cup \{+\infty\}$ as
    \begin{equation*}
      \val(x) = 
      \begin{cases}
        +\infty & \text{ if } x = 0 \\
        -\val(a) & \text{ if } ax \in D^*,
      \end{cases}
    \end{equation*}
    then $\val$ satisfies the listed conditions.
  \end{itemize}
  These three claims can be expressed by a first-order sentence, so we
  may pass to a resplendent elementary extension.  Then we may assume
  $D$ is $K/\mm$, in which case the three claims are straightforward.
\end{proof}

\subsection{Diffeovalued fields}
\begin{definition}
  A \emph{diffeovalued field} is a structure $(K,\Oo,D,\partial)$
  where
  \begin{itemize}
  \item $(K,\Oo)$ is a valued field (of equicharacteristic 0).
  \item $k \hookrightarrow D$ is a mock $K/\mm$.
  \item $\partial : \Oo \to D$ is a derivation.
  \end{itemize}
\end{definition}
The theory of diffeovalued fields is first-order.
\begin{definition}
  A \emph{normalization} of a diffeovalued field is a choice of an
  isomorphism $D \cong K/\mm$ (respecting the embedding $k
  \hookrightarrow D$).
\end{definition}
Every sufficiently resplendent diffeovalued field admits a
normalization, by Proposition~\ref{prop:mog}.
\begin{definition}
A \emph{normalized diffeovalued field} is a diffeovalued field with a choice of a
normalization.
\end{definition}
Equivalently, a normalized diffeovalued field is a valued field $(K,\Oo)$ with a
derivation $\partial : \Oo \to K/\mm$.
\begin{definition}
  Let $K$ be a normalized diffeovalued field.  A \emph{lifting} is a derivation
  $\delta : K \to K$ such that
  \[ \partial x = (\delta x) + \mm\]
  for $x \in \Oo$.
\end{definition}
\begin{proposition}\label{prop:lift}
  Let $K$ be a sufficiently resplendent diffeovalued field.  Suppose
  the value group $\Gamma$ is $p$-divisible for at least one prime.
  Then $K$ admits a normalization and a lifting.
\end{proposition}
\begin{proof}
  The normalization comes from Proposition~\ref{prop:mog}, and the
  lifting comes from Corollary~\ref{metaval-lifting} in the appendix.
\end{proof}
Some assumption on the value group is necessary:
Proposition~\ref{z-probs} in the appendix gives an example of a
normalized diffeovalued field which cannot be lifted, even after
passing to an elementary extension.
\begin{definition}
  A \emph{lifted diffeovalued field} is a normalized diffeovalued
  field with a choice of a lifting.
\end{definition}
Equivalently, a lifted diffeovalued field is a field with a derivation and a
valuation.
\begin{lemma}\label{inverses}
  Let $K$ be a diffeovalued field.  If $x \in \Oo^\times$, then
  $\partial(x^{-1}) = -x^{-2} \partial x$.
\end{lemma}
\begin{proof}
  This follows as usual from
  \[ 0 = \partial (x x^{-1}) = x^{-1} \partial x + x \partial (x^{-1}). \qedhere\]
\end{proof}

\subsection{Dense diffeovalued fields}
\begin{definition}
  A diffeovalued field is \emph{dense} if for every $x \in D$, the
  fiber
  \begin{equation*}
    \{y \in \Oo : \partial y = x\}
  \end{equation*}
  is dense in $\Oo$, with respect to the valuation topology.
\end{definition}
The theory of dense diffeovalued fields is first-order.
\begin{remark} \label{rem:dnon}
  Denseness implies that the value group $\Gamma$ is non-trivial.
\end{remark}
\begin{proof}
  If the valuation is trivial, then $\Oo = K$ and the valuation
  topology is discrete.  Then every fiber
  \begin{equation*}
    \{y \in K : \partial y = x\}
  \end{equation*}
  is dense in $\Oo$, hence equal to $\Oo$.  This is absurd unless $D$
  is a singleton.  But $D$ contains a submodule isomorphic to $k \cong
  K$.
\end{proof}
\subsection{The diffeovaluation topology} \label{sec:dtop}
Fix a \emph{dense} diffeovalued field $K$.  Define
\begin{align*}
  R = \{x \in K &: \val(x) \ge 0 \text{ and } \val(\partial x) \ge 0\} \\
  Q = \{x \in K &: \val(x) \ge 0 \text{ and } \val(\partial x) > 0\} \\
  I = \{x \in K &: \val(x) > 0 \text{ and } \val(\partial x) > 0\}.
\end{align*}
Note that in the definition of $Q, I$,
\begin{equation*}
  \val(\partial x) > 0 \iff \partial x = 0,
\end{equation*}
because $D$ is a mock $K/\mm$.  In fact, $Q$ is merely the kernel of
$\partial : \Oo \to D$.
\begin{lemma}\label{lem-friends-v2}
  ~
  \begin{enumerate}
  \item \label{lef1} $R, Q$ are proper subrings of $K$.
  \item $I$ is a proper ideal in $R$ and in $Q$.
  \item \label{lef3} $\Frac(Q) = \Frac(R) = K$.
  \item \label{lef4} $Q$ is a local ring with maximal ideal $I$.
  \item $I \ne 0$.
  \end{enumerate}
\end{lemma}
\begin{proof}
  ~
  \begin{enumerate}
  \item Easy.  Properness holds because $R, Q \subseteq \Oo$, and
    $\Oo$ is a proper subring by Remark~\ref{rem:dnon}.
  \item Easy.  Properness holds because $1 \notin I$.
  \item As $Q \subseteq R$, it suffices to show $\Frac(Q) = K$.  Given
    $a \in K$, we must show $a \in \Frac(Q)$.  Replacing $a$ with
    $a^{-1}$, we may assume $a \in \Oo$.  If $a \in Q$, we are done.
    Otherwise, $\val(\partial a) \le 0$.  By density, there is $b$
    such that
    \begin{align*}
      \val(b) &> - \val(\partial a) \\
      b &\ne 0 \\
      \partial b &= 0.
    \end{align*}
    Then $b \in Q$, and so $ab \in \Oo$.  Also,
    \[ \partial (ab) = a \partial b + b \partial a = b \partial a = 0.\,\]
    because $\partial b = 0$, and $\val(b)$ is high enough for $b$ to
    annihilate $\partial a$.  Thus $ab \in Q$ and $a = (ab)/b \in
    \Frac(Q)$.
  \item Since $I$ is a proper ideal in $Q$, it suffices to show
    \begin{equation*}
      x \in Q \setminus I \implies x^{-1} \in Q.
    \end{equation*}
    Suppose $x \in Q \setminus I$, so that $\val(x) = 0$ and $\partial
    x = 0$.  Then $x^{-1} \in \Oo$, and
    \begin{equation*}
      \partial(x^{-1}) = -x^{-2} \partial x = 0,
    \end{equation*}
    by Lemma~\ref{inverses}.  So $x^{-1} \in Q$.
  \item By (\ref{lef4}), $Q/I$ is a field.  If $I = 0$, then $Q$ is a
    field, and $Q = K$ by (\ref{lef3}).  This contradicts
    (\ref{lef1}). \qedhere
  \end{enumerate}
\end{proof}
\begin{proposition}\label{prop:dv-top}
  Let $K$ be a dense diffeovalued field.  There is a locally bounded
  field topology on $K$ characterized by the fact that either of the
  following are a neighborhood basis of 0:
  \begin{align*}
    \{a R : a \in K^\times\} \\
    \{a Q : a \in K^\times\}.
  \end{align*}
  The topology is locally bounded, non-discrete, and Hausdorff.
\end{proposition}
\begin{proof}
  The ring $Q$ induces a locally bounded, non-discrete, Hausdorff
  field topology because $Q$ is a proper local subring of $K =
  \Frac(Q)$.  For example, see Theorem 2.2(b) in
  \cite{prestel-ziegler}.  If $a$ is a non-zero element of $I$, then
  \[ R \cdot a \subseteq R \cdot I = I \subseteq Q,\]
  so $R$ and $Q$ induce the same topology.
\end{proof}
\begin{definition}
  The \emph{diffeovaluation topology} is the topology induced by $Q$ or $R$ as in
  Proposition~\ref{prop:dv-top}.
\end{definition}
\begin{proposition}\label{not-a-v}
  The diffeovaluation topology is not a V-topology.
\end{proposition}
\begin{proof}
  In a V-topology, the following local sentence holds, where $U, V$
  range over neighborhoods of 0:
  \begin{equation*}
    \forall U~ \exists V~ \forall x, y : ((xy \in V) \rightarrow (x \in
    U \text{ or } y \in U)).
  \end{equation*}
  However, this fails for $U = R$.  Indeed,
  suppose $V = a R$ is such that
  \begin{equation*}
    xy \in aR \implies (x \in R \vee y \in R).
  \end{equation*}
  Shrinking the set $aR$, we may assume $a \in R$.  By density, there
  is $x$ such that
  \begin{align*}
    \val(x) &> \val(a) \\
    \val(\partial x) &< 0.
  \end{align*}
  Let $y = a/x$.  Then
  \begin{equation*}
    xy = a \in a R.
  \end{equation*}
  On the other hand, $x \notin R$ by choice of $\val(\partial x)$,
  and $y \notin \Oo \supseteq R$, by choice of $\val(x)$.
\end{proof}
\begin{definition} \label{def:dv-top}
  A \emph{DV-topology} is a field topology that is locally equivalent
  to a dense diffeovaluation topology.
\end{definition}

\subsection{Lifted diffeovalued fields}
For lifted diffeovalued fields, we can characterize density and the
diffeovaluation topology more naturally.
\begin{proposition}\label{prop:strong-density}
  Let $(K,\delta,\val)$ be a lifted diffeovalued field.  Then $K$ is dense if and
  only if for every $a, b \in K$ and every $\gamma \in \Gamma$, there
  is $x \in K$ such that
  \begin{align*}
    \val(x - a) & > \gamma \\
    \val(\delta x - b) & > \gamma
  \end{align*}
\end{proposition}
\begin{proof}
  Unwinding the definition, density says that we can solve equations
  of the form
  \begin{align*}
    \val(x - a) & > \gamma \\
    \val(\delta x - b) & > 0
  \end{align*}
  when $a \in \Oo$.  So the listed conditions certainly imply density.
  Conversely, suppose density holds.
  \begin{claim}
    For any $\gamma \in \Gamma$ and $b \in K$, there is $x$ such that
    \begin{align*}
      \val(x) & > \gamma \\
      \val(\delta x - b) & > \gamma.
    \end{align*}
  \end{claim}
  \begin{claimproof}
    Take some non-zero $a$ such that $\val(a) > \gamma$.

    By density, there is $y$ such that
    \begin{align*}
      \val(y) & > \max(\gamma - \val(\delta a), 0) \\
      \val(\delta y - b/a) & > 0.
    \end{align*}
    Let $x = ya$.  Then
    \begin{align*}
      \val(x) &= \val(ya) = \val(y) + \val(a) > 0 + \gamma \\
      \val(a \delta y - b) & > \val(a) > \gamma.
    \end{align*}
    Also,
    \begin{equation*}
      \val(y \delta a) = \val(y) + \val(\delta a) > \gamma.
    \end{equation*}
    So we see that
    \begin{equation*}
      \val(\delta x - b) = \val(y \delta a + (a \delta y - b)) > \gamma. \qedhere
    \end{equation*}
  \end{claimproof}
  Now given any $a, b, \gamma$, we can find $\varepsilon$ such that
  \begin{align*}
    \val(\varepsilon) &> \gamma \\
    \val(\delta \varepsilon + \delta a - b) &> \gamma.
  \end{align*}
  Set $x = a + \varepsilon$.  Then
  \begin{align*}
    \val(x - a) &= \val(\varepsilon) > \gamma \\
    \val(\delta x - b) &= \val(\delta \varepsilon + \delta a - b) > \gamma. \qedhere
  \end{align*}
\end{proof}

\begin{proposition}\label{prop:dense-dvtop}
  If $(K,\delta,\val)$ is a dense, lifted diffeovalued field, then sets of the
  form
  \begin{equation*}
    \{x \in K : \val(x - a_1) > \gamma_1 \text{ and } \val(\delta x
    - a_2) > \gamma_2\} \text{ for } \gamma_1, \gamma_2 \in \Gamma,~
    a_1, a_2 \in K
  \end{equation*}
  form a basis of opens in the diffeovaluation topology.
\end{proposition}
\begin{proof}
  Let $B_\gamma = \{x \in K : \val(x) > \gamma \text{ and } \val(\delta
  x) > \gamma\}$.  It suffices to show that the $B_\gamma$ form a
  neighborhood basis of 0.  First of all, one sees by a straightforward calculation that
  $B_\gamma$ is an $R$-submodule of $K$.  The strong form of density in Proposition~\ref{prop:strong-density} can be used to show that $B_\gamma$ is strictly bigger than $\{0\}$; for example take $a, b$ very small relative to $\gamma$, and then take $x$ with $x - a$ and $\delta x - b$ very small relative to $a, b, \gamma$.  Thus $B_\gamma$ is a
  neighborhood of 0 in the diffeovaluation topology.  Conversely, given any $a \in
  K^\times$, we claim that
  \begin{equation*}
    B_\gamma \subseteq a^{-1} \cdot R
  \end{equation*}
  for sufficiently large $\gamma$.  Indeed, if $\gamma$ is sufficiently large relative to $a$, and $x \in B_\gamma$, then
  \begin{align*}
    \val(a x) &= \val(a) + \val(x) \ge \val(a) + \gamma \ge 0 \\
    \val(\delta(ax)) &= \val(a \delta x + x \delta a) \ge \gamma
    + \min(\val(a),\val(\delta a)) \ge 0. \qedhere
  \end{align*}
\end{proof}

\subsection{Diffeovaluation inflators} \label{dvi:sec}
Fix a \textbf{dense} diffeovalued field $K$.  Let $D_0$ be the image
of the embedding $k \hookrightarrow D$, and let $\res' : D_0 \to k$ be
the inverse of this embedding.  Let $Q, R, I$ be as in
\S\ref{sec:dtop}.  Note that if $a \in R$, then $\partial a \in D_0$,
and so $\res'(\partial a)$ makes sense.
\begin{lemma}
  ~
  \begin{enumerate}
  \item \label{za1} If $a \in \Oo$ and $b \in R$, then
    \begin{equation*}
      \res'(a \partial b) = \res(a) \res'(\partial b).
    \end{equation*}
  \item The quotient $Q/I$ is isomorphic (as a ring) to $k$ via the
    map $\res(-)$.  Therefore, we can regard $k$-modules as
    $Q$-modules.
  \item $R$ and $I$ are $Q$-submodules of $K$.
  \item For $x \in R$, let $\wres(x) = (\res(x),\res'(\partial x)) \in
    k^2$.  Then $\wres$ induces an isomorphism of $Q$-modules from
    $R/I$ to $k^2$.  In particular, $R/I$ is a semisimple $Q$-module
    of length 2.
  \end{enumerate}
\end{lemma}
\begin{proof}
  \begin{enumerate}
  \item This holds because $\res'$ is an $\Oo$-linear map from $D_0$
    to $k$, and the $\Oo$-module structure on $k$ comes from $\res :
    \Oo \to k$.
  \item The ring homomorphism $\res : Q \to k$ is onto, by density.
    Indeed, given any $x \in k$, we can find $y \in \Oo$ such that
    \begin{align*}
      \res(x) &= y \\
      \partial x &= 0.
    \end{align*}
    Then $x \in Q$ and $\res(x) = y$.  The kernel of $\res : Q \to k$
    is $I$, by definition of $Q$ and $I$.
  \item $R$ is a $Q$-module because $R$ is a superring of $Q$.  $I$ is
    a $Q$-module because $I$ is an ideal in $R$.
  \item The map $\wres : R \to k^2$ is obviously $\Zz$-linear.  It is
    surjective by density.  The kernel is $I$, by definition of $R$
    and $I$.  For $Q$-linearity, suppose $x \in Q$ and $y \in R$.
    Then
    \begin{equation*}
      \res(xy) = \res(x)\res(y).
    \end{equation*}
    Also, $x \in Q$ implies $\partial x = 0$, and so
    \begin{equation*}
      \res'(\partial(xy)) = \res'(x \partial y + y \partial x) =
      \res'(x \partial y) = \res(x)\res'(\partial y)
    \end{equation*}
    by part (\ref{za1}).  Thus
    \begin{equation*}
      \wres(xy) = \res(x) \cdot (\res(y), \res'(\partial y)) = \res(x)\wres(y). \qedhere
    \end{equation*}
  \end{enumerate}
\end{proof}
\begin{lemma} \label{lem:qw3}
  For any $a, b, c \in K$, the $Q$-submodule generated by $\{a,b,c\}$
  is generated by a two-element subset of $\{a,b,c\}$.
\end{lemma}
\begin{proof}
  Without loss of generality, $\val(c) \le \val(a)$ and $\val(c) \le
  \val(b)$.  Rescaling, we may assume $c = 1$.  Then $a, b \in \Oo$.

  Without loss of generality, $\val(\partial a) \le \val(\partial b)$.
  Now, if $\val(\partial b) > 0$, then $b \in Q \cdot 1$, and we are
  done.  So we may assume
  \[ \val(\partial a) \le \val(\partial b) \le 0.\]
  Take $x_0 \in \Oo$ such that $\partial b = x_0 \partial a$.
  By density, there is some $x$ such that $\partial x = 0$ and
  \begin{equation*}
    \val(x - x_0) > -\val(\partial a) \ge 0.
  \end{equation*}
  Note that $\val(x - x_0) \ge 0$, so $x - x_0 \in \Oo$.  As $x_0 \in
  \Oo$, we see $x \in \Oo$.  Then $x \in Q$.

  Now $\val(x - x_0)$ is large enough that $x - x_0$ annihilates
  $\partial a$, so
  \[ x \partial a = x_0 \partial a = \partial b.\]
  Then
  \[ \partial (xa) = x \partial a + a \partial x = \partial b + 0.\]
  Let $y = b - xa$.  Then
  \begin{equation*}
    \partial y = \partial b - \partial(xa) = 0.
  \end{equation*}
  Also, $b, x, a \in \Oo$, and so $y \in
  \Oo$.  Thus $y \in Q$.  Then
  \[ b = xa + y \in Q \cdot a + Q \cdot 1. \qedhere\]
\end{proof}
\begin{theorem}\label{dv-inflator-construction}
  Let $K$ be a dense diffeovalued field.  There is a malleable
  $\Qq$-linear 2-inflator
  \begin{align*}
    \Dir_K(K) & \to \Dir_k(k^2) \\
    \Sub_K(K^n) & \to \Sub_k(k^{2n}) \\
    V & \mapsto \{(\wres(x_1),\ldots,\wres(x_n)) : \vec{x} \in V \cap R^n\},
  \end{align*}
  where $\wres(x) = (\res(x),\res'(\partial x))$.
\end{theorem}
\begin{proof}
  By Proposition~\ref{prop:wn} and Lemma~\ref{lem:qw3}, the reduced
  rank of $\Sub_Q(K)$ is at most 2.  As $R/I$ is a semisimple
  $Q$-module of length 2, we see that $\Sub_Q(K)$ has reduced rank
  exactly two.  Let
  \begin{itemize}
  \item $\mathcal{C}$ the category of $Q$-modules.
  \item $F : K\Vect \to Q\Mod$ the forgetful functor.
  \item $G : Q\Mod \to \Qq\Vect$ the forgetful functor.
  \end{itemize}
  Then Assumptions~8.1 and 8.11 of \cite{prdf3} hold.  Applying
  Propositions~8.9 and 8.12 in \cite{prdf3},
  we obtain a malleable 2-inflator
  \begin{align*}
    \Dir_K(K) & \to \Dir_Q(R/I) \\
    \Sub_K(K^n) & \to \Sub_Q((R/I)^n) \\
    V & \mapsto (V \cap R^n + I^n)/I^n.
  \end{align*}
  Now $(V \cap R^n + I^n)/I^n$ can be described as the image of $V
  \cap R^n$ under the projection $R^n \twoheadrightarrow (R/I)^n$.
  Under the isomorphism $\Dir_Q(R/I) \cong \Dir_Q(k^2) \cong
  \Dir_k(k^2)$, this is exactly
  \begin{equation*}
    \{(\wres(x_1),\ldots,\wres(x_n)) : \vec{x} \in V \cap R^n\}. \qedhere
  \end{equation*}
\end{proof}
\begin{definition}
  A \emph{diffeovaluation inflator} on $K$ is a 2-inflator on $K$
  arising from a dense diffeovaluation on $K$ via
  Theorem~\ref{dv-inflator-construction}.
\end{definition}

\subsection{Characterization of diffeovaluation inflators}
We can summarize \S\ref{sec:canform}--\ref{sec:der} as follows:
\begin{theorem}\label{thm:one-dir}
  Let $\varsigma$ be an isotypic, malleable 2-inflator on a field $K$
  of characteristic 0.  If no mutation of $\varsigma$ is weakly
  multi-valuation type, then $\varsigma$ is a diffeovaluation
  inflator.
\end{theorem}
\begin{proof}
  By Corollary~\ref{cor:o}, we have a valuation ring $\Oo$.  By
  Propositions~\ref{prop:q-to-o} and \ref{derivation}, we have an
  $\Oo$-module $D$ and a derivation $\partial : \Oo \to D$.  By
  Propositions~\ref{divisibility-v2}, \ref{divisibility-4real}, and
  \ref{res-connection-v2} (with Proposition
  \ref{same-residue-2:prop}), $D$ is a mock $K/\mm$, and the sets $R,
  I, Q$ of \S\ref{sec:canform}--\ref{sec:der} are exactly
  \begin{align*}
    R &= \{x \in \Oo : \val(x) \ge 0\} \\
    I &= \{x \in \mm : \partial x = 0\} \\
    Q &= \{x \in \Oo : \partial x = 0\}.
  \end{align*}
  Thus we have a diffeovaluation, and the sets $R, I, Q$ agree with
  the ones defined in \S\ref{sec:dtop}.
  
  The map $\partial : \Oo \to D$ is surjective, by its construction in
  \S\ref{ssec:der}.  Proposition~\ref{density} says that $Q$ is dense
  in $\Oo$.  Every other fiber of $\partial$ is a translate of $Q$, by
  surjectivity.  Therefore every fiber is dense.  So the
  diffeovaluation data is dense.

  Finally, by Propositions~\ref{exact-form} and \ref{res-connection-v2},
  $\varsigma$ has the same form as the diffeovaluation inflator
  constructed in Theorem~\ref{dv-inflator-construction}.
\end{proof}
Under the Weak Assumptions of \S\ref{sec:isotyp}--\ref{sec:der}, we
can say the following:
\begin{corollary}
  Let $\varsigma$ be a malleable 2-inflator on a field $K$ of
  characteristic 0.  Then some mutation $\varsigma'$ of $\varsigma$ is
  either weakly multi-valuation type, or a diffeovaluation inflator.
\end{corollary}
\begin{proof}
  If no mutation of $\varsigma$ is weakly multi-valuation type, then
  $\varsigma$ satisfies the Weak Assumptions of \S\ref{sec:isotyp}.
  By Corollary~\ref{cor:rediso}, there is some mutation $\varsigma'$
  which is isotypic.  By Remark~\ref{rem:inherit}, $\varsigma'$
  satisfies the Strong Assumptions of
  \S\ref{sec:canform}--\ref{sec:der}, and so $\varsigma'$ is a
  diffeovaluation inflator by Theorem~\ref{thm:one-dir}.
\end{proof}
We can use Theorem~\ref{thm:one-dir} to characterize diffeovaluation
inflators.  We first need some lemmas.
\begin{lemma}\label{lolsome}
  Let $K$ be a field of characteristic 0 and $\Oo_1, \ldots, \Oo_n$ be
  some valuation rings on $K$.  Let $b$ be an element of $K$.  Then
  there is non-zero $q \in \Qq$ such that $1/(b - q) \in \Oo_1 \cap
  \cdots \cap \Oo_n$.
\end{lemma}
\begin{proof}
  Let $\mm_i$ be the maximal ideal of $\Oo_i$.  We need non-zero $q$
  such that $b - q \notin \mm_i$ for $1 \le i \le n$.  For each $i$,
  let
  \begin{equation*}
    B_i = \{q \in \Qq : b - q \in \mm_i\}.
  \end{equation*}
  We must show that $B_1 \cup \cdots \cup B_n \cup \{0\}$ fails to
  cover all of $\Qq$.  There are three possibilities for each $B_i$:
  \begin{itemize}
  \item If $b \notin \Qq + \mm_i$, then $B_i$ is empty.
  \item Otherwise, if $\Oo_i$ has residue characteristic $0$, then $B_i$ is a singleton.
  \item Otherwise, if $\Oo_i$ has residue characteristic $p > 0$, then
    $B_i$ is a $p$-adic ball in $\Qq$ (of radius $1/p$).
  \end{itemize}
  We can take $q = 1/n!$, for $n \gg 0$.
\end{proof}
\begin{lemma}\label{lem:non-multival}
  Let $K$ be a dense diffeovalued field.  As usual, let
  \begin{equation*}
    R = \{x \in \Oo : \val(\partial x) \ge 0\}.
  \end{equation*}
  Let $S$ be a multi-valuation ring on $K$.  If $a S \subseteq R$,
  then $a = 0$.
\end{lemma}
\begin{proof}
  Note that $a = a \cdot 1 \in a \cdot S \subseteq R$, so $a \in R
  \subseteq \Oo$.

  Suppose $a \ne 0$.  By density, we can find $b \in K$ such that
  \begin{align*}
    \val(b) &> 0 \\
    \val(\partial b) & < \min(\val(\partial a),0) - \val(a)
  \end{align*}
  Then $b \in \mm$.  By Lemma~\ref{lolsome}, there is non-zero $q \in
  \Qq$ such that $1/(b-q) \in S$.  Then
  \begin{equation*}
    \frac{a}{b-q} \in aS \subseteq R.
  \end{equation*}
  Now $b - q \in \Oo^\times$ because $\Oo$ is equicharacteristic 0 and
  $q \ne 0$.  By Lemma~\ref{inverses},
  \begin{equation*}
    \partial \left( \frac{1}{b-q}\right) = \frac{-\partial b}{(b-q)^2},
  \end{equation*}
  and
  \begin{equation*}
    \val\left( \frac{-a \partial b}{(b-q)^2} \right) = \val(a) +
    \val(\partial b),
  \end{equation*}
  because the right hand side is less than 0.  Then
  \begin{equation*}
    \partial \left(\frac{a}{b-q}\right) = \frac{\partial a}{b-q} +
    \frac{-a \partial b}{(b-q)^2}.
  \end{equation*}
  But
  \begin{align*}
    \val\left(\frac{\partial a}{b-q}\right) &= \val(\partial a) \\
    \val\left( \frac{-a \partial b}{(b-q)^2} \right) &= \val(a) +
    \val(\partial b) < \val(\partial a).
  \end{align*}
  So
  \begin{equation*}
    \val\left(\partial \left(\frac{a}{b-q}\right)\right) = \val(a) +
    \val(\partial b) < 0.
  \end{equation*}
  So $a/(b-q) \notin R$, a contradiction.
\end{proof}

\begin{proposition} \label{prop:arr}
  Let $K$ be a dense diffeovalued field, and let $\varsigma :
  \Dir_K(K) \to \Dir_k(k^2)$ be the induced 2-inflator.  Then no
  mutation of $\varsigma$ is weakly multi-valuation type.
\end{proposition}
\begin{proof}
  Recall that $\varsigma$ is the 2-inflator induced by the pedestal
  $I$ in $\Sub_Q(K)$.  Suppose $\varsigma'$ is the mutation of
  $\varsigma$ along the line $K \cdot (a_1, a_2, \ldots, a_m)$.  By
  Proposition~10.15 in \cite{prdf3}, $\varsigma'$ is the 2-inflator
  induced by the pedestal
  \begin{equation*}
    I' = a_1^{-1} I \cap \cdots \cap a_m^{-1} I.
  \end{equation*}
  Note that $I'$ is non-zero, because it is open in the
  diffeovaluation topology on $K$.

  By Proposition~8.10 in \cite{prdf3}, the fundamental ring $R'$ of
  $\varsigma'$ is the ``stabilizer''
  \begin{equation*}
    R' = \{x \in K : xI' \subseteq I'\}.
  \end{equation*}
  Suppose for the sake of contradiction that $\varsigma'$ is weakly
  multi-valuation type.  Then there is a multivaluation ring $S$ on
  $K$ such that $R'$ contains a non-zero $S$-module.  So there is some
  non-zero $a \in K$ such that
  \begin{equation*}
    a \cdot S \subseteq R'.
  \end{equation*}
  Let $b$ be a non-zero element in $I'$.  Choose $i$ so that $a_i \ne
  0$.  Then
  \begin{equation*}
    a_i \cdot b \cdot a \cdot S \subseteq a_i \cdot b \cdot R'
    \subseteq a_i \cdot I' \subseteq I \subseteq R.
  \end{equation*}
  By Lemma~\ref{lem:non-multival}, $a_i b a = 0$, which is absurd.
\end{proof}
\begin{theorem}\label{thm:char}
  Let $\varsigma$ be a 2-inflator on a field $K$ of characteristic 0.
  Then $\varsigma$ is a diffeovaluation inflator if and only if the
  following conditions hold:
  \begin{itemize}
  \item $\varsigma$ is malleable.
  \item $\varsigma$ is isotypic.
  \item No mutation of $\varsigma$ is weakly multi-valuation type.
  \end{itemize}
\end{theorem}
\begin{proof}
  If the listed properties hold, then $\varsigma$ is a diffeovaluation
  inflator by Theorem~\ref{thm:one-dir}.  Conversely, suppose
  $\varsigma : \Dir_K(K) \to \Dir_k(k^2)$ is a diffeovaluation
  inflator.  Then $\varsigma$ is plainly isotypic, and malleable by
  Theorem~\ref{dv-inflator-construction}.  The final property holds by
  Proposition~\ref{prop:arr}.
\end{proof}

\section{The canonical topology in characteristic 0}
Let $(\Kk,+,\cdot,\ldots)$ be a field, possibly with extra structure.
Assume
\begin{itemize}
\item $\Kk$ is sufficiently resplendent
\item $\dpr(\Kk) \le 2$ and $\characteristic(\Kk) = 0$.
\item $\Kk$ is unstable, and the canonical topology is not a
  V-topology.
\end{itemize}
\begin{theorem}\label{thm:can-dv}
  There is a valuation $\val : \Kk \to \Gamma$ and a derivation $\delta :
  \Kk \to \Kk$ such that for every $a, b \in \Kk$ and $\gamma \in
  \Gamma$, the set
  \begin{equation*}
    \{x \in \Kk : \val(x - a) > \gamma \text{ and } \val(\delta x -
    b) > \gamma\}
  \end{equation*}
  is non-empty, and these sets form a basis for the canonical topology
  on $\Kk$.
\end{theorem}
\begin{proof}
  By resplendence and the uniform definability of the canonical
  topology (Theorem~\ref{thm:uni-def}), we may replace $\Kk$ with an
  elementarily equivalent field $K$.  By
  Propositions~\ref{prop:strong-density} and \ref{prop:dense-dvtop},
  it suffices to produce a dense, lifted diffeovaluation structure on
  $K$ such that the diffeovaluation topology agrees with the canonical
  topology.
  
  Take a magic subfield $k_0 \preceq \Kk$.  By Corollary~4.6 in
  \cite{CKS}, there is a prime $p$ such that the embedding
  \begin{equation*}
    k_0^\times/(k_0^\times)^p \hookrightarrow \Kk^\times/(\Kk^\times)^p
  \end{equation*}
  is an isomorphism, and so $\Kk^\times = k_0^\times \cdot
  (\Kk^\times)^p$.  Thus $\Kk^\times/k_0^\times$ is $p$-divisible.
  
  By Theorem~\ref{thm:what}, there is a $k_0$-linear 2-inflator
  $\varsigma$ on $\Kk$ satisfying the Strong Assumptions of \S
  \ref{sec:canform}--\ref{sec:der}.  By Theorem~\ref{thm:definable},
  its fundamental ring $R$ induces the canonical topology on $\Kk$.
  By Theorem~\ref{thm:one-dir} (and its proof), $\varsigma$ is the
  2-inflator induced by some diffeovaluation data $(\Oo,D,\partial)$,
  and
  \begin{equation*}
    R = \{x \in \Oo : \val(\partial x) \ge 0\}.
  \end{equation*}
  Thus, the canonical topology agrees with the diffeovaluation
  topology.

  Note that $R$ is a $k_0$-algebra, and therefore $k_0 \subseteq R
  \subseteq \Oo$.  So the value group $\Kk^\times/\Oo^\times$ is a
  quotient of $\Kk^\times/k_0^\times$, and is $p$-divisible.

  Let $\Kk^+$ be the expansion of $\Kk$ by the
  diffeovaluation data.  By Theorem~\ref{thm:uni-def}, there is a
  sentence $\sigma$ holding in $\Kk^+$, expressing that
  \begin{itemize}
  \item the diffeovaluation is dense
  \item the value group is $p$-divisible
  \item the diffeovaluation topology agrees with the canonical
    topology (of the reduct).
  \end{itemize}
  Let $K$ be a sufficiently resplendent elementary extension of
  $\Kk^+$.  Then $\sigma$ holds in $K$, and $K$ admits a lifting, by
  Proposition~\ref{prop:lift}.
\end{proof}
Recall from Definition~\ref{def:dv-top} that a \emph{DV-topology} is a
field topology that is ``locally equivalent'' in the sense of
\cite{prestel-ziegler} to a diffeovaluation topology on a dense
diffeovalued field.
\begin{corollary}\label{cor:vdv-top}
  If $K$ is a field of dp-rank 2 and characteristic 0, then one of the
  following holds:
  \begin{itemize}
  \item $K$ is stable.
  \item The canonical topology on $K$ is a V-topology.
  \item The canonical topology on $K$ is a DV-topology.
  \end{itemize}
\end{corollary}
\begin{proof}
  Theorem~\ref{thm:can-dv} and Theorem~\ref{thm:uni-def}.\ref{ud2}.
\end{proof}

\section{Counterexample to the valuation conjecture} \label{sec:grail}
As outlined in \S 10 of \cite{prdf2}, it would be very helpful if the
Valuation Conjecture~\ref{con:vc} were true.  Unfortunately,
algebraically closed dense diffeovalued fields turn out to be a
counterexample, as hinted by Theorem~\ref{thm:can-dv} and
Corollary~\ref{cor:vdv-top}.
\begin{theorem}\label{ctex-0:thm}
  Let $\ADVF$ be the theory of algebraically closed, dense diffeovalued
  fields of residue characteristic 0.  Then $\ADVF$ is consistent,
  complete, unstable, has dp-rank 2, and is not valuation type.
\end{theorem}
This will take some work to prove.  In order to get a cleaner
quantifier elimination result, it helps to work in a slightly
different theory expanding $\ADVF$.
\begin{definition}
  Let $(K,\Oo,\mm)$ be a valued field of residue characteristic 0.
  \begin{itemize}
  \item Let $M$ be an $\Oo$-module.  An \emph{$M$-valued log
    derivation} on $K$ is a group homomorphism
    \begin{equation*}
      \dlog : K^\times \to (M,+)
    \end{equation*}
    such that
    \begin{equation*}
      (x + y) \dlog(x + y) = x \dlog x + y \dlog y
    \end{equation*}
    for $x, y \in \Oo$.
  \item A \emph{truncated log derivation} on $K$ is a log derivation
    taking values in $K/\mm$.
  \end{itemize}
\end{definition}
If $\dlog : K^\times \to (M,+)$ is a log derivation, and we define
$\partial x = x \cdot \dlog x$ for $x \in \Oo$, then $\partial : \Oo
\to M$ is a derivation.

\begin{definition}
  \emph{LDVF} is the theory of $(K,\val,\dlog)$, where
  \begin{itemize}
  \item $(K,\val) \models \ACVF_{0,0}$
  \item $\dlog$ is a truncated log derivation on $K$.
  \item Every fiber of $\dlog$ is dense in $K$, with respect to the
    valuation topology.
  \end{itemize}
\end{definition}

\begin{remark}
  The notion of ``log derivation'' used here is probably related to a
  standard construction of \emph{logarithmic differentials} in log
  geometry.  Specifically, an $M$-valued log derivatiion on $\Oo$ is
  probably the same thing as an $\Oo$-linear morphism
  $\Omega_{\Oo/\Qq}(\log \Gamma_{> 0}) \to M$, where the module of log
  differentials $\Omega_{\Oo/\Qq}(\log \Gamma_{> 0})$ is as defined in
  \S 6.4.14 of \cite{GabberRamero}.  I have not traced through the
  definitions to verify this.
\end{remark}

\subsection{Consistency}
Recall that if $L/K$ is an extension of fields of characteristic 0, if
$V$ is an $L$-vector space, and if $\partial : K \to V$ is a
derivation, then we can extend $\partial$ to a derivation $\partial' :
L \to V$.  In the case where $L = K(t)$ (a pure transcendental
extension), we can arrange for $\partial' t$ to equal any value we
want in $V$.
\begin{lemma}\label{exister:lem}
  There is an algebraically closed dense, lifted diffeovalued field.
  In other words, there is an algebraically closed field $K$ of
  characteristic 0, a non-trivial valuation $\val : K \to \Gamma$ with
  residue characteristic 0, and a derivation $\partial : K \to K$ such
  that for any $a, b \in K$ and $\gamma \in \Gamma$, there is $x \in
  K$ such that
  \begin{align*}
    \val(x - a) & \ge \gamma \\
    \val(\partial x - b) & \ge \gamma.
  \end{align*}
  (Compare with Proposition~\ref{prop:strong-density}.)
\end{lemma}
\begin{proof}
  Choose some extension of the $t$-adic valuation on $\Qq(t)$ to
  $\Qq(t)^{alg}$.  Let $K$ be the completion of $\Qq(t)^{alg}$.  Then
  $K$ is an algebraically closed field with a complete rank 1
  valuation of residue characteristic 0.  Moreover, the valuation
  topology on $K$ is metrizable, separable, and complete.  As $(K,+)$
  is a non-discrete topological group, it has no isolated points.  As $K$ is a
  perfect Polish space, it is uncountable.
  
  Let $\{U_i \times V_i\}_{i \in \Nn}$ be a countable basis of opens
  in $K \times K$.  Let $K_0 \preceq K$ be a countable elementary
  substructure which defines the $U_i$ and $V_i$, and is dense in $K$.
  Recursively choose $t_i, s_i \in K$ such that
  \begin{itemize}
  \item $t_i \in U_i$ and $s_i \in V_i$
  \item $t_i$ is transcendental over $K_0(t_0, t_1, \ldots, t_{i-1})$.
  \end{itemize}
  This is possible because $K_0(t_0, t_1, \ldots, t_{i-1})$ is
  countable, so at least one transcendental $t' \in K$ exists.
  Replacing $t'$ with its inverse, we may arrange for $t' \in \Oo$.
  Then, $K_0$-definability of $U_i$ ensures that there are $a, b \in
  K_0^\times$ such that $a \cdot \Oo + b \subseteq U_i$.  Take $t_i =
  a t' + b$.

  Take the trivial derivation $K_0 \to K$ and extend it successively
  to $K_0(t_0)$, $K_0(t_0,t_1)$, \ldots, arranging for $\partial t_i =
  s_i$.  This determines a derivation $K(t_0,t_1,\ldots) \to K$, which
  we can then extend to a derivation $K \to K$.  The collection of
  $t_i$ witnesses the required density statement.
\end{proof}

\begin{lemma}\label{convert:lem}
  Let $(K,\val,\delta)$ be an algebraically closed, dense, lifted
  diffeovalued field.  Let $\dlog : K^\times \to K/\mm$ be the composition
  \begin{equation*}
    K^\times \to K \twoheadrightarrow K/\mm,
  \end{equation*}
  where the first map is the usual log derivation $x \mapsto (\delta
  x)/x$, and the second map is the quotient map $x \mapsto x + \mm$.
  Then $\dlog$ is a truncated log derivation, and $(K,\val,\dlog)$ is
  a model of LDVF.
\end{lemma}
\begin{proof}
  If we set $\partial x = x \cdot \dlog x$ for $x \in \Oo$, then
  $\partial : \Oo \to K/\mm$ is exactly the composition
  \begin{equation*}
    K \stackrel{\delta}{\to} K \twoheadrightarrow K/\mm.
  \end{equation*}
  Thus $\partial$ is a derivation, and $\dlog$ is a truncated log
  derivation.

  By choice of $(K,\val)$, it is a model of $\ACVF_{0,0}$.  Finally,
  we verify the density axiom.  Given $b \in K$, we must show that the fiber
  \begin{equation*}
    \left\{ x \in K : \frac{\delta x}{x} \in b + \mm\right\}
  \end{equation*}
  is dense in $K$, or equivalently, dense in $K^\times$.  Fix $a \in
  K^\times$ and $\gamma$ in the value group.  By continuity of
  division, there is $\gamma'$ such that for any $x, y \in K$,
  \begin{equation*}
    (\val(x - a) > \gamma' \text{ and } \val(y - ab) > \gamma')
    \implies \val\left( \frac{y}{x} - b \right) > 0.
  \end{equation*}
  By choice of $(K,\val,\delta)$, there is $x$ such that
  \begin{align*}
    \val(x - a) &> \max(\gamma',\gamma) \\
    \val(\delta x - ab) &> \gamma'.
  \end{align*}
  Then
  \begin{align*}
    \val(x - a) &> \gamma \\
    \val\left(\frac{\delta x}{x} - b\right) &> 0.
  \end{align*}
  Thus $x$ is within $\gamma$ of $a$, and $\dlog x = b + \mm$.  So
  $(K,\val,\dlog)$ is a model of LDVF.
\end{proof}
As an immediate corollary,
\begin{proposition} \label{consist:prop}
  The theory LDVF is consistent.
\end{proposition}

\subsection{Calculations in ACVF}
\begin{lemma} \label{interior:lem}
  Let $K$ be a model of $\ACVF_{0,0}$.  Let $S$ be a subset of $K$ and
  $S'$ be a subset of $K/\mm$.  Let $D \subseteq K$ be definable over
  $S \cup S'$.  Then one of the following holds:
  \begin{itemize}
  \item $D$ has interior.
  \item $D$ is finite, and every element is field-theoretically
    algebraic over $S$.
  \end{itemize}
\end{lemma}
\begin{proof}
  Replacing $K$ with an elementary extension, and we may assume that
  $K$ is a monster model.  We may assume $S, S'$ are finite.  For $A
  \subseteq K$, let $A^{alg}$ denote the field-theoretic algebraic
  closure, i.e., the algebraic closure of the subfield generated by
  $A$.

  The swiss cheese decomposition ensures that $D$ has interior unless
  $D$ is finite.  So we may assume $D$ is finite.  Then $D \subseteq
  \acl(S \cup S')$.  Let $\{x_1, \ldots, x_n\}$ enumerate the elements
  of $S'$.  Enlarging $S'$, we may assume $x_1 = 0$.  Let $\pi : K \to
  K/\mm$ be the quotient map.  The fibers of $\pi$ are infinite, and
  thus uncountable, by saturation of the monster.  Therefore we can
  find $y_i, y'_i \in \pi^{-1}(x_i)$ such that
  \begin{align*}
    y_i & \notin (S y_1 y_2 \cdots y_{i-1})^{alg} \\
    y'_i & \notin (S y_1 y_2 \cdots y_n y'_1 y'_2 \cdots y'_{i-1})^{alg}.
  \end{align*}
  Then the sequence $y_1, y_2, \ldots, y'_1, y'_2, \ldots, y'_n$ is a
  sequence of independent transcendentals over $S^{alg}$.

  We arranged for $y_1, y'_1 \in \pi^{-1}(x_1) = \pi^{-1}(0) = \mm$.
  Therefore, $y_1$ and $y'_1$ have non-trivial valuation.  Then
  \begin{align*}
    M := (S y_1 y_2 \cdots y_n)^{alg} & \preceq K \\
    M' := (S y'_1 y'_2 \cdot y'_n)^{alg} & \preceq K,
  \end{align*}
  by model completeness of ACVF.  Note that $x_i \in
  \dcl^{\textrm{eq}}(y_i)$, and so $S' \subseteq
  \dcl^{\textrm{eq}}(M)$.  Then $D$ is $M$-definable, and so $D
  \subseteq M$ because $D$ is finite.  Similarly, $D \subseteq M'$.

  On the other hand, we arranged for the following to hold in the ACF
  reduct:
  \begin{equation*}
    y_1 y_2 \cdots y_n \forkindep_S y_1' y_2' \cdots y_n'.
  \end{equation*}
  Therefore $M \cap M' = S^{alg}$, and so $D \subseteq S^{alg}$.
\end{proof}
\begin{lemma}\label{newton}
  Let $K$ be a model of ACVF.  Let $P(x) = a_n x^n + \cdots + a_1 x +
  a_0$ be a polynomial such that $\min_{i \le n}(\val(a_i)) = 0$.
  Then the number of roots of $P$ in $\Oo$, counted with
  multiplicities, is equal to the largest $i$ such that $\val(a_i) =
  0$.
\end{lemma}
\begin{proof}
  This is a basic statement about Newton polygons.  Let $r_1, \ldots,
  r_n$ be the roots of $P(x)$, counted with multiplicity.  Reordering,
  we may assume $r_1, \ldots, r_m \in \Oo$, and $r_{m+1}, \ldots, r_n
  \notin \Oo$.  Then
  \begin{align*}
    P(x) &= c Q(x) \\
    Q(x) &= \prod_{i = 1}^m (x - r_i) \cdot \prod_{i = m+1}^n (1 - x/r_i).
  \end{align*}
  for some $c \in K^\times$.  Then $Q(x) \in \Oo[x]$, and its
  reduction modulo $\mm$ is
  \begin{equation*}
    \prod_{i = 1}^m (x - \res(r_i)),
  \end{equation*}
  a nonzero polynomial in $k[x]$, of degree $m$.  If we write $Q(x) =
  b_n x^n + \cdots + b_1 x + b_0$, then $\min_{i \le n}(\val(b_i)) =
  0$, and so $\val(c) = 0$.  Then
  \begin{equation*}
    \max \{i \le n : \val(a_i) = 0\} = \max \{i \le n : \val(b_i) = 0\} = m. \qedhere
  \end{equation*}
\end{proof}

Some form of Rolle's theorem holds in models of $\ACVF_{0,0}$:
\begin{lemma}\label{rolle}
  Let $K$ be a model of $\ACVF_{0,0}$.  Let $B$ be a ball.  Let $P(x)$
  be a polynomial in $K[x]$.  If $P$ has two distinct zeros in $B$,
  then $P'(x)$ has a zero in $B$.
\end{lemma}
\begin{proof}
  We may assume $P$ is non-zero; otherwise the result it trivial.  Let
  $r_1, r_2$ be two zeros in $B$.  Shrinking $B$ to the smallest ball
  containing $r_1, r_2$, we may assume $B$ is a closed ball.  Shifting
  everything by an affine transformation, we may assume $B = \Oo$.
  Let $P(x) = a_n x^n + \cdots + a_1 x + a_0$.  Multiplying $P$ by a
  constant from $K^\times$, we may assume $\min_{i \le n} \val(a_i) =
  0$.  The polynomial $P(x)$ has at least two roots in $\Oo$, so by
  Lemma~\ref{newton}, there is some $m \ge 2$ such that $\val(a_m) =
  0$.  Note
  \begin{equation*}
    P'(x) = n a_n x^{n-1} + \cdots + 2 a_2 x + a_1.
  \end{equation*}
  Also, $\val(i a_i) = \val(a_i)$ for $i \ge 1$, because of residue
  characteristic 0.  Therefore, $\val(i a_i) \ge 0$, and $\val(m a_m)
  = 0$.  So the coefficient of $x^{m-1}$ has valuation 0 for some $m
  \ge 2$.  By Lemma~\ref{newton}, $P'(x)$ has at least 2 - 1 roots in
  $\Oo$.
\end{proof}

\subsection{Calculations with log derivations}

\begin{lemma}\label{diffs:lem}
  Let $(K,\Oo,\mm)$ be a valued field, let $M$ be an $\Oo$-module, and
  let $\dlog : K^\times \to M$ be a log derivation.  If $x,y \in K$
  satisfy
  \begin{equation}
    \val(x - y) \le \max(\val(x),\val(y)), \label{underkill}
  \end{equation}
  then $x/(x-y), y/(x-y) \in \Oo$, and
  \begin{equation*}
    \dlog(x-y) = \frac{x}{x-y} \cdot \dlog(x) - \frac{y}{x-y} \cdot
    \dlog(y).
  \end{equation*}
\end{lemma}
\begin{proof}
  First note that
  \begin{equation*}
    \val(x-y) \le \min(\val(x),\val(y)).
  \end{equation*}
  This is automatic if $\val(x) \ne \val(y)$, and equivalent to
  (\ref{underkill}) otherwise.
  
  Recall the derivation $\partial : \Oo \to M$ given by $\partial x =
  x \cdot \dlog x$.  Then
  \begin{align*}
    0 &= \partial(1) = \partial \left(\frac{x-y}{x-y} \right) =
    \partial \left( \frac{x}{x-y} \right) - \partial \left(
    \frac{y}{x-y} \right) \\
    &= \frac{x}{x-y} \dlog \left( \frac{x}{x-y} \right) -
    \frac{y}{x-y} \dlog \left( \frac{y}{x-y} \right) \\
    &= \frac{x}{x-y} \left[ \dlog(x) - \dlog(x-y) \right] -
    \frac{y}{x-y} \left[ \dlog(y) - \dlog(x-y) \right] \\
    &= \left( \frac{x}{x-y} \cdot \dlog(x) - \frac{y}{x-y} \cdot
    \dlog(y) \right) - \left( \frac{x}{x-y} - \frac{y}{x-y} \right) \dlog(x-y) \\
    &= \left( \frac{x}{x-y} \cdot \dlog(x) - \frac{y}{x-y} \cdot
    \dlog(y) \right) - \dlog(x-y). \qedhere
  \end{align*}
\end{proof}

\begin{proposition}\label{alg-case}
  Let $(K,\Oo)$ be an algebraically closed field with a log derivation
  $\dlog : K \to M$, for some $\Oo$-module $M$.  Suppose that $\dlog$
  vanishes on some subfield $F \subseteq K$, and $K$ is algebraic over
  $F$ (so that $K = F^{alg}$).  Then $\dlog$ vanishes on $K$.
\end{proposition}
\begin{proof}
  Increasing $F$, we may assume $F$ is maximal among subfields on
  which $\dlog$ vanishes.  Suppose for the sake of contradiction that
  $F \subsetneq K$.  Take minimal $n > 1$ such that $F$ has a finite
  extension of degree $n$.  If $P(x) \in F[x]$ has degree $\le n$,
  then one of the following happens:
  \begin{itemize}
  \item $P(x)$ factors into linear polynomials
  \item $P(x)$ is irreducible of degree $n$.
  \end{itemize}
  \begin{claim} \label{factors}
    If $a \in K$ and $[F(a) : F] = n$, then
    \begin{itemize}
    \item $F(a)^\times$ is generated by $F^\times$ and the elements $a
      - b$ with $b \in F$.
    \item There is $b \in F$ such that $\dlog(a - b) \ne 0$.
    \end{itemize}
  \end{claim}
  \begin{claimproof}
    Every element of $F(a)$ is of the form $P(a)$ for some polynomial
    $P(x) \in F[x]$ of degree less than $n$.  Then $P$ splits into
    linear factors, so
    \begin{equation*}
      P(a) = c(a-b_1)(a-b_2)\cdots(a-b_n)
    \end{equation*}
    for some $c, b_1, b_2, \ldots, b_n \in F$.  This proves the first
    point.  If $\dlog(a - b) = 0$ for all $b \in F$, then $\dlog$ must
    vanish on $F(a)^\times$, contradicting the choice of $F$.
  \end{claimproof}
  Take an arbitrary extension $L/F$ of degree $n$, and break into
  cases:
  \begin{itemize}
  \item If $\val(L)$ is strictly larger than $\val(F)$, take $\gamma
    \in \val(L) \setminus \val(F)$.  The inequality
    \begin{equation*}
      |\val(L)/\val(F)| \le [L : F]
    \end{equation*}
    implies that $m \gamma \in \val(F)$ for some $m \le n$.  Take $c
    \in F$ with $\val(c) = m \gamma$.  The polynomial $x^m - c$ has no
    roots in $F$, so $m = n$ and $x^n - c$ is irreducible.  Take $a
    \in K$ such that $a^n = c$.  Note that
    \begin{equation*}
      \dlog(a) = (1/n)\dlog(c) = 0,
    \end{equation*}
    because $c \in F$ and the residue characteristic is 0.

    By Claim~\ref{factors}, there is $b
    \in F(a)$ such that $\dlog(a - b) \ne 0$.  Now $\val(b) \ne \gamma =
    \val(a)$, by choice of $\gamma$, and so
    \begin{equation*}
      \val(a-b) = \min(\val(a),\val(b)).
    \end{equation*}
    Also $\dlog(a) = \dlog(b) = 0$.  By Lemma~\ref{diffs:lem},
    $\dlog(a-b) = 0$, a contradiction.
  \item If $\res(L)$ is strictly larger than $\res(F)$, take $\alpha
    \in \res(L) \setminus \res(F)$.  The inequality
    \begin{equation*}
      [\res(L) : \res(F)] \le [L : F]
    \end{equation*}
    implies that $[\res(F)(\alpha) : \res(F)] \le n$.  Let
    \begin{equation*}
      x^m + \beta_{m-1} x^{m-1} + \cdots + \beta_1 x + \beta_0
    \end{equation*}
    be the monic irreducible polynomial of $\alpha$ over $\res(F)$.
    Because of residue characteristic 0, this polynomial is separable,
    and so
    \begin{equation}
      m \alpha^{m-1} + (m-1) \beta_{m-1} \alpha^{m-2} + \cdots + 2
      \beta_2 \alpha + \beta_1 \ne 0. \label{simple-1}
    \end{equation}
    Take $b_i \in F$ with $\res b_i = \beta_i$, and let $P(x)$ be the
    polynomial
    \begin{equation*}
      x^m + b_{m-1} x^{m-1} + \cdots + b_1 x + b_0 \in F[x].
    \end{equation*}
    Then $P(x)$ is irreducible, and so $m = [\res(F)(\alpha) :
      \res(F)] = n$.  Let $a \in K$ be the root of $P(x)$ with
    $\res(a) = \alpha$.  Then
    \begin{equation*}
      a^n + b_{n-1} a^{n-1} + \cdots + b_1 a + b_0 = 0.
    \end{equation*}
    Applying the derivation $\partial : \Oo \to M$, which vanishes on
    the $b_i$, we obtain
    \begin{equation*}
      (n a^{n-1} + (n-1) b_{n-1} a^{n-2} + \cdots + 2 b_2 a + b_1 ) \partial a = 0.
    \end{equation*}
    The expression inside the parentheses has nonzero residue, by
    (\ref{simple-1}), and so it is an element of $\Oo^\times$.
    Therefore $\partial a = 0$.  Now $\res(a) = \alpha \notin
    \res(F)$, so $\res(a) \ne 0$ and $a$ is invertible as well.
    Therefore $\dlog a = (\partial a)/a = 0$.

    By Claim~\ref{factors}, there is some $b \in F$ such that $\dlog(a
    - b) \ne 0$.  Then
    \begin{equation*}
      \val(a - b) \le \max(\val(a),\val(b)).
    \end{equation*}
    (Otherwise, $\res(a) = \res(b) \in F$, contradicting the choice of
    $a$ and $\alpha$.)  By Lemma~\ref{diffs:lem}, $\dlog(a - b) = 0$,
    a contradiction.
  \item Lastly, suppose that $L/F$ is an immediate extension.  By
    maximality of $F$, there is $a \in L$ with $\dlog(a) \ne 0$.  Let
    $\mathcal{C}$ be the collection of balls containing $a$, with center and radius from $F$.  Let $I$ be the intersection $\bigcap \mathcal{C}$.  As usual, $I
    \cap F = \emptyset$.  (Suppose $b \in I \cap F$.  Then $rv(a-b) =
    rv(b'-b)$ for some $b' \in F$, because the extension is immediate.
    The ball centered around $b'$ of radius $\val(a-b')$ does not
    contain $b$, contradicting the choice of $b$.)

    Let $P(x)$ be the minimal polynomial of $a$ over $F$.  Then $P(x)$
    has degree $n$.  Let $a_1, \ldots, a_n$ be the roots of $P(x)$,
    with $a_1 = a$.  Note that $P'(x)$ has degree $n - 1$, and
    therefore splits over $F$.  So no root of $P'(x)$ is in $I$.  By
    Rolle's Theorem (Lemma~\ref{rolle}), $a$ is the unique root of
    $P(x)$ in $I$.

    Therefore $I$ has empty intersection with the finite set
    $\{0,a_2,\ldots,a_n\}$.  We can find $b \in F$ such that
    \begin{equation*}
      \val(a - b) > \max(\val(0 - b),\val(a_2 - b),\val(a_3 -
      b),\ldots,\val(a_n - b)).
    \end{equation*}
    Take $c \in F$ with $\val(a - b) = \val(c)$, and let $e_i = (a_i -
    b)/c$.  Then $\val(e_1) = 0$, and $\val(e_i) < 0$ for $i > 1$.
    The $e_i$ are the roots of the irreducible polynomial
    \begin{equation*}
      Q(x) = P(cx + b) = s_n x^n + s_{n-1} x^{n-1} + \cdots + s_1 x +
      s_0 \in F[x].
    \end{equation*}
    By Newton polygons, $\val(s_0) = \val(s_1) < \val(s_i)$ for $i >
    1$.  Then we can apply $\partial$ to the equation
    \begin{equation*}
      (s_n/s_1) e_1^n + \cdots + (s_2/s_1) e_1^2 + e_1 + (s_0/s_1) =
      0,
    \end{equation*}
    and obtain
    \begin{equation*}
      (n (s_n/s_1) e_1^{n-1} + \cdots + 2 (s_2 / s_1) e_1 + 1) \partial e_1 = 0,
    \end{equation*}
    because the coefficients $s_n/s_1$ lie in $F$, where $\partial$ vanishes.  But
    the expression in parentheses has valuation 0, because $e_1 \in
    \Oo$ and $s_i/s_1 \in \mm$ for $2 \le i \le n$.  Therefore
    $\partial e_1 = 0$.  As $\val(e_1) = 0$, we have $e_1 \in
    \Oo^\times$ as well, and then $\dlog e_1 = (\partial e_1)/e_1 =
    0$.  Then $\dlog (a - b) = \dlog e_1 + \dlog c = 0$, as $c \in F$.
    Finally,
    \begin{equation*}
      \val(a - b) > \val(b) = \val(a),
    \end{equation*}
    and so $\dlog(a-b)$ and $\dlog(b)$ determine $\dlog(a)$, by
    Lemma~\ref{diffs:lem}.  Thus $\dlog(a) = 0$, contradicting the
    choice of $a$. \qedhere
  \end{itemize}
\end{proof}
Proposition~\ref{alg-case} is probably a consequence of Lemma~6.5.12
and Claim~6.5.14 in \cite{GabberRamero}, but I am not entirely
certain.

\subsection{Quantifier elimination and completeness}
Let $\mathcal{L}_0$ be the language for $\ACVF_{0,0}$ with two sorts,
$K$ and $K/\mm$, and the following functions and relations:
\begin{itemize}
\item The field operations on $K$, including the constants 0, 1, \emph{and
  division}.
\item The $\Oo$-module structure on $K/\mm$, i.e., the group structure
  (including 0 and negation) and the multiplication map
  \begin{equation*}
    \Oo \times K/\mm \to K/\mm,
  \end{equation*}
  understood as a partial function on $K \times K/\mm$.
\item All $\emptyset$-definable relations on $K$ and $K/\mm$.
\end{itemize}
Then $\ACVF_{0,0}$ has quantifier elimination in $\mathcal{L}_0$,
because we Morleyized.  If $K \models \ACVF_{0,0}$, an
$\mathcal{L}_0$-substructure of $K$ consists of a pair $(F,D)$, where
\begin{itemize}
\item $F$ is a subfield of $K$.
\item $D$ is an $\Oo_F$-submodule of $K/\mm$.
\end{itemize}
Note that $D$ need not contain the image of $F$ under $K
\twoheadrightarrow K/\mm$, as we did not include this map as one of
the functions in the signature.

Let $\mathcal{L}$ be the language for $\LDVF$ obtained by expanding
$\mathcal{L}_0$ with a function symbol for the map $\dlog : K^\times
\to K/\mm$.  If $K$ is a model of $\LDVF$, then an
$\mathcal{L}$-substructure is a pair $(F,D)$, where
\begin{itemize}
\item $F$ is a subfield of $K$
\item $D$ is an $\Oo_F$-submodule of $K/\mm$
\item $D$ contains $\dlog x$ for $x \in F$.
\end{itemize}

\begin{lemma}\label{qe:lem}
  Let $K$ be a model of $\LDVF$.  Let $K'$ be a $|K|^+$-saturated
  model of $\LDVF$.  Let $(F,D)$ be a proper $\mathcal{L}$-substructure of $K$.  Let
  $f : (F,D) \hookrightarrow K'$ be an $\mathcal{L}$-embedding (an isomorphism onto
  a substructure of $K'$).  Then $f$ can be extended to an $\mathcal{L}$-embedding
  $f' : (F',D') \hookrightarrow K'$ for some strictly larger $\mathcal{L}$-substructure $(F',D')$.
\end{lemma}
\begin{proof}
  First suppose $D < K/\mm$.  Note that $(F,K/\mm)$ is an
  $\mathcal{L}$-substructure of $K$.  By quantifier elimination of
  $\ACVF_{0,0}$ in the language $\mathcal{L}_0$, we can extend $f$ to
  an $\mathcal{L}_0$-embedding $f' : (F,K/\mm) \hookrightarrow K'$.
  Then $f'$ is already an $\mathcal{L}$-embedding, because
  \begin{equation*}
    f'(\dlog(x)) = f(\dlog(x)) = \dlog(f(x)) = \dlog(f'(x))
  \end{equation*}
  for any $x \in F$.  (The first equation holds because $x \in F
  \implies \dlog(x) \in D$, and $f'$ extends $f$ on $D$.)

  So we may assume $D = K/\mm$, and $F < K$.
  \begin{claim}\label{delta-helper}
    If $F'$ is a subfield of $K$ containing $F$, and $f' : (F',K/\mm)
    \hookrightarrow K'$ is an $\mathcal{L}_0$-embedding extending $f$,
    then
    \begin{itemize}
    \item $f'$ induces a map from the valuation ring of $F'$ to the
      valuation ring of $K'$, and so we can regard $K'/\mm'$ as a
      module over the valuation ring of $F'$.
    \item If $\Delta : F' \to K'/\mm'$ is defined by
      \begin{equation*}
        \Delta(x) = f'(\dlog(x)) - \dlog(f'(x)),
      \end{equation*}
      then $\Delta$ is a log derivation $F' \to K/\mm$.
    \item $\Delta$ vanishes on $F$.
    \item If $\Delta$ vanishes on $F'$, then $f'$ is an $\mathcal{L}$-embedding.
    \end{itemize}
  \end{claim}
  \begin{claimproof}
    The first point is clear, since $f'$ is a partial elementary map
    in the ACVF reduct.  The second point is a direct calculation:
    \begin{align*}
      \Delta(xy) &= f'(\dlog(xy)) - \dlog(f'(xy)) \\ &= f'(\dlog(x) +
      \dlog(y)) - \dlog(f'(x)f'(y)) \\ &= f'(\dlog(x)) + f'(\dlog(y)) -
      \dlog(f'(x)) - \dlog(f'(y)) \\ & = \Delta(x) + \Delta(y) \\ (x+y)
      \Delta(x + y) &= f'(x+y) \Delta(x+y) \\ & = f'(x+y) f'(\dlog(x+y)) -
      f'(x+y) \dlog(f'(x+y)) \\ &= f'((x+y) \dlog(x+y)) - (f'(x) +
      f'(y)) \dlog(f'(x) + f'(y)) \\ &= f'(x \dlog x + y \dlog y) - f'(x)
      \dlog(f'(x)) - f'(y) \dlog(f'(y)) \\ &= f'(x) f'(\dlog x) + f'(y)
      f'(\dlog y) - f'(x) \dlog(f'(x)) - f'(y) \dlog(f'(y)) \\ &= f'(x)
      \Delta(x) + f'(y) \Delta(y).
    \end{align*}
    The third point expresses that $f$ is an $\mathcal{L}$-embedding.
    The fourth point is clear.
  \end{claimproof}

  Next suppose $F \ne F^{alg}$.  By quantifier elimination of
  $\ACVF_{0,0}$, we can extend $f$ to an $\mathcal{L}_0$-embedding $f'
  : (F^{alg},K/\mm) \hookrightarrow K'$.  Let $\Delta : F^{alg} \to
  K'/\mm'$ be as in Claim~\ref{delta-helper}.  Then $\Delta$ vanishes
  on $F$, and therefore on $F^{alg}$, by Proposition~\ref{alg-case}.
  Therefore $f'$ is an $\mathcal{L}$-embedding.

  Finally, suppose that $F = F^{alg}$.  Take a transcendental $a \in K
  \setminus F$.  Let $\vec{s}$ be an infinite tuple enumerating $F$,
  and $\vec{t}$ be an infinite tuple enumerating $K/\mm$.  Let
  $\Sigma(x;\vec{y};\vec{z})$ be the complete $\mathcal{L}_0$-type of
  $(a;\vec{s};\vec{t})$.  Note that $a' \in K'$ satisfies
  $\Sigma(x;f(\vec{s});f(\vec{t}))$, if and only if there is an
  $\mathcal{L}_0$-embedding $f' : (F(a),K/\mm) \hookrightarrow K'$
  extending $f$ and sending $a \mapsto a'$.

  For any $b \in F$, let $\psi_b(x)$ be the
  type in $K'$ asserting that
  \begin{equation*}
    \dlog(x - f(b)) = f(\dlog(a - b)).
  \end{equation*}
  (The right hand side makes sense, because we arranged $D = K/\mm$.)
  \begin{claim}\label{duck1}
    For any $b \in F$, the type $\Sigma(x;f(\vec{s});f(\vec{t})) \cup
    \{\psi_b(x)\}$ is realized in $K'$.
  \end{claim}
  \begin{claimproof}
    By saturation, it suffices to show finite satisfiability.  Suppose
    $\varphi(x;\vec{y};\vec{z})$ is an $\mathcal{L}_0$-formula satisfied
    by $(a;\vec{s};\vec{t})$.  We must find $a' \in K'$ satisfying
    \begin{equation*}
      \varphi(a';f(\vec{s});f(\vec{t})) \wedge \psi_b(a')
    \end{equation*}
    The definable set $\varphi(K;\vec{s};\vec{t})$) has interior, by
    Lemma~\ref{interior:lem} and transcendence of $a$ over $\vec{s}$.
    As $f$ is a partial elementary map in the ACVF reduct, the
    definable set $\varphi(K';f(\vec{s});f(\vec{t}))$ has interior as
    well.  By the density axiom of LDVF, there is $x \in K'$ such
    that
    \begin{align*}
      x + f(b) &\in \varphi(K';f(\vec{s});f(\vec{t})) \\
      \dlog(x) &= f(\dlog(a - b)).
    \end{align*}
    Take $a' = x + f(b)$.
  \end{claimproof}
  \begin{claim}\label{duck2}
    The type $\Sigma(x;f(\vec{s});f(\vec{t})) \cup \{\psi_b(x) : b \in
    F\}$ is realized in $K'$.
  \end{claim}
  \begin{proof}
    By saturation, it suffices to show finite satisfiability.  Let
    $b_1, \ldots, b_n$ be elements of $F$.  We claim that the type
    \begin{equation*}
      \Sigma(x;f(\vec{s});f(\vec{t})) \cup \{\psi_{b_1}(x), \ldots, \psi_{b_n}(x)\}
    \end{equation*}
    is realized in $K'$.  Without loss of generality,
    \begin{equation*}
      \val(a - b_1) \ge \val(a - b_2) \ge \cdots \ge \val(a - b_n).
    \end{equation*}
    By Claim~\ref{duck1} there is $a'$ realizing
    $\Sigma(x;f(\vec{s});f(\vec{t})) \cup \{\psi_{b_1}(x)\}$.  Let $f'
    : (F(a),K/\mm) \to K'$ be the $\mathcal{L}_0$-embedding extending
    $f$ and sending $a$ to $a'$.  Note that
    \begin{equation*}
      \dlog(f'(a - b_1)) = \dlog(f'(a) - f'(b_1)) = \dlog(a' - f(b_1))
      \stackrel{\ast}{=} f(\dlog(a - b_1)) = f'(\dlog(a-b_1)).
    \end{equation*}
    The starred equation holds because of $\psi_{b_1}(a')$.  By
    Claim~\ref{delta-helper}, there is a log derivation $F(a) \to
    K'/\mm'$ given by
    \begin{equation*}
      \Delta(x) := f'(\dlog(x)) - \dlog(f'(x)).
    \end{equation*}
    Then $\Delta(a - b_1) = 0$.  Also, $\Delta(b_i - b_1) = 0$ for any
    $i$, because $b_i - b_1 \in F$.  Moreover,
    \begin{equation*}
      \val((a - b_1) - (b_i - b_1)) = \val(a - b_i) \le \val(a - b_1),
    \end{equation*}
    and so $\Delta(a - b_i) = 0$, by Lemma~\ref{diffs:lem}.  Then for
    each $i$,
    \begin{equation*}
      \dlog(a' - f(b_i)) = \dlog(f'(a) - f'(b_i)) = \dlog(f'(a - b_i))
      = f'(\dlog(a - b_i)) = f(\dlog(a - b_i)).
    \end{equation*}
    Therefore $\psi_{b_i}(a')$ holds.
  \end{proof}
  Using Claim~\ref{duck2}, take $a' \in K'$ realizing
  \begin{equation*}
    \Sigma(x;f(\vec{s});f(\vec{t})) \cup \{\psi_b(x) : b \in F\}.
  \end{equation*}
  Let $f' : (F(a),K/\mm) \hookrightarrow K'$ be the
  $\mathcal{L}_0$-embedding extending $f$ and mapping $a$ to $a'$.
  Let $\Delta(x) = f'(\dlog(x)) - \dlog(f'(x))$ as in
  Claim~\ref{delta-helper}.  The statement $\psi_b(a')$ implies that
  \begin{equation*}
    \dlog(a' - f(b)) = f(\dlog(a - b)).
  \end{equation*}
  Therefore
  \begin{equation*}
    \dlog(f'(a - b)) = \dlog(f'(a) - f'(b)) = \dlog(a' - f(b)) =
    f(\dlog(a - b)) = f'(\dlog(a - b)).
  \end{equation*}
  So $\Delta(a - b) = 0$ for any $b \in F$.  As $F$ is algebraically
  closed and $a$ is transcendental, the multiplicative group
  $F(a)^\times$ is generated by
  \begin{equation*}
    F^\times \cup \{a - b : b \in F\}.
  \end{equation*}
  Then $\Delta$ must vanish on $F(a)$, because it vanishes on the
  generators.  By Claim~\ref{delta-helper}, $f'$ is an
  $\mathcal{L}$-embedding.
\end{proof}
  
\begin{theorem}
  $\LDVF$ has quantifier elimination in the language $\mathcal{L}$.
\end{theorem}
\begin{proof}
  This follows from Lemma~\ref{qe:lem}, by well-known model-theoretic
  techniques.
\end{proof}

\begin{corollary}
  $\LDVF$ is complete.
\end{corollary}
\begin{proof}
  LDVF is consistent by Proposition~\ref{consist:prop}.  Take two
  models $M_1$ and $M_2$, viewed as $\mathcal{L}$-structures.  Note
  that $(\Qq,0)$ is an $\mathcal{L}$-substructure of $M_i$, for each
  $i$.  The identity map $(\Qq,0) \to (\Qq,0)$ is an isomorphism of
  $\mathcal{L}$-structures:
  \begin{itemize}
  \item For the $\mathcal{L}_0$-structure, this holds because
    $\ACVF_{0,0}$ is complete.
  \item For the map $\dlog$, this holds because $\dlog$ is trivial on
    both copies of $\Qq$.
  \end{itemize}
  By quantifier elimination, $M_1 \equiv M_2$.
\end{proof}

\begin{corollary}\label{ADVF}
  Let $\ADVF$ be the theory of algebraically closed dense diffeovalued
  fields $(K,\partial,D)$.
  \begin{enumerate}
  \item If $(K, \dlog)$ is a model of $\LDVF$, then
    $(K,\partial,K/\mm)$ is a model of $\ADVF$.
  \item \label{point-2} Up to elementary equivalence, every model of
    $\ADVF$ arises in this way.
  \item $\ADVF$ is complete.
  \end{enumerate}
\end{corollary}
\begin{proof}
  It suffices to prove Point~\ref{point-2}.  Let $(K,\partial,\val)$
  be a model of $\ADVF$.  Passing to a resplendent elementary extension,
  we may assume that a lifting exists, by Proposition~\ref{prop:lift}.
  So we obtain a lifted diffeovalued field $(K,\delta,\val)$.  Define
  \begin{align*}
    \dlog : K^\times &\to K/\mm \\
    \dlog(x) &= \frac{\delta x}{x} + \mm.
  \end{align*}
  Then $(K,\dlog,\val) \models \LDVF$, by Lemma~\ref{convert:lem}.
  The derivation $\partial : \Oo \to K/\mm$ determined by $\dlog$ is
  the original derivation $\partial$.
\end{proof}

\subsection{Upper bound on dp-rank}
Let $T$ be the theory of dense, lifted diffeovalued fields.  By
Proposition~\ref{prop:strong-density}, a model of $T$ is a field $K$
with a non-trivial valuation (of residue characteristic 0) and a
derivation $\delta : K \to K$ such that for any $a, b \in K$ and
$\gamma$ in the value group, there is $x \in K$ such that
\begin{align*}
  \val(x - a) & \ge \gamma \\
  \val(\delta x - b) & \ge \gamma.
\end{align*}
The theory $T$ probably has no nice properties, other than being
consistent.

\begin{remark}\label{expand-to-t}
  Lemma~\ref{convert:lem} says that if $(K,\delta,\val) \models T$, and
  we define
  \begin{equation*}
    \dlog(x) := \frac{\delta x}{x} + \mm,
  \end{equation*}
  then $(K,\dlog,\val) \models \LDVF$.  As LDVF is complete, it
  follows that every sufficiently resplendent model of LDVF can be
  expanded to a model of $T$.
\end{remark}

Fix some one-sorted language for $T$.  Say that $\varphi(\vec{x})$ is a
\emph{ACVF-formula} if $\varphi(\vec{x})$ is defined in the ACVF-reduct,
and similarly for \emph{LDVF-formulas}.  Let $\tp_{\ACVF}(\vec{a}/B)$
be the set of all ACVF-formulas with parameters in $B$, satisfied by
$\vec{a}$.  Define $\tp_{\LDVF}(\vec{a}/B)$ similarly.

If $M \models T$ and $\vec{a}$ is a tuple in $M$, define
$\delta(\vec{a})$ coordinatewise.

\begin{lemma}
  Let $M_1, M_2$ be two models of $T$.  Let $\vec{a}$ be a tuple in
  $M_1$ and $\vec{b}$ be a tuple of the same length in $M_2$.
  Suppose that $\tp_{ACVF}(\vec{a}\delta(\vec{a})/\emptyset) =
  \tp_{ACVF}(\vec{b}\delta(\vec{b})/\emptyset)$.  Then
  $\tp_{LDVF}(\vec{a}/\emptyset) = \tp_{LDVF}(\vec{b}/\emptyset)$.
\end{lemma}
\begin{proof}
  After replacing $M_1$ and $M_2$ with elementary extensions, there is
  an isomorphism of the ACVF-reducts
  \begin{equation*}
    f : (M_1,\val) \to (M_2,\val)
  \end{equation*}
  such that $f(\vec{a}) = \vec{b}$ and $f(\delta(\vec{a})) =
  \delta(\vec{b})$.  With $\mathcal{L}_0$ as in the previous section,
  this induces an isomorphism of $\mathcal{L}_0$-structures
  \begin{equation*}
    f' : (M_1,M_1/\mm^{M_1}) \to (M_2,M_2/\mm^{M_2}).
  \end{equation*}
  Restricting the first sort, we obtain an $\mathcal{L}_0$-isomorphism
  \begin{equation*}
    f'' : (\Qq(\vec{a}),M_1/\mm^{M_1}) \to (\Qq(\vec{b}),M_2/\mm^{M_2}),
  \end{equation*}
  because $f(\vec{a}) = \vec{b}$.  We claim that $f''$ preserves
  $\dlog$.  If $c$ is in $\Qq[\vec{a}]$, then
  \begin{equation*}
    c = P(a_1,\ldots,a_n)
  \end{equation*}
  for some $P(x_1,\ldots,x_n) \in \Qq[x_1,\ldots,x_n]$.  Then
  \begin{equation*}
    \delta(c) = \sum_{i = 1}^n \frac{\partial P}{\partial x_i} (\vec{a})
    \cdot \delta(a_i).
  \end{equation*}
  Since $f$ sends $\delta(a_i)$ to $\delta(b_i)$,
  \begin{equation*}
    f(\delta(c)) = \sum_{i = 1}^n \frac{\partial P}{\partial x_i}
    (\vec{b}) \cdot \delta(b_i) = \delta(P(\vec{b})) = \delta(f(c)).
  \end{equation*}
  Then $f(\delta(c)) = \delta(f(c))$, implying that $f'(\dlog(c)) =
  \dlog(f'(c))$.  More generally, if $c \in \Qq(\vec{a})$, then $c =
  c_1/c_2$ for $c_i \in \Qq[\vec{a}]$, and
  \begin{align*}
    f'(\dlog(c)) &= f'(\dlog(c_1) - \dlog(c_2)) = f'(\dlog(c_1)) -
    f'(\dlog(c_2)) \\ & = \dlog(f'(c_1)) - \dlog(f'(c_2)) =
    \dlog(f'(c_1)/f'(c_2)) \\ & = \dlog(f'(c_1/c_2)) = \dlog(f'(c)).
  \end{align*}
  Thus $f'(\dlog(c)) = \dlog(f'(c))$ for $c \in \Qq(\vec{a})$, and
  $f''$ is an $\mathcal{L}$-isomorphism.  By quantifier elimination of
  LDVF in the language $\mathcal{L}$, it follows that
  $\tp_{\LDVF}(\vec{a}/\emptyset) = \tp_{\LDVF}(\vec{b}/\emptyset)$.
\end{proof}
\begin{lemma}\label{formula-convert}
  For every LDVF-formula $\varphi(\vec{x})$, there is an ACVF-formula
  $\psi(\vec{x};\vec{y})$ such that
  \begin{equation*}
    T \vdash \varphi(\vec{x}) \iff \psi(\vec{x};\delta(\vec{x})).
  \end{equation*}
\end{lemma}
\begin{proof}
  A standard compactness argument.
\end{proof}

\begin{theorem}
  If $(K,\dlog,\val) \models \LDVF$, then $\dpr(K) \le 2$.
\end{theorem}
\begin{proof}
  Suppose there is an ict-pattern of depth 3:
  \begin{align*}
    & \varphi_0(x;\vec{b}_{0,0}), \varphi_0(x;\vec{b}_{0,1}), \ldots \\
    & \varphi_1(x;\vec{b}_{1,0}), \varphi_0(x;\vec{b}_{1,1}), \ldots \\
    & \varphi_2(x;\vec{b}_{2,0}), \varphi_0(x;\vec{b}_{2,1}), \ldots
  \end{align*}
  Replacing $K$ with an elementary extension, we may assume that $K$
  can be expanded to a model of $T$, by Remark~\ref{expand-to-t}.
  Applying Lemma~\ref{formula-convert}, we obtain an ict-pattern of
  depth 3,
  \begin{align*}
    & \varphi'_0(x,y;\vec{c}_{0,0}), \varphi'_0(x,y;\vec{c}_{0,1}), \ldots \\
    & \varphi'_1(x,y;\vec{c}_{1,0}), \varphi'_0(x,y;\vec{c}_{1,1}), \ldots \\
    & \varphi'_2(x,y;\vec{c}_{2,0}), \varphi'_0(x,y;\vec{c}_{2,1}), \ldots
  \end{align*}
  made of ACVF-formulas.  But in ACVF, the set $K^2$ has dp-rank less than 3.
\end{proof}

\subsection{Non-valuation type}
\begin{lemma}\label{yatir:lem}
  If $R$ is an interpretable integral domain in some structure, and $K
  = \Frac(R)$, then $\dpr(K) = \dpr(R)$.
\end{lemma}
\begin{proof}
  Work in a monster model.  The inequality $\dpr(R) \le \dpr(K)$ is
  clear.  Conversely, suppose there is an ict-pattern of depth
  $\kappa$ in $K$.  Then there are formulas
  $\{\varphi_\alpha(x;y_\alpha)\}_{\alpha < \kappa}$, coefficients
  $\{b_{\alpha,i}\}_{\alpha < \kappa,~i < \omega}$, and witnesses
  $\{a_\eta\}_{\eta : \kappa \to \omega}$ in $K$, such that
  \begin{equation*}
    \varphi_\alpha(a_\eta,b_{\alpha,i}) \iff \eta(\alpha) = i.
  \end{equation*}
  For any $x_1, \ldots, x_n \in K$, we can find a non-zero common
  denominator $s \in R$ such that
  \begin{equation*}
    \{x_1s, \ldots, x_ns\} \subseteq R.
  \end{equation*}
  By saturation, we can find some non-zero $s \in R$ such that $s
  a_\eta$ lies in $R$ for every $\eta$.  Let $a'_\eta = s a_\eta$, and
  let $\psi_\alpha(x,y,z)$ be the formula
  \begin{equation*}
    \psi_\alpha(x,y,z) \equiv \varphi_\alpha(x/y,z).
  \end{equation*}
  Then there is an ict-pattern of depth $\kappa$ in $R$:
  \begin{equation*}
    \psi_\alpha(a'_\eta,s,b_{\alpha,i}) \iff \eta(\alpha) = i. \qedhere
  \end{equation*}
\end{proof}
\begin{theorem}
  If $(K,\dlog,\val) \models \LDVF$, then $K$ is unstable, not of
  valuation type, and has dp-rank exactly 2.
\end{theorem}
\begin{proof}
  Consider the reduct $(K,\partial,\val)$, a dense diffeovalued field
  by Corollary~\ref{ADVF}.  Let $R$ be the ring
  \begin{equation*}
    R = \{x \in \Oo : \partial x \in \Oo/\mm\},
  \end{equation*}
  as in \S\ref{sec:dtop}.  By Proposition~\ref{prop:dv-top}, the sets
  $\{a R : a \in K^\times\}$ form a neighborhood basis for a
  definable non-trivial Hausdorff topology.  So certainly $K$ is
  unstable.  Also $\Frac(R) = K$, so $\dpr(R) = \dpr(K)$ by
  Lemma~\ref{yatir:lem}.  Therefore $R - R = R$ is a
  basic neighborhood in the canonical topology.  So the canonical
  topology is finer than the diffeovaluation topology.  If the
  canonical topology is a V-topology, so is every coarsening, by
  Theorem 3.2 in \cite{prestel-ziegler}.  But the diffeovaluation
  topology is not a V-topology, by Proposition~\ref{not-a-v}.  So
  $(K,\dlog,\val)$ does not have valuation type.  Then $\dpr(K) \ne
  1$, because dp-minimal fields have valuation type (Theorem~9.3.28 in
  \cite{myself}).  Therefore, $K$ has dp-rank 2.
\end{proof}
\begin{remark}
  The same argument applies to intermediate reducts between the full
  model of $\LDVF$, and the reduct $(K,+,\cdot,R)$.

  The structure $(K,+,\cdot,\Oo,R)$ is similar to Example~7.1 in
  \cite{hhj-v-top}: both are NIP valued fields in which some infinite
  definable set has empty interior.  (In our case, the set is $R$.)
  Unlike \cite{hhj-v-top}, our example is not a \emph{pure} valued
  field, but an expansion.
\end{remark}

\section{Concluding remarks}
The example of \S\ref{sec:grail} derails some promising strategies to
attack the Shelah conjecture and henselianity conjecture.  For
example, it disproves our ``valuation conjecture''
(Conjecture~\ref{con:vc}), which would have implied the Shelah
conjecture (\cite{prdf2}, Theorem~9.9).

Consider the even simpler conjecture:
\begin{conjecture}\label{yeah-right}
  If $(K,+,\cdot,\Oo,\ldots)$ is a dp-finite valued field, and $S
  \subseteq K$ is a definable set of full dp-rank ($\dpr(S) =
  \dpr(K)$), then $S$ has non-empty interior.
\end{conjecture}
Conjecture~\ref{yeah-right} would imply the Henselianity conjecture
for dp-finite fields, by Theorem~7.5 in \cite{hhj-v-top}.  To the best
of my knowledge, all the known results on the henselianity conjecture
use this strategy.

However, the theory $\LDVF$ of \S\ref{sec:grail} is a counterexample
to Conjecture~\ref{yeah-right}.  Indeed, the definable set $R$ has
full rank, but empty interior with respect to $\Oo$.

Conjecture~\ref{yeah-right} probably holds for \emph{pure} valued
fields $(K,+,\cdot,\Oo)$, but this is probably impossible to prove
without first classifying dp-finite valued fields using some other
strategy.  (The purity assumption is hard to use in proofs.)

It seems we need a new strategy to attack the dp-finite Shelah and
henselianity conjectures.  Perhaps the analysis of
\S\ref{sec:isotyp}--\ref{sec:der} can be extended to higher ranks.

Here is a conjectural sketch.  For any $n \ge 1$, there should be a
class of ``field topologies of type $W_n$,'' cut out by a local
sentence (in the sense of \cite{prestel-ziegler}).  The canonical
topology on an unstable dp-finite field $K$ should be a definable
$W_n$-topology for some $n \le \dpr(K)$.  For $n = 1$, a
$W_1$-topology should be the same thing as a V-topology.  For $n = 2$,
a $W_2$-topology should either be a DV-topology in the sense of
\S\ref{sec:dtop}, or a topology generated by two independent
V-topologies.

For $n = 3$, there should be four types:
\begin{itemize}
\item A topology generated by three independent V-topologies.
\item A topology generated by a V-topology and an independent
  DV-topology.
\item A topology that is like a DV-topology, but with basic opens
  \begin{equation*}
    \{ x \in K : \val(x - a) \ge \gamma,~ \val(\delta_1 x - b) \ge
    \gamma,~ \val(\delta_2 x - c) \ge \gamma\}.
  \end{equation*}
  for \emph{two} derivations $\delta_1, \delta_2 : K \to K$.
\item A topology that is like a DV-topology, but involving second
  derivatives, with basic opens
  \begin{equation*}
    \{ x \in K : \val(x - a) \ge \gamma,~ \val(\delta x - b) \ge
    \gamma,~ \val(\delta^2 x - c) \ge \gamma\}.
  \end{equation*}
\end{itemize}
Now suppose $K$ is a field of dp-rank 3.  Using results from
\cite{prdf2}, it should be possible to prove that the squaring map
$f(x) = x^2$ is an open map from $K^\times$ to $K^\times$.  This
should exclude the first two cases.  In the latter two cases, it
should be possible to show that the $W_3$-topology has a unique
V-topology coarsening.  This would imply that $K$ admits a unique
definable V-topology.  Generalizations of these arguments should work
for $n > 3$.

From this point of view, the rank 2 case is too easy: the rank is so
small that there is no room for two independent topologies, unless
both are V-topologies.

\appendix

\section{Appendix: Resplendent lifting}
In the appendix, we assume that all rings are $\Qq$-algebras, all
fields extend $\Qq$, and all valued fields have residue characteristic
0.
\subsection{Extending derivations}
Say that an ordered abelian group $\Gamma$ is \emph{$\Zz$-less} if it satisfies
the following equivalent conditions:
\begin{itemize}
\item For every $a > 0$ in $\Gamma$, if $\Delta^+$ is the minimal
  convex subgroup containing $a$ and $\Delta^-$ is the maximal convex
  subgroup avoiding $a$, then $\Delta^+/\Delta^- \centernot \cong
  \Zz$.
\item For every $a > 0$ in $\Gamma$, there is $b \in \Gamma$ such that
  \begin{equation*}
    (1/3)a < b < (2/3)a,
  \end{equation*}
  i.e., $a < 3b < 2a$.
\item For every $a > 0$ in $\Gamma$ and every $p < q$ in $\Qq$, there
  is $b \in \Gamma$ such that
  \begin{equation*}
    pa < b < qa.
  \end{equation*}
\end{itemize}
For example, if $\Gamma$ is $p$-divisible for some prime $p$, then
$\Gamma$ is $\Zz$-less.

\begin{remark}
  Let $\Gamma' > \Gamma$ be an extension of ordered abelian groups.
  Suppose $\Gamma'/\Gamma$ is torsion (i.e., $\Gamma' \le \Gamma
  \otimes_\Zz \Qq$).  If $\Gamma$ is $\Zz$-less, then $\Gamma'$ is
  $\Zz$-less.
\end{remark}

\begin{lemma} \label{z-less-calc}
  Let $L/K$ be an algebraic extension of valued fields.  Suppose that
  the value group of $K$ is $\Zz$-less.  Let $a$ be a nonzero element
  of $\Oo_L$ with positive valuation.  Suppose that $a^n \in K$.  Then
  there are $b,c \in \Oo_L$ such that
  \begin{itemize}
  \item $a = bc^n$
  \item $bc^{n-1}$ is in $K$.
  \item $c^n$ is in $K$.
  \end{itemize}
\end{lemma}
\begin{proof}
  By $\Zz$-lessness, there is $\gamma \in \Gamma_K$ such that
  \begin{equation*}
    \frac{n-1}{n} \val(a) < \gamma < \val(a).
  \end{equation*}
  Let $e \in K$ have $\val(e) = \gamma$.  Let $b = e^na^{1-n}$ and $c
  = ae^{-1}$.  Then
  \begin{align*}
    \val(b) &= n \cdot \gamma - (n-1) \val(a) > 0 \\
    \val(c) &= \val(ae^{-1}) = \val(a) - \gamma > 0 \\
    bc^n &= e^n a^{1-n} a^n e^{-n} = a \\
    bc^{n-1} &= e^na^{1-n}a^{n-1}e^{1-n} = e \in K. \\
    c^n &= a^n e^{-n} \in K. \qedhere
  \end{align*}
\end{proof}

\begin{proposition}\label{alg-case-2}
  Let $(K,\Oo)$ be an algebraically closed field with a derivation
  $\partial : \Oo \to M$ for some $\Oo$-module $M$.  Suppose that
  $\partial$ vanishes on some subfield $F \subseteq K$, and $K$ is
  algebraic over $F$ (so that $K = F^{alg}$).  Suppose $\val(F)$ is
  $\Zz$-less.  Then $\partial$ vanishes on $K$.
\end{proposition}
The proof is nearly identical to the proof of
Proposition~\ref{alg-case}.
\begin{proof}
  Increasing $F$, we may assume $F$ is maximal among subfields on
  which $\partial$ vanishes.  Suppose for the sake of contradiction that
  $F \subsetneq K$.  Take minimal $n > 1$ such that $F$ has a finite
  extension of degree $n$.  If $P(x) \in F[x]$ has degree $\le n$,
  then one of the following happens:
  \begin{itemize}
  \item $P(x)$ factors into linear polynomials
  \item $P(x)$ is irreducible of degree $n$.
  \end{itemize}
  Take an arbitrary extension $L/F$ of degree $n$, and break into
  cases:
  \begin{itemize}
  \item If $\val(L)$ is strictly larger than $\val(F)$, take $\gamma
    \in \val(L) \setminus \val(F)$.  The inequality
    \begin{equation*}
      |\val(L)/\val(F)| \le [L : F]
    \end{equation*}
    implies that $m \gamma \in \val(F)$ for some $m \le n$.  Take $c
    \in F$ with $\val(c) = m \gamma$.  The polynomial $x^m - c$ has no
    roots in $F$, so $m = n$ and $x^n - c$ is irreducible.  Take $a
    \in K$ such that $a^n = c$.

    By maximality of $F$, there is some $x \in \Oo_{F(a)}$ such that
    $\partial x \ne 0$.  We can write
    \begin{equation*}
      x = y_0 + y_1 a + \cdots + y_{n-1} a^{n-1}
    \end{equation*}
    for some $y_i \in F$.  Note that $i \gamma \notin \val(F)$ for $i
    < n$, and so the non-zero terms $y_i a^i$ have pairwise distinct
    valuations.  Therefore
    \begin{equation*}
      0 \le \val(x) = \min_i(\val(y_i a^i)).
    \end{equation*}
    So every $y_i a^i$ is in $\Oo$, and
    \begin{equation*}
      \partial x = \sum_{i = 0}^{n-1} \partial(y_i a^i) = \sum_{i =
        1}^{n-1} \partial(y_i a^i).
    \end{equation*}
    For each $i > 0$, we have $\val(y_i a^i) \ne 0$.  Then
    Lemma~\ref{z-less-calc} gives $b, c \in \Oo_{F(a)}$ and $e \in \Oo_F$
    such that
    \begin{align*}
      e &= bc^{n-1} \\
      y_ia^i &= bc^n = ec \\
      c^n &\in F.
    \end{align*}
    Therefore
    \begin{equation*}
      \partial(y_ia^i) = \partial(ec) = e \partial(c) = b c^{n-1}
      \partial(c) = \frac{b}{n} \partial(c^n) = 0,
    \end{equation*}
    as $e, c^n \in \Oo_F$, and $1/n \in \Oo_K$ (by the assumption of
    residue characteristic 0).  So $\partial(x) = 0$, a contradiction.
  \item If $\res(L)$ is strictly larger than $\res(F)$, take $\alpha
    \in \res(L) \setminus \res(F)$.  The inequality
    \begin{equation*}
      [\res(L) : \res(F)] \le [L : F]
    \end{equation*}
    implies that $[\res(F)(\alpha) : \res(F)] \le n$.  Let
    \begin{equation*}
      x^m + \beta_{m-1} x^{m-1} + \cdots + \beta_1 x + \beta_0
    \end{equation*}
    be the monic irreducible polynomial of $\alpha$ over $\res(F)$.
    Because of residue characteristic 0, this polynomial is separable,
    and so
    \begin{equation}
      m \alpha^{m-1} + (m-1) \beta_{m-1} \alpha^{m-2} + \cdots + 2
      \beta_2 \alpha + \beta_1 \ne 0. \label{simple-1-clone}
    \end{equation}
    Take $b_i \in F$ with $\res b_i = \beta_i$, and let $P(x)$ be the
    polynomial
    \begin{equation*}
      x^m + b_{m-1} x^{m-1} + \cdots + b_1 x + b_0 \in F[x].
    \end{equation*}
    Then $P(x)$ is irreducible, and so $m = [\res(F)(\alpha) :
      \res(F)] = n$.  Let $a \in K$ be the root of $P(x)$ with
    $\res(a) = \alpha$.  Then
    \begin{equation*}
      a^n + b_{n-1} a^{n-1} + \cdots + b_1 a + b_0 = 0.
    \end{equation*}
    Applying the derivation $\partial : \Oo \to M$, which vanishes on
    the $b_i$, we obtain
    \begin{equation*}
      (n a^{n-1} + (n-1) b_{n-1} a^{n-2} + \cdots + 2 b_2 a + b_1 ) \partial a = 0.
    \end{equation*}
    The expression inside the parentheses has nonzero residue, by
    (\ref{simple-1-clone}), and so it is an element of $\Oo^\times$.
    Therefore $\partial a = 0$.

    Now if
    \begin{equation*}
      x = y_0 + y_1 a + \cdots + y_{n-1} a^{n-1}
    \end{equation*}
    is any element of $F(a)$, then $\val(x) = \min_i \val(y_i)$.  To
    see this, one reduces to the case where $\min_i(\val(y_i)) = 0$;
    then
    \begin{equation*}
      \res(y_0) + \res(y_1)\alpha + \cdots + \res(y_{n-1})\alpha^{n-1}
      \ne 0,
    \end{equation*}
    by linear independence of $\{1,\alpha,\ldots,\alpha^{n-1}\}$ over
    $\res(K)$.

    Consequently, $\Oo_{F(a)} = \Oo_F[a]$.  As $\partial$
    vanishes on $\Oo_F$ and $a$, it vanishes on $\Oo_{F(a)}$,
    contradicting the maximality of $F$.
  \item Lastly, suppose that $L/F$ is an immediate extension.  By
    maximality of $F$, there is $a \in \Oo_L$ with $\partial(a) \ne 0$.  Let
    $\mathcal{C}$ be the collection of balls containing $a$, with center and radius from $F$.  Let $I$ be the intersection $\bigcap \mathcal{C}$.  As usual, $I
    \cap F = \emptyset$.
    
    Let $P(x)$ be the minimal polynomial of $a$ over $F$.  Then $P(x)$
    has degree $n$.  Let $a_1, \ldots, a_n$ be the roots of $P(x)$,
    with $a_1 = a$.  Note that $P'(x)$ has degree $n - 1$, and
    therefore splits over $F$.  So no root of $P'(x)$ is in $I$.  By
    Rolle's Theorem (Lemma~\ref{rolle}), $a$ is the unique root of
    $P(x)$ in $I$.

    Therefore $I$ has empty intersection with the finite set
    $\{0,a_2,\ldots,a_n\}$.  We can find $b \in F$ such that
    \begin{equation*}
      \val(a - b) > \max(\val(0 - b),\val(a_2 - b),\val(a_3 -
      b),\ldots,\val(a_n - b)).
    \end{equation*}
    Take $c \in F$ with $\val(a - b) = \val(c)$, and let $e_i = (a_i -
    b)/c$.  Then $\val(e_1) = 0$, and $\val(e_i) < 0$ for $i > 1$.
    The $e_i$ are the roots of the irreducible polynomial
    \begin{equation*}
      Q(x) = P(cx + b) = s_n x^n + s_{n-1} x^{n-1} + \cdots + s_1 x +
      s_0 \in F[x].
    \end{equation*}
    By Newton polygons, $\val(s_0) = \val(s_1) < \val(s_i)$ for $i >
    1$.  Then we can apply $\partial$ to the equation
    \begin{equation*}
      (s_n/s_1) e_1^n + \cdots + (s_2/s_1) e_1^2 + e_1 + (s_0/s_1) =
      0,
    \end{equation*}
    and obtain
    \begin{equation*}
      (n (s_n/s_1) e_1^{n-1} + \cdots + 2 (s_2 / s_1) e_1 + 1) \partial e_1 = 0,
    \end{equation*}
    because the coefficients $s_n/s_1$ lie in $F$, where $\partial$ vanishes.  But
    the expression in parentheses has valuation 0, because $e_1 \in
    \Oo$ and $s_i/s_1 \in \mm$ for $2 \le i \le n$.  Therefore
    $\partial e_1 = 0$.

    Meanwhile, $\val(a-b) > \val(b)$ implies that
    \begin{equation*}
      0 \le \val(a) = \min(\val(a-b),\val(b)) = \min(\val(c),\val(b)),
    \end{equation*}
    and so $b, c \in \Oo_F$.  But $a = b + e_1 c$, and so $\partial(a)
    = c \partial e_1 = 0$, contradicting the choice of $a$. \qedhere
  \end{itemize}
\end{proof}

\subsection{Review of K\"ahler differentials}
If $A \to B$ is a morphism of (commutative unital) rings, then
$\Omega_{B/A}$ denotes the module of K\"ahler differentials.  This is
the $B$-module generated by terms $db$ for $b \in B$, subject to the
relations
\begin{align*}
  d(b_1 + b_2) &= db_1 + db_2 \\
  d(b_1b_2) &= b_1db_2 + b_2db_1 \\
  da &= 0 \qquad \textrm{if } a \in A.
\end{align*}
If $M$ is a $B$-module, there is an isomorphism
\begin{equation*}
  \Hom_B(\Omega_{B/A},M) \cong \Der_A(B,M)
\end{equation*}
natural in $M$, where $\Der_A(B,M)$ denotes the set of $A$-linear
derivations $B \to M$.

The following facts about K\"ahler differentials are well-known:
\begin{fact}\label{k-right-exact}
  If $A \to B \to C$ is a morphism of rings, then
  \begin{equation*}
    \Omega_{B/A} \otimes_B C \to \Omega_{C/A} \to \Omega_{C/B} \to 0
  \end{equation*}
  is exact.
\end{fact}
\begin{fact}\label{k-localization}
  If $A \to B$ is a morphism of rings and $S \subseteq B$ is a
  multiplicative subset, then
  \begin{equation*}
    S^{-1} \Omega_{B/A} \cong \Omega_{S^{-1}B/A}.
  \end{equation*}
\end{fact}
\begin{fact}\label{k-fields}
  If $L/K$ is an extension of (characteristic 0) fields and
  $\{t_i\}_{i \in I}$ is a transcendence basis (possibly infinite),
  then $\{dt_i\}_{i \in I}$ is an $L$-basis of $\Omega_{L/K}$.
\end{fact}

\begin{remark}\label{valuation-flatness}
  If $\Oo$ is a valuation ring and $M$ is an $\Oo$-module, the
  following are equivalent:
  \begin{enumerate}
  \item $M$ is flat.
  \item $M$ is torsionless.
  \item Every finitely-generated submodule of $M$ is free.
  \item $M$ is a direct limit of free modules.
  \item The natural map $M \to M \otimes_{\Oo} K$ is an injection.
  \end{enumerate}
\end{remark}

We will use two flatness results from (\cite{GabberRamero},
Corollary~6.5.21 and Theorem~6.5.15).
\begin{fact}\label{gr-flat}
  If $\Oo$ is a valuation ring with residue characteristic 0, then
  $\Omega_{\Oo/\Qq}$ is flat as an $\Oo$-module.
\end{fact}
\begin{fact}
  \label{acf-base}
  Let $\Oo'/\Oo$ be an extension of valuation rings.  Suppose
  $\Frac(\Oo) \models \ACF$.  Then $\Omega_{\Oo'/\Oo}$ is flat as an
  $\Oo$-module.
\end{fact}

\subsection{Flatness and extensions}
\begin{lemma}\label{flat-2-3}
  Let $\Oo$ be a valuation ring and $0 \to A \to B \to C \to 0$ be a
  short exact sequence of $\Oo$-modules.
  \begin{itemize}
  \item If $B$ is flat, then $A$ is flat.
  \item If $A$ and $C$ are flat, then $B$ is flat.
  \end{itemize}
\end{lemma}
\begin{proof}
  By Remark~\ref{valuation-flatness}, an $\Oo$-module is flat if and
  only if it is torsionless.
  
  For the first point: submodules of torsionless modules are
  torsionless.

  For the second point, suppose that $A$ and $C$ are torsionless.  For
  any nonzero $r \in \Oo$, there is a diagram
  \begin{equation*}
    \xymatrix{
      0 \ar[r] & A \ar[r] \ar[d] & B \ar[r] \ar[d] & C \ar[r] \ar[d] & 0 \\
      0 \ar[r] & A \ar[r] & B \ar[r] & C \ar[r] & 0 }
  \end{equation*}
  where the rows are exact and the vertical maps are multiplication by
  $r$.  Because $A$ and $C$ are torsionless, the outer vertical maps
  are injective.  By the snake lemma, the inner vertical map is
  injective.  As $r$ is arbitrary, $C$ is torsionless.
\end{proof}
\begin{lemma}\label{faithful-flat}
  Let $\Oo'/\Oo$ be an extension of valuation rings.  Let $M$ be an
  $\Oo$-module.  Then $M$ is flat (as an $\Oo$-module) if and only if
  $M \otimes_{\Oo} \Oo'$ is flat (as an $\Oo'$-module).
\end{lemma}
\begin{proof}
  If $M$ is flat, then $M$ is a direct limit of free $\Oo$-modules,
  and so $M \otimes_\Oo \Oo'$ is a direct limit of free
  $\Oo'$-modules.  Conversely, suppose that $M$ is not flat.  Then
  there is an injection
  \begin{equation*}
    \Oo/I \hookrightarrow M
  \end{equation*}
  for some non-zero proper ideal $I$ in $\Oo$.  As $\Oo'$ is
  torsionless over $\Oo$, it is flat as an $\Oo$-module.  Therefore
  the functor $- \otimes_\Oo \Oo'$ is exact, and the map
  \begin{equation*}
    (\Oo/I) \otimes_\Oo \Oo' \hookrightarrow M \otimes_\Oo \Oo'
  \end{equation*}
  is injective.  But $(\Oo/I) \otimes_\Oo \Oo' \cong \Oo'/I\Oo'$.  The
  ideal $I\Oo'$ is non-trivial, because it contains the non-trivial
  elements of $I$.  And $I\Oo'$ is a proper ideal, because it is
  generated by elements of positive valuation.  Therefore $\Oo'/I\Oo'$
  is not torsionless, and neither is the larger module $M \otimes_\Oo
  \Oo'$.
\end{proof}

\begin{lemma}\label{k-case}
  Let $K_1 \subseteq K_2 \subseteq K_3$ be a chain of three fields (of
  characteristic 0).  Then the map
  \begin{equation*}
    \Omega_{K_2/K_1} \otimes_{K_2} K_3 \to \Omega_{K_3/K_1}
  \end{equation*}
  is injective.
\end{lemma}
\begin{proof}
  Let $B$ be a transcendence basis of $K_2/K_1$, and $B'$ be a
  transcendence basis of $K_3/K_1$ extending $B$.  By
  Fact~\ref{k-fields}, the set $\{dt : t \in B\}$ is a $K_3$-linear
  basis of $\Omega_{K_2/K_1} \otimes_{K_2} K_3$, and the set $\{dt : t
  \in B'\}$ is a $K_3$-linear basis of $\Omega_{K_3/K_1}$.  The map in
  question is induced by the inclusion $B \hookrightarrow B'$, and is
  therefore injective.
\end{proof}
\begin{lemma}\label{o-case-over-q}
  Let $\Oo'/\Oo$ be an extension of valuation rings (with residue
  characteristic 0).  Then the map $\Omega_{\Oo/\Qq} \otimes_{\Oo} \Oo'
  \to \Omega_{\Oo'/\Qq}$ is injective.
\end{lemma}
\begin{proof}
  This follows from the commuting diagram
  \begin{equation*}
    \xymatrix{ \Omega_{\Oo/\Qq} \otimes_{\Oo} \Oo' \ar@{^(->}[r]
      \ar[d] & (\Omega_{\Oo/\Qq} \otimes_{\Oo} \Oo') \otimes_{\Oo'} K' \ar[d] \ar@{=}[r] & (\Omega_{\Oo/\Qq} \otimes_{\Oo} K) \otimes_{K} K'
       \ar[d]
      \ar@{=}[r] & \Omega_{K/\Qq} \otimes_{K} K' \ar@{^(->}[d]
      \\ \Omega_{\Oo'/\Qq} \ar@{^(->}[r] & \Omega_{\Oo'/\Qq}
      \otimes_{\Oo'} K' \ar@{=}[r] & \Omega_{\Oo'/\Qq}
      \otimes_{\Oo'} K' \ar@{=}[r] & \Omega_{K'/\Qq} },
  \end{equation*}
  where the left horizontal arrows are injective by flatness
  (Fact~\ref{gr-flat}), the right horizontal arrows are isomorphisms
  by Fact~\ref{k-localization}, and the rightmost vertical map is
  injective by Lemma~\ref{k-case}.
\end{proof}
\begin{lemma}\label{o-case}
  Let $\Oo_1 \subseteq \Oo_2 \subseteq \Oo_2$ be a chain of two
  valuation ring extensions\footnote{Meaning that the inclusions are
    local homomorphisms.}.  Then the map $\Omega_{\Oo_2/\Oo_1} \otimes_{\Oo_2} \Oo_3
  \to \Omega_{\Oo_3/\Oo_1}$ is injective.
\end{lemma}
\begin{proof}
  There is a commutative diagram
  \begin{equation}
    \xymatrix{
      0 \ar[r] & \Omega_{\Oo_1/\Qq} \otimes_{\Oo_1} \Oo_3 \ar@{=}[d]
      \ar[r] & \Omega_{\Oo_2/\Qq} \otimes_{\Oo_2} \Oo_3 \ar[r] \ar@{^(->}[d]
      & \Omega_{\Oo_2/\Oo_1} \otimes_{\Oo_2} \Oo_3 \ar[d] \ar[r] & 0
      \\
      0 \ar[r] & \Omega_{\Oo_1/\Qq} \otimes_{\Oo_1} \Oo_3 \ar[r] & \Omega_{\Oo_3/\Qq} \ar[r]
      & \Omega_{\Oo_3/\Oo_1} \ar[r] & 0
    } \label{confound}
  \end{equation}
  The bottom row is right exact by Fact~\ref{k-right-exact}, and left
  exact by Lemma~\ref{o-case-over-q}.  The same argument shows that
  \begin{equation*}
    0 \to \Omega_{\Oo_1/\Qq} \otimes_{\Oo_1} \Oo_2 \to
    \Omega_{\Oo_2/\Qq} \to \Omega_{\Oo_2/\Oo_1} \to 0
  \end{equation*}
  is an exact sequence.  Applying the exact functor $- \otimes_{\Oo_2}
  \Oo_3$ yields the exactness of the top row of (\ref{confound}).  The
  middle vertical map of (\ref{confound}) is injective by
  Lemma~\ref{o-case-over-q}.  The snake lemma then implies that the
  right vertical map is injective.
\end{proof}
\begin{definition}
  Let $\Oo'/\Oo$ be an extension of valuation rings.  Then $\Oo'/\Oo$
  is \emph{pseudosmooth} if $\Omega_{\Oo'/\Oo}$ is flat (as an
  $\Oo'$-module).
\end{definition}
\begin{proposition}\label{pseudo-2-3}
  Let $\Oo_1 \subseteq \Oo_2 \subseteq \Oo_3$ be a chain of two
  valuation ring extensions.
  \begin{enumerate}
  \item\label{p23a} If $\Oo_2/\Oo_1$ and $\Oo_3/\Oo_2$ are
    pseudosmooth, then $\Oo_3/\Oo_1$ is pseudosmooth.
  \item\label{p23b} If $\Oo_3/\Oo_1$ is pseudosmooth, then
    $\Oo_2/\Oo_1$ is pseudosmooth.
  \end{enumerate}
\end{proposition}
\begin{proof}
  The sequence
  \begin{equation*}
    0 \to \Omega_{\Oo_2/\Oo_1} \otimes_{\Oo_2} \Oo_3 \to
    \Omega_{\Oo_3/\Oo_1} \to \Omega_{\Oo_3/\Oo_2} \to 0
  \end{equation*}
  is right exact by Fact~\ref{k-right-exact}, and left exact by
  Lemma~\ref{o-case}.  Then
  \begin{align*}
    \Oo_2/\Oo_1 \text{ is pseudosmooth} & \iff \Omega_{\Oo_2/\Oo_1}
    \text{ is flat} \iff \Omega_{\Oo_2/\Oo_1} \otimes_{\Oo_2} \Oo_3
    \text{ is flat} \\
    \Oo_3/\Oo_1 \text{ is pseudosmooth} & \iff \Omega_{\Oo_3/\Oo_1} \text{ is flat} \\
    \Oo_3/\Oo_2 \text{ is pseudosmooth} & \iff \Omega_{\Oo_3/\Oo_2} \text{ is flat,}    
  \end{align*}
  using Lemma~\ref{faithful-flat} in the first line.  The desired
  statements follow from Lemma~\ref{flat-2-3}.
\end{proof}

\begin{proposition}\label{flat-extension}
  Let $\Oo'/\Oo$ be an extension of valued fields of residue
  characteristic 0.  Suppose the value group of $\Oo$ is $\Zz$-less.
  Then $\Omega_{\Oo'/\Oo}$ is flat as an $\Oo'$-module.
\end{proposition}
\begin{proof}
  We must show that $\Oo'/\Oo$ is pseudosmooth.  By
  Proposition~\ref{pseudo-2-3}.\ref{p23b}, we may replace $\Oo'$ with
  a larger valued field.  Since valuations can be extended along any
  field extension, we may assume that $\Frac(\Oo')$ contains the
  algebraic closure of $\Frac(\Oo)$.  Let $\Oo''$ be the induced
  valuation ring on $\Frac(\Oo)^{alg}$.  Then $\Oo \subseteq \Oo''
  \subseteq \Oo'$.  Now $\Oo'/\Oo''$ is pseudosmooth by
  Fact~\ref{acf-base}, as $\Frac(\Oo'')$ is algebraically closed.  By
  Proposition~\ref{pseudo-2-3}.\ref{p23a}, it remains to show that
  $\Oo''/\Oo$ is pseudosmooth.  In fact, $\Omega_{\Oo''/\Oo}$ vanishes,
  by Proposition~\ref{alg-case-2}.
\end{proof}

\subsection{Resplendent lifting in the $\Zz$-less case}

\begin{lemma} \label{ortho}
  Let $R$ be a ring, and 
  \begin{equation*}
    \xymatrix{
      A \ar@{_(->}[d]_{f'} \ar[r]^{g'} & M \ar@{->>}[d]^{f} \\ A \oplus B \ar[r]_g & N
    }
  \end{equation*}
  be a diagram of $R$-modules, with $f$ surjective, and $f'$ the inclusion of the first factor.  If $B$ is
  free, then there is a diagonal map $h : A \oplus B \to M$ making the
  diagram commute:
  \begin{equation*}
    \xymatrix{ A \ar@{_(->}[d]_{f'} \ar[r]^{g'} & M \ar@{->>}[d]^{f}
      \\ A \oplus B \ar[ur]_h \ar[r]_g & N.  }
  \end{equation*}
\end{lemma}
The proof is well-known but included for completeness.
\begin{proof}
  Consider the morphism $B \to N$ given by $x \mapsto g(0,x)$.
  Because $B$ is free, there is $h_1 : B \to M$ lifting this, so that
  \begin{equation*}
    g(0,x) = f(h_1(x))
  \end{equation*}
  for $x \in B$.  Define $h : A \oplus B \to M$ by the formula $h(x,y)
  = g'(x) + h_1(y)$.  Then
  \begin{align*}
    h(f'(x)) &= h(x,0) = g'(x) \\
    f(h(x,y)) &= f(g'(x)) + f(h_1(y)) = g(f'(x)) + g(0,y) = g(x,0) + g(0,y) = g(x,y). \qedhere
  \end{align*}
\end{proof}

\begin{lemma}\label{important-helper}
  Let $R$ be a valuation ring.  Let
  \begin{equation*}
    \xymatrix{
      A \ar@{_(->}[d]_{f'} \ar[r]^{g'} & M \ar@{->>}[d]^{f} \\ B \ar[r]_g & N
      }
  \end{equation*}
  be a diagram of $R$-modules, with $f' : A \hookrightarrow B$
  injective and $f : M \twoheadrightarrow N$ surjective.  Suppose the
  following hold:
  \begin{itemize}
  \item $\coker(f')$ is flat.
  \item Let $(M,N)$ be the two-sorted structure with the $R$-module
    structure on $M$ and $N$, and the surjection $f : M
    \twoheadrightarrow N$.  Then $(M,N)$ is
    $(|A|+|B|+|R|)^+$-saturated.
  \end{itemize}
  Then there is a morphism $h : B \to M$ making the diagram commute
  \begin{equation*}
    \xymatrix{
      A \ar@{_(->}[d]_{f'} \ar[r]^{g'} & M \ar@{->>}[d]^{f} \\ B \ar[ur]_h \ar[r]_g & N.
      }
  \end{equation*}
\end{lemma}
\begin{proof}
  Without loss of generality, $f' : A \hookrightarrow B$ is an
  inclusion.  Then $B/A$ is flat.
  
  Let $\vec{x} = \langle x_b \rangle_{b \in B}$ be a tuple of
  variables indexed by $B$.  For every submodule $C \subseteq B$
  containing $A$, let $\Sigma_C(\vec{x})$ be the $\ast$-type in
  $(M,N)$ asserting the following:
  \begin{itemize}
  \item If $b \in C$, then $x_b \in M$ and $f(x_b) = g(b)$.
  \item If $b \in A$, then $x_b = g'(x_b)$.
  \item If $r \in R$ and $b \in C$, then $r x_b = x_{rb}$.
  \item If $b, b' \in C$, then $x_{b+b'} = x_b + x_{b'}$.
  \end{itemize}
  Then $\Sigma_C(\vec{x})$ is realized in $(M,N)$ if and only if there
  is a morphism $h_C : C \to M$ such that the diagram commutes
  \begin{equation}
    \xymatrix{
      A \ar[r]^{g'} \ar[d]_{\subseteq} & M \ar[d]^f \\
      C \ar@{-->}[ur]_{h_C} \ar[r]_{g|_C} & N} \label{c-target}
  \end{equation}
  It suffices to realize $\Sigma_B(\vec{x})$.  This type is a directed
  union of the types
  \begin{equation*}
    \{ \Sigma_C(\vec{x}) : C/A \text{ is finitely generated}\}.
  \end{equation*}
  By saturation, it suffices to realize the types in this family.
  Suppose $C/A$ is finitely generated.  Then $C/A$ injects into $B/A$,
  so $C/A$ is free by Remark~\ref{valuation-flatness}.  The sequence
  \begin{equation*}
    0 \to A \to C \to C/A \to 0
  \end{equation*}
  therefore splits.  By Lemma~\ref{ortho}, there is a dashed arrow
  making (\ref{c-target}) commute.
\end{proof}
Note that $\Zz$-lessness is a conjunction of first-order axioms, so it
is preserved in elementary equivalence of ordered abelian groups.
\begin{theorem}\label{resplendent}
  Let $(K,\Oo)$ be a valued field with residue characteristic 0 and
  $\Zz$-less value group.  Let $f : M \twoheadrightarrow N$ be a
  surjective morphism of $\Oo$-modules.  Let $\partial : \Oo \to N$
  be a derivation.  Consider the three-sorted structure $(\Oo,M,N)$
  with the ring structure on $\Oo$, the module structures on $M, N$,
  the epimorphism $f$, and the derivation $\partial$.
  \begin{itemize}
  \item If the structure $(\Oo,M,N)$ is sufficiently saturated and
    resplendent, then there is a derivation $\delta : \Oo \to M$
    making the diagram commute:
    \begin{equation*}
      \xymatrix{ & M \ar@{->>}[d]^f \\
        \Oo \ar[ur]^{\delta} \ar[r]_{\partial} & N }
    \end{equation*}
  \item In general, such a lifting exists after passing to an elementary
    extension.
  \end{itemize}
\end{theorem}
\begin{proof}
  The two statements are clearly equivalent, by definition of
  resplendence and existence of resplendent elementary extensions.  We
  prove the second statement.  Consider an elementary chain
  \begin{equation*}
    (\Oo,M,N) = (\Oo_0,M_0,N_0) \preceq (\Oo_1,M_1,N_1) \preceq
    (\Oo_2,M_2,N_2) \preceq \cdots
  \end{equation*}
  where each structure is saturated over the previous structure.  Let
  $f_i, \partial_i$ be the structure maps in $(\Oo_i,M_i,N_i)$.  We
  will recursively build a sequence of derivations $\delta_i
  : \Oo_i \to M_{i+1}$ such that
  \begin{align*}
    f_{i+1}(\delta_i(x)) &= \partial_i(x) \\
    \delta_{i+1}(x) &= \delta_i(x)
  \end{align*}
  for $i \ge 0$ and $x \in \Oo_i$.  If this can be done successfully,
  then the union of the $\delta_i$'s is the desired lifting
  of $\partial$ on the structure $\bigcup_i (\Oo_i,M_i,N_i)$, an
  elementary extension of $(\Oo,M,N)$, and we are done.

  At step $i = 0$, we must find an $\Oo_0$-linear map
  $\Omega_{\Oo_0/\Qq} \to M_1$ making the diagram commute
  \begin{equation}
    \xymatrix{ & & M_1 \ar@{-->}[d]^{f_1} \\ \Omega_{\Oo_0/\Qq}
      \ar[urr] \ar[r]_{\partial_0} & N_0 \ar[r]_{\subseteq} & N_1.} \label{lift-1}
  \end{equation}
  At step $i > 1$, we must find an $\Oo_i$-linear map
  $\Omega_{\Oo_i/\Qq} \to M_{i+1}$ making the diagram commute
  \begin{equation}
    \xymatrix{
    (\Oo_i \otimes_{\Oo_{i-1}} \Omega_{\Oo_{i-1}/\Qq}) \ar[d]
    \ar[r]^{\delta_{i-1}} & M_i \ar[r]^{\subseteq} & M_{i+1} \ar[d]^{f_{i+1}}
    \\
    \Omega_{\Oo_i/\Qq} \ar@{-->}[urr] \ar[r]_{\partial_i} & N_i \ar[r]_{\subseteq} & N_{i+1}.} \label{lift-2}
  \end{equation}
  The dashed map exists in both cases by Lemma~\ref{important-helper}.  For (\ref{lift-1}), Fact~\ref{gr-flat} shows that $\Omega_{\Oo_0/\Qq}$ is flat.  For (\ref{lift-2}), the map
  \[ \Oo_i \otimes_{\Oo_{i-1}} \Omega_{\Oo_{i-1}/\Qq} \to \Omega_{\Oo_i/\Qq}\]
  is injective by Lemma~\ref{o-case-over-q}, the cokernel is $\Omega_{\Oo_i/\Oo_{i-1}}$ by Fact~\ref{k-right-exact}, and $\Omega_{\Oo_i/\Oo_{i-1}}$ is flat by Proposition~\ref{flat-extension}.
\end{proof}
\begin{corollary}\label{metaval-lifting}
  Let $(K,\partial)$ be a sufficiently resplendent normalized diffeovalued
  field.  If the value group of $K$ is $\Zz$-less, or $p$-divisible
  for some $p$, then $(K,\partial)$ admits a
  lifting.
\end{corollary}
\begin{proof}
  If the value group is $p$-divisible, then it is $\Zz$-less.  Theorem~\ref{resplendent} allows us to lift the given derivation $\partial : \Oo \to K/\mm$ to a derivation $\delta_0 : \Oo \to K$.  This corresponds to an $\Oo$-linear map $\Omega_{\Oo/\Qq} \to K$, which in turn yields a $K$-linear map
  \[ \Omega_{K/\Qq} \cong \Omega_{\Oo/\Qq} \otimes_{\Oo} K \to K\]
  by Fact~\ref{k-localization}.  Thus $\delta_0 : \Oo \to K$ extends to a derivation $\delta : K \to K$.
\end{proof}

\subsection{An unliftable example}
The assumption that $\Gamma$ is $\Zz$-less is necessary in
Corollary~\ref{metaval-lifting}.
\begin{proposition}\label{z-probs}
  There is a valued field $(K,\Oo,\mm)$ of residue
  characteristic 0, and a derivation $\partial : \Oo \to K/\mm$, such
  that in any elementary extension $(K^*,\Oo^*,\mm^*,\partial) \succeq
  (K,\Oo,\mm,\partial)$, there is no derivation $\delta :
  \Oo^* \to K^*$ making the diagram commute:
  \begin{equation*}
    \xymatrix{ & K^* \ar[d] \\
      \Oo^* \ar[ur]^{\delta} \ar[r]_{\partial} & K^*/\mm^*}
  \end{equation*}
  One can even take $(K,\Oo)$ to be dp-minimal as a pure valued field.
\end{proposition}
\begin{proof}
  Let $\Zz + \Zz \omega$ be the free abelian group on two generators
  $1, \omega$, ordered so that $\omega > n \cdot 1$ for all $n \in
  \Zz$.  In other words, $\Zz + \Zz \omega$ is the lexicographic
  product $\Zz \times \Zz$, with generators $\omega := (1,0)$ and $1
  := (0,1)$.

  Let $L$ be the Hahn field $\Qq^{alg}((t^{\Zz + \Zz \omega}))$, and let $K$
  be the relative algebraic closure of $\Qq(t,t^\omega)$ in $L$.  Let
  $\Oo_L, \Oo_K$ denote the valuation rings on $L$ and $K$, and
  $\mm_L, \mm_K$ denote their maximal ideals.  The valued field
  $(K,\Oo_K)$ is henselian with residue characteristic 0,
  algebraically closed residue field, and dp-minimal value group, so
  $(K,\Oo_K)$ is a dp-minimal valued field.

  Let $\val'$ be the coarsening of $\val$ by the convex subgroup $\Zz
  \le \Zz + \Zz \omega$, and let $\pp \lhd \Oo_L$ be the associated
  maximal ideal.  If $\val(x) = i + j\omega$, then $\val'(x) = j$.
  Moreover, for any $x$,
  \begin{equation*}
    x \in \pp \iff \val'(x) > 0 \iff \val(x) > \Zz.
  \end{equation*}
  Note that for $x \in L$,
  \begin{equation}
    \val'(x) \ge 0 \implies x \in \mm_L + \Qq^{alg}[t^{-1}], \label{slice}
  \end{equation}
  because one can split $x = \sum_{i,j} a_{i,j} t^{i+j\omega}$ as
  \begin{equation*}
    x = \sum_{i + j\omega \le 0} a_{i,j} t^{i+j\omega} + \sum_{i + j\omega > 0} a_{i,j} t^{i+j\omega}.
  \end{equation*}
  The assumption $\val'(x) \ge 0$ ensures that the first sum only
  involves $i+j\omega$ with $j = 0$.  The support is well-ordered, so
  the first sum is finite, and belongs to $\Qq^{alg}[t^{-1}]$.  The other
  sum is in $\mm_L$, proving (\ref{slice}).

  Choose
  \begin{equation*}
    u = 1 + a_1t + a_2t^2 + \cdots \in 1 + t\Qq^{alg}[[t]] \subseteq
    \Qq^{alg}((t)) = \Qq^{alg}((t^\Zz)) \subseteq \Qq^{alg}((t^{\Zz + \Zz\omega})) = L
  \end{equation*}
  such that $u \not\equiv v \pmod{\pp}$ for all $v \in K$.  Such a $u$ exists because $1 + t\Qq^{alg}[[t]]$ is
  uncountable, $K$ is countable, and the elements of $\Qq^{alg}((t))$ are
  pairwise distinct modulo $\pp$.  (The valuation $\val'$ restricts to
  the trivial valuation on $\Qq^{alg}((t))$.)

  Consider the derivation
  \begin{align*}
    \partial_0 : L & \to L \\
    \sum_{i,j} a_{i,j}t^{i + j\omega} &\mapsto \sum_{i,j} a_{i,j}jt^{i + (j-1)\omega}.
  \end{align*}
  Note that for $x \in L$,
  \begin{equation}
    x \in \Oo_L \implies \val'(\partial_0 x) \ge 0 \label{lowbound}
  \end{equation}
  Let $\partial_1 : L \to L$ be the derivation $\partial_1 x := u
  \partial_0 x$.  Let $\partial$ be the composition
  \begin{equation*}
    \Oo_K \hookrightarrow L \stackrel{\partial_1}{\to} L \to L/\mm_L.
  \end{equation*}
  We claim that $\partial$ factors through the inclusion $K/\mm_K
  \hookrightarrow L/\mm_L$.  Indeed,
  \begin{equation*}
    x \in \Oo_K \implies x \in \Oo_L \implies \val'(\partial_0 x) \ge
    0 \iff \val'(u \partial_0 x) \ge 0,
  \end{equation*}
  because $\val(u) = 0$.  By (\ref{slice}),
  \begin{align*}
    \val'(u \partial_0 x) \ge 0 &\implies u \partial_0 x \in \mm_L +
    \Qq^{alg}[t^{-1}] \\ & \implies \partial x \in (\Qq^{alg}[t^{-1}]+\mm_L)/\mm_L
    \subseteq (K+\mm_L)/\mm_L \cong K/(K \cap \mm_K) = K/\mm_K.
  \end{align*}
  So $\partial$ is a well-defined derivation from $\Oo_K$ to
  $K/\mm_K$.

  We claim that the following first-order statement $\sigma$ holds in
  the structure $(K,\Oo_K,\partial)$:
  \begin{quote}
    There is an $a \in \Oo_K$ such that for every $a' \in K$, there
    are $b, c \in \Oo_K$ such that $a = bc$ and for every $b', c' \in
    K$, the following identities do \emph{not} all hold:
  \end{quote}
  \begin{align*}
    b' &\equiv \partial b \pmod{\mm_K} \\
    c' &\equiv \partial c \pmod{\mm_K} \\
    a' &= bc' + cb'.
  \end{align*}
  Before proving this, note that this would complete the proof:
  \begin{itemize}
  \item The statement $\sigma$ is first-order, so it remains true in
    any elementary extension of $(K,\Oo,\partial)$.
  \item If the lifting $\delta : \Oo \to K$ exists, the
    statement $\sigma$ is false, because an adversary can choose
    \begin{align*}
      a' &= \delta a \\
      b' &= \delta b \\
      c' &= \delta c.
    \end{align*}
    and the three equations would hold.
  \end{itemize}
  We now prove $\sigma$.  For our opening move, we choose $a =
  t^\omega \in \Oo_K$.  The opponent chooses $a' \in K$.  Note
  \begin{align*}
    \partial_0 a &= \partial_0 t^\omega = 1 \\
    \partial_1 a &= u \partial_0 a = u.
  \end{align*}
  By choice of $u$, we know that $a' - u \notin \pp$, so $\val(a' - u)
  < n$ for some $n \in \Zz$.

  For our next move, we take $b = t^n$ and $c = t^{\omega - n}$.  The
  condition $a = bc$ holds, so we haven't lost the game yet.  The
  opponent chooses $b', c' \in K$.  Suppose that all three identities
  hold:
  \begin{align*}
    b' &\equiv \partial b \pmod{\mm_K} \\
    c' &\equiv \partial c \pmod{\mm_K} \\
    a' &= bc' + cb'.
  \end{align*}
  Then
  \begin{align*}
    \partial_1 b & \equiv b' \pmod{\mm_L} \\
    \partial_1 b & \equiv c' \pmod{\mm_L}.
  \end{align*}
  Now $b, c$ are divisible by $t^n$, so
  \begin{align*}
    c \partial_1 b & \equiv cb' \pmod{t^n\mm_L} \\
    b \partial_1 c & \equiv bc' \pmod{t^n\mm_L}
  \end{align*}
  Adding the two equations, and using the identities
  \begin{align*}
    a' &= bc' + cb' \\
    \partial_1(bc) &= b\partial_1c + c \partial_1 b,
  \end{align*}
  we obtain
  \begin{equation*}
    \partial_1(bc) \equiv a' \pmod{t^n\mm_L}
  \end{equation*}
  On the other hand,
  \begin{equation*}
    \partial_1(bc) = \partial_1(a) = u,
  \end{equation*}
  so $u \equiv a' \pmod{t^n\mm_L}$.  Then $\val(u - a') > n$,
  contradicting the choice of $n$.  So it is impossible for all three
  identities to hold, and we have won the game.  This proves the
  sentence $\sigma$ and completes the proof.
\end{proof}

\begin{acknowledgment}
The author would like to thank Meng Chen, Hagen Knaf, and Franz-Viktor
Kuhlmann for some helpful information on K\"ahler differentials.
{\tiny This material is based upon work supported by the National Science
Foundation under Award No. DMS-1803120.  Any opinions, findings, and
conclusions or recommendations expressed in this material are those of
the author and do not necessarily reflect the views of the National
Science Foundation.}
\end{acknowledgment}

\bibliographystyle{plain} \bibliography{mybib}{}

\end{document}